\documentclass{article}
\usepackage[utf8]{inputenc}
\usepackage{bm,amsfonts,amsthm,amsmath,amssymb}
\usepackage{graphicx,color}
\usepackage{mathrsfs}
\usepackage{indentfirst}
\usepackage{bbm}
\usepackage{appendix}

\DeclareMathAlphabet{\mathpzc}{OT1}{pzc}{m}{it}

\def\eqdefa{\buildrel\hbox{\footnotesize def}\over =}
\newcommand{\ve}{\varepsilon}
\newcommand{\ud}{\mathrm{d}}

\newcommand{\vv}{\mathbf{v}}

\newcommand{\xx}{\mathbf{x}}

\newcommand{\nn}{\mathbf{n}}

\newcommand{\hh}{\mathbf{h}}
\newcommand{\mm}{\mathbf{m}}

\newcommand{\GG}{\mathbf{G}}
\newcommand{\HH}{\mathbf{H}}
\newcommand{\UU}{\mathbf{U}}

\newcommand{\NN}{\mathbf{N}}

\newcommand{\DD}{\mathbf{D}}
\newcommand{\FF}{\mathbf{F}}

\newcommand{\JJ}{\mathbf{J}}
\newcommand{\II}{\mathbf{I}}

\newcommand{\CB}{\mathcal{B}}

\newcommand{\CC}{\mathcal{C}}

\newcommand{\CE}{\mathcal{E}}
\newcommand{\CN}{\mathcal{N}}
\newcommand{\CF}{\mathcal{F}}

\newcommand{\CA}{\mathcal{A}}

\newcommand{\CH}{\mathcal{H}}
\newcommand{\CL}{\mathcal{L}}

\newcommand{\CJ}{\mathcal{J}}

\newcommand{\MP}{\mathscr{P}}

\newcommand{\fE}{\mathfrak{E}}
\newcommand{\fF}{\mathfrak{F}}

\newcommand{\BS}{{\mathbb{S}^2}}
\newcommand{\BR}{{\mathbb{R}^3}}

\newcommand{\BOm}{\boldsymbol{\Omega}}

\usepackage{color}

\newtheorem{theorem}{Theorem}
\newtheorem{proposition}[theorem]{Proposition}
\newtheorem{lemma}[theorem]{Lemma}

\newtheorem{remark}[theorem]{Remark}
\numberwithin{theorem}{section}
\numberwithin{equation}{section}
\newcommand{\PP}{\mathbf{P}_R}

\newcommand{\TT}{\mathbf{T}}
\newcommand{\DP}{\mathbf{P}}
\newcommand{{\wfE}}{\widehat{\mathfrak{E}}}

\allowdisplaybreaks[4]
\title{Rigorous uniaxial limit of the Qian--Sheng inertial Q-tensor hydrodynamics for liquid crystals}

\author{Sirui Li\footnote{School of Mathematics and Statistics, Guizhou University, Guiyang 550025, China (srli@gzu.edu.cn)},
Wei Wang\footnote{School of Mathematical Sciences, Zhejiang University, Hangzhou 310058, China (wangw07@zju.cn)},
Qi Zeng\footnote{ School of Mathematics and Statistics, Guizhou University, Guiyang 550025, China (qzengmath@163.com)}
}
\date{}

\begin{document}
\maketitle
\begin{abstract}
This article is concerned with the rigorous connections between the inertial Qian--Sheng model and the Ericksen--Leslie model for the liquid crystal flow, under a more general condition of coefficients. More specifically, in the framework of Hilbert expansions, we show that:  (i) when the elastic coefficients tend to zero (also called the uniaxial limit), the smooth solution to the inertial Qian--Sheng model converges to that to the full inertial Ericksen--Leslie model; (ii) when the elastic coefficients and the inertial coefficient tend to zero simultaneously, the smooth solution to the inertial Qian--Sheng model converges to that to the noninertial Ericksen--Leslie model.


\end{abstract}


\section{Introduction}

Liquid crystals are a kind of complex fluids with local orientational order, in which physical properties lie between liquid and solid. Its local anisotropy typically originates from the nonuniform distribution of nonspherical rigid molecules. The classification of liquid-crystal theories is mainly based on different characterizations of the local anisotropy. As for the uniaxial nematic phase formed by rodlike molecules, the local anisotropy can be described by an orientational distribution function, a second-order symmetric traceless tensor, or a unit vector. The resulting hydrodynamics are molecular theories (such as the Doi--Onsager model \cite{DE}), tensor theories (such as the Beris--Edwards model \cite{BE} and the Qian--Sheng model \cite{QS} in the Landau--de Gennes framework and the molecular-theory-based tensor model \cite{HLWZZ} obtained by closure approximations) and vector theories (such as the Ericksen--Leslie model\cite{Ericksen1, Leslie1}), respectively.
Since they are derived from different physical considerations and are investigated widely in liquid crystals, it is important to explore the connection between different theories. In this article, we are concerned with the rigorous connection between the Ericksen--Leslie model and the Qian--Sheng model--a typical example of $Q$-tensor hydrodynamics.

The hydrodynamics for the uniaxial nematic phase, i.e., the well-known Ericksen--Leslie model, has been studied extensively (see \cite{LW4, Ball, WZZ4} and the references therein). In $\mathbb{R}^2$, the existence of global weak solutions is discussed \cite{Lin2, HX, HLW, WW}. Under a special assumption of initial datum, the global existence of weak solutions in $\mathbb{R}^3$ is investigated in \cite{LW2}. The well-posedness of smooth solutions to the original Ericksen--Leslie model has been studied for the whole space \cite{WZZ1, WW} and for bounded domains \cite{HNPS}. The role of Parodi's relation in the well-posedness and stability of the Ericksen--Leslie model is analyzed in \cite{WXL}. For an inertial analogue, the corresponding smooth well-posedness and global regularity for small datum are also considered \cite{JL2, CW, HJLZ1}. On the other hand, the $Q$-tensor hydrodynamics leads to a system, which couples the incompressible Navier--Stokes equation with the evolution equation of the $Q$-tensor field. The analysis concerning the $Q$-tensor dynamical model has been well-studied in recent years. The well-posedness results of the Beris--Edwards model have been established for the whole space \cite{PZ1, PZ2, HD, DHW} and for bounded or periodic domains \cite{ADL1, LYW1, Wilk}, respectively. Concerning the inertial Qian--Sheng model and its variations, the well-posedness results are studied in \cite{DZ, FRSZ}. More analytic results for the above two types of models are summarized in several review articles \cite{LW4, Ball, WZZ4}.

We turn to the connection between different liquid-crystal theories. The fundamental subjection usually involves the singular limit problem.
From molecular theories or tensor theories, one could derive the Ericksen--Leslie theory with its coefficients expressed by those in the molecular models or tensor models. Such derivations are based on the fact that the minimum of the bulk energy must be uniaxial and the whole procedure is done by the Hilbert expansion. In this aspect, recent research has produced important progress, but many issues still remain.
The connection between the Doi--Onsager model and the noninertial Ericksen--Leslie model has been investigated for formal derivations \cite{KD, EZ, HLWZZ} and for rigorous derivations \cite{WZZ2}, under the small Deborah number limit. The rigorous derivations of the Ericksen--Leslie model from $Q$-tensor hydrodynamics have been established in \cite{WZZ3, LWZ, LW}. When the biaxial minimizer of the bulk energy is considered, the connections between the biaxial frame hydrodynamics and the molecular-theory-based two-tensor hydrodynamics are studied, both in the sense of formal expansion \cite{LX1} and rigorous biaxial limit \cite{LX2}. In a sense of weak solutions, the relations between different dynamical theories are also explored in \cite{LYW2, XinZ}.

The aim of this article is to rigorously justify connections between the inertial Qian--Sheng model and the Ericksen--Leslie model in the framework of smooth solutions. The main novelty of our works is stated as follows:
\begin{itemize}
\item Under a more general condition of coefficients, we rigorously prove that the smooth solution to the inertial Qian--Sheng model converges to that of the full inertial Ericksen--Leslie model. The previous work \cite{LW} established this result based on a rather strong assumption: $\mu_1\gg J$, i.e., the damping coefficient is much greater than the inertia one. In this article, we remove the strong assumption and obtain an extension of the work \cite{LW}, see Section \ref{m0-section}.
\item We investigate the rigorous singular limit of the inertial Qian--Sheng model as two small parameters tend to zero simultaneously. Some attempts have been carried out when the inertial coefficient is regarded as the unique small parameter. The zero inertia limit has been studied for the inertial Ericksen--Leslie model \cite{JL1, JNT1} and for the inertial Qian--Sheng model \cite{LM}. At the moment, the solution to the inertial Ericksen--Leslie or Qian--Sheng model converges to the solution of the corresponding noninertial one. In our work, the elastic coefficient $\ve$ and the inertial coefficient $\eta$ are all considered as small parameters satisfying $\eta=\ve^m$ with $m$ a nonnegative integer. We rigorously justify that the smooth solution to the inertial Qian--Sheng model converges to that to the noninertial Ericksen--Leslie model. The structure and form of the limit model have changed essentially. One could refer to Section \ref{mneq0-section}.
\end{itemize}
The main strategy of our proof is to construct approximate solutions near the solution to the limit model by the Hilbert expansion, and then derive the uniform estimates for the difference between true solutions and certain approximate ones. Some new issues and challenges will arise since the $Q$-tensor system enjoys the nonlinear hyperbolic structure and the strong condition of coefficients is discarded. To deal with these difficulties, we will give a new and general way to obtain the desired estimates for the singular terms in the remainder system. A more delicate modified energy with a key additional term is provided to handle the loss of the strong assumption for coefficients. We also deeply study the intrinsic symmetric structure of the remainder system to eliminate higher-order derivative terms, and then provide relatively simple estimates compared with the previous work \cite{LW}. We will clarify the main challenges of our work at the end of Theorem \ref{main theorem}.

In the rest of this section, let us introduce some notations, followed by the Ericksen--Leslie theory and the Landau--de Gennes theory. Then, we state the main results.

{\it Notation and Conventions.}
The space of symmetric traceless tensors is defined by
\begin{align*}
    \mathbb{S}^3_0\eqdefa\{Q\in\mathbb{R}^{3\times3}: Q_{ij}=Q_{ji},~~Q_{ii}=0\},
\end{align*}
which is endowed with the inner product $Q_1:Q_2=Q_{1ij}Q_{2ij}$. Here, we have adopted the Einstein summation convention on repeated indices and will assume it throughout the article. The matrix norm on the five-dimensional space $\mathbb{S}^3_0$ is defined by $|Q|\eqdefa\sqrt{Q_{ij}Q_{ij}}$. For any two tensors $A,B\in \mathbb{S}_0^3$, we denote $(A\cdot B)_{ij}=A_{ik}B_{kj}$ and $A:B=A_{ij}B_{ij}$, and their commutator $[A,B]=A\cdot B-B\cdot A$. For any $Q_1,Q_2\in L^2(\mathbb{R}^3,\mathbb{R}^{3\times3})$, their inner product is defined as
\begin{align*}
\langle Q_1,Q_2\rangle\eqdefa\int_{\mathbb{R}^3}Q_{1ij}(\xx):Q_{2ij}(\xx)\ud\xx.
\end{align*}
The notation $\nn_1\otimes\nn_2$ represents the tensor product of two vectors $\nn_1$ and $\nn_2$, and the symbol $\otimes$ is usually omitted for simplicity.  We utilize $f_{,j}$ to denote $\partial_j f$ and $\II$ to denote the second-order identity tensor.

\subsection{Ericksen--Leslie theory}

The uniaxial hydrodynamics for nematic phases is called the Ericksen--Leslie theory \cite{Ericksen1, Leslie1}, which is a coupled system between the evolution equation of the unit vector field $\nn(t,\xx)$ and the Navier--Stokes equation of the velocity field $\vv(t,\xx)$. The full form of the Ericksen--Leslie model is given by
\begin{align}
&\nn\times(I\ddot{\nn}-\hh+\gamma_1\NN+\gamma_2\DD\cdot\nn)=0,\label{EL-vv}\\
    &\dot{\vv}=-\nabla p+\nabla\cdot(\sigma^L+\sigma^E), \label{EL-nn}\\
    &\nabla \cdot \vv=0, \label{EL-div}
\end{align}
where the notation $\dot{f}=(\partial_t+\vv\cdot\nabla) f$ stands for the material derivative, and $I$ is {\it the moment of inertial density} which is usually small. The inertial term $\ddot{\nn}$ is the material derivative of $\dot{\nn}$. The symmetric and skew-symmetric parts of the velocity gradient are denoted as, respectively,
\begin{align*}
\DD=\frac{1}{2}\left(\nabla\vv+(\nabla\vv)^T\right),~~\BOm=\frac{1}{2}\left(\nabla\vv-(\nabla\vv)^T\right).
\end{align*}
The co-rotational time flux of the director $\nn$ is defined by $\NN=\dot{\nn}-\BOm\cdot\nn$. The molecular field $\hh$ is given by
\begin{align*}
\hh=-\frac{\delta E_F}{\delta \nn}=-\frac{\partial E_F}{\partial\nn}+\nabla\cdot\frac{\partial E_F}{\partial(\nabla\nn)}.
\end{align*}
Here, $E_F$ is called the Oseen--Frank distortion energy, i.e.,
\begin{align*}
    E_F(\nn,\nabla\nn)=&\frac{k_1}{2}(\nabla\cdot\nn)^2+\frac{k_2}{2}(\nn\cdot(\nabla\times\nn))^2+\frac{k_3}{2}|\nn\times(\nabla\times\nn)|^2\\
    &+\frac{k_2+k_4}{2}(\text{Tr}(\nabla\nn)^2-(\nabla\cdot\nn)^2),
\end{align*}
where the constants $k_1, k_2$ and $k_3$ represent deformation moduli of {\it splay}, {\it twist} and {\it bend}, respectively, and the last term is a null Lagrangian which is related to boundary terms.

In the equation of $\vv$, the pressure $p$ ensures the incompressible condition (\ref{EL-div}).  The constitutive equations for the viscous (Leslie) stress $\sigma^L$ and the elastic (Ericksen) stress $\sigma^E$ are given by, respectively,
\begin{align}
\sigma^L&=\alpha_1(\nn\nn:\DD)\nn\nn+\alpha_2\nn\NN+\alpha_3\NN\nn+\alpha_4\DD+\alpha_5\nn\nn\cdot\DD+\alpha_6\DD\cdot\nn\nn,\\
   \sigma^E&=-\frac{\partial E_F}{\partial(\nabla\nn) }\cdot(\nabla\nn)^T,
\end{align}
where the constants $\alpha_1,\dots,\alpha_6$, called the Leslie viscosities, together with the coefficients $\gamma_1, \gamma_2$ satisfy
 \begin{align}
      \alpha_2&+\alpha_3=\alpha_6-\alpha_5,\label{Parodi1}\\
      \gamma_1=\alpha_3&-\alpha_2,\quad\gamma_2=\alpha_6-\alpha_5.\label{Parodi2}
 \end{align}
The equality  (\ref{Parodi1}) is the so-called Parodi relation. The relations (\ref{Parodi1})--(\ref{Parodi2}) will ensure a basic energy law of the system (\ref{EL-vv})--(\ref{EL-div}):
\begin{align}
    \frac{\ud}{\ud t}\int_{\BR}\Big(&\frac{1}{2}|\vv|^2+\frac{I}{2}|\dot{\nn}|^2+E_F\Big)\ud \xx=-\int_{\BR}\bigg(\Big(\alpha_1+\frac{\gamma_2^2}{\gamma_1}\Big)(\DD:\nn\nn)+\alpha_4|\DD|^2\nonumber\\
    &+\Big(\alpha_5+\alpha_6-\frac{\gamma_2^2}{\gamma_1}\Big)|\DD\cdot\nn|^2+\frac{1}{\gamma_1}|\nn\times(\hh-I\ddot{\nn})|^2\bigg)\ud\xx.\label{EL-energy-law}
\end{align}
Here, the coefficients in (\ref{EL-energy-law}) satisfy
\begin{align}   \alpha_4>0,\quad 2\alpha_4+\alpha_5+\alpha_6-\frac{\gamma_2^2}{\gamma_1}>0,\quad\alpha_1+\frac{3}{2}\alpha_4+\alpha_5+\alpha_6>0,\label{alpha-relation}
\end{align}
which implies that the energy in (\ref{EL-energy-law}) is dissipated (see Lemma \ref{nn, D} for details).

It is worth emphasizing that (\ref{EL-vv})--(\ref{EL-div}) is a hyperbolic-parabolic system due to the appearance of the inertial term $I\ddot{\nn}$. If the inertial term is neglected, then the system (\ref{EL-vv})--(\ref{EL-div}) becomes immediately a parabolic Ericksen--Leslie system. In general, the inertial system with the hyperbolic feature will bring more challenges in analysis, compared with its noninertial counterpart.

\subsection{Landau--de Gennes  theory}

In this theory, the local orientational order of nematic phases is described by a second-order symmetric traceless tensor $Q(\xx)$. Physically, $Q(\xx)$ can be understood as the second-order moment of $f$:
\begin{align*}
Q(\xx)=\int_\BS\Big(\mm\mm-\frac{1}{3}\II\Big)f(\xx,\mm)\ud\mm,
\end{align*}
where $f(\xx,\mm)$ represents the distribution function on $\BS$ with the orientation parallel to $\mm$ at material point $\xx$. We could classify $Q(\xx)$ into three classes.
The tensor $Q(\xx)$ is said to be {\it isotropic} if it has three equal eigenvalues, i.e., $Q(\xx)=0$, {\it uniaxial} if it has two equal nonzero eigenvalues, and {\it biaxial} if it has three distinct  eigenvalues.

The Landau--de Gennes free energy contains two parts, the bulk energy and the elastic energy, which is given by
\begin{align}
\mathcal{F}(Q,\nabla Q)
&=\int_{\BR}\bigg\{-\frac{a}2\text{Tr}(Q^2)
-\frac{b}{3}\text{Tr}(Q^3)+\frac{c}{4}(\text{Tr}(Q^2))^2\nonumber\\
&\qquad\quad+\frac{1}{2}\Big(L_1|\nabla Q|^2+L_2Q_{ij,j}Q_{ik,k}
+L_3Q_{ij,k}Q_{ik,j}\Big) \bigg\}\ud\xx\nonumber\\
&\eqdefa\int_{\BR}\Big(f_b(Q)+f_e(\nabla Q)\Big)\ud\xx,\label{eq:Landau-energy}
\end{align}
where the bulk energy density $f_b$ describes transitions between homogeneous phases, and nonnegative parameters $a, b, c$ depend on the material and temperature.
The elastic energy density $f_e$ penalizing spatial nonhomogeneity contains some quadratic terms of $\nabla Q$, and the material dependent elastic constants $L_i(i=1,2,3)$ are usually small. More detailed introductions can be referred to \cite{DeG,MN}.

To ensure the strict positiveness of the elastic energy, we assume that the elastic coefficients $L_i(i=1,2,3)$ satisfy
\begin{align*}
    L_1>0,\quad L_1+L_2+L_3>0,
\end{align*}
i.e., there exists a constant $L_0=L_0(L_1,L_2,L_3)>0$, such that
\begin{align}\label{L:postive}
\int_{\BR}f_e(\nabla Q)\ud\xx\geq L_0\|\nabla Q\|^2_{L^2}.
\end{align}

The inertial Qian--Sheng model  \cite{QS} is a representative $Q$-tensor hydrodynamics based on the Landau--de Gennes framework, which is given by
\begin{align}
J\ddot{Q}+\mu_1\dot{Q}=&~\HH-\frac{\mu_2}{2}\DD+\mu_1[\BOm,  Q], \label{eq:Q-general-intro1}\\
\partial_t\vv+\vv\cdot\nabla\vv=&-\nabla{p}+\nabla\cdot\big(\sigma+\sigma^{d}\big),
\label{eq:Q-general-intro2}\\
\nabla\cdot\vv=&~0, \label{eq:Q-general-intro3}
\end{align}
where $\dot{Q}=(\partial_t+\vv\cdot\nabla)Q$ and $\ddot{Q}=(\partial_t+\vv\cdot\nabla)\dot{Q}$. The constant $J$ represents the {\it inertial density} which is responsible for the hyperbolic feature of the system. In this article, we will consider the case that $J$ is not a small parameter, and the case that $J$ is very small, respectively. Moreover, the viscous stress $\sigma$, the distortion stress $\sigma^d$ and the molecular field $\HH$ are expressed by
\begin{align} \label{vis-stressQ}\nonumber
 \sigma(Q,\vv)=&~\beta_1 Q(Q:\DD)+\beta_4 \DD+\beta_5\DD\cdot Q+\beta_6Q\cdot\DD
+\beta_7(\DD\cdot Q^2+Q^2\cdot \DD)\\
&+\frac{\mu_2}{2}(\dot{Q}-[\BOm, Q])+\mu_1\big[Q, (\dot{Q}-[\BOm, Q])\big],  \\
\sigma^d=&~\sigma^d(Q,Q), \label{vis-fieldQ}\\
\HH_{ij} =& -\Big(\frac{\delta{\mathcal{F}(Q, \nabla Q)}}{\delta Q}\Big)_{ij}=-\frac{\partial f_b}{\partial Q_{ij}}+\partial_k\Big(\frac{\partial f_e}{\partial Q_{ij, k}}\Big),\nonumber
\end{align}
where the tensor $\sigma^d(Q,\widetilde{Q})$ is defined as
\begin{align*}
\sigma^d_{ji}(Q, \widetilde{Q})\eqdefa -\frac{\partial f_e}{\partial Q_{kl, j}}\partial_i\widetilde{Q}_{kl}=-(L_1Q_{kl, j}\widetilde{Q}_{kl, i}+L_2Q_{km, m}\widetilde{Q}_{kj, i}+L_3Q_{kj, l}\widetilde{Q}_{kl, i}).
\end{align*}
 The viscosity coefficients $\beta_1,  \beta_4,  \beta_5,  \beta_6,  \beta_7,  \mu_1$,  and $\mu_2$ in (\ref{vis-stressQ}) can be linked by
\begin{align}
\beta_6-\beta_5=\mu_2,\label{Q-Parodi}
\end{align}
which corresponds to the Parodi relation (\ref{Parodi1}). We also assume that the coefficients in the system (\ref{eq:Q-general-intro1})--(\ref{eq:Q-general-intro3}) satisfy
\begin{equation}\label{Beta-relation}
    \left\{\begin{aligned}
        &\beta_1,\beta_4,\mu_1>0,~\beta_4-\frac{\mu_2^2}{4\mu_1}>0,~\beta_7\ge 0;\\&(\beta_5+\beta_6)^2<8\beta_7\left(\beta_4-\frac{\mu_2^2}{4\mu_1}\right)~\text{if }~\beta_7\ne 0;~\beta_5+\beta_6=0~\text{if }\beta_7=0.
    \end{aligned} \right.
\end{equation}
The relations (\ref{Q-Parodi})--(\ref{Beta-relation}) will guarantee that the system (\ref{eq:Q-general-intro1})--(\ref{eq:Q-general-intro3}) enjoys the following energy dissipation law \cite{LW}:
\begin{align}
    &\frac{\ud}{\ud t}\bigg(\int_{\BR}\frac{1}{2}(|\vv|^2+J|\dot{Q}|^2)\ud\xx+\CF(Q,\nabla Q)\bigg)\nonumber\\
    &\quad=-\beta_1\|Q:\DD\|^2_{L^2}-\left(\beta_4-\frac{\mu_2^2}{4\mu_1}\right)\|\DD\|^2_{L^2}-(\beta_5+\beta_6)\langle\DD\cdot Q,\DD\rangle\nonumber\\
    &\qquad-2\beta_7\|\DD\cdot Q\|^2_{L^2}-\mu_1\left\|\dot{Q}-[\BOm,Q]+\frac{\mu_2}{2\mu_1}\DD\right\|^2_{L^2}.
\end{align}

\subsection{Main results}

This subsection will be devoted to formulating our main results. In the inertial Qian--Sheng model, two different small parameters originate from the elastic coefficients $L_i(i=1,2,3)$ and the inertial constant $J$, respectively. Since the elastic coefficients are very small, we consider the rescaled energy functional with a small parameter $\ve$,
\begin{align*}
\CF^{\ve}(Q,\nabla Q)=\int_{\BR}\left(\frac{1}{\ve}f_b(Q)+f_e(\nabla Q)\right)\ud \xx,
\end{align*}
where the coefficients satisfy $a, b, c, L_i(1\le i\le 3)\sim O(1)$. Another is the small inertial density parameter denoted by $\eta$. The relationship between $\ve$ and $\eta$ determines the type of problem we have. In this article, we take $\eta=\ve^m$ with $m$ being a nonnegative integer.

The main aim of this article is to study the singular limit problem of the inertial Qian--Sheng system with small parameter $\ve$, which is given by
\begin{align}
\ve^mJ\ddot{Q}^{\ve}+\mu_1\dot{Q}^{\ve}=&~\HH^{\ve}-\frac{\mu_2}{2}\DD^{\ve}+\mu_1[\BOm^{\ve},  Q^{\ve}], \label{eq:Q}\\
\partial_t{\vv}^{\ve}+\vv^{\ve}\cdot\nabla {\vv}^{\ve}=&-\nabla{p^{\ve}}+\nabla\cdot\Big(\sigma(Q^{\ve}, \vv^{\ve})+\sigma^{d}(Q^{\ve}, Q^{\ve})\Big),
\label{eq:vv}\\
\nabla\cdot\vv^{\ve}=&~0, \label{eq:free div}
\end{align}
where
$\dot{Q}^{\ve}=(\partial_t+\vv^{\ve}\cdot\nabla )Q^{\ve}$, $\ddot{Q}^{\ve}=(\partial_t+\vv^{\ve}\cdot\nabla)Q^{\ve}$
and
\begin{align*}
    \DD^{\ve}=\frac{1}{2}\left(\nabla\vv^{\ve}+(\nabla\vv^{\ve})^T\right),~\BOm^{\ve}=\frac{1}{2}\left(\nabla\vv^{\ve}-(\nabla\vv^{\ve})^T\right).
\end{align*}
The viscous stress $\sigma(Q^{\ve}, \vv^{\ve})$ and the distortion stress $\sigma^{d}(Q^{\ve}, Q^{\ve})$ can be expressed by (\ref{vis-stressQ}) and (\ref{vis-fieldQ}), respectively.
The molecular field $\HH^{\ve}$ is given by
\[\HH^{\ve}=-\frac{\delta\CF^{\ve}}{\delta Q^\ve}=-\frac{1}{\ve}\CJ(Q^{\ve})-\CL(Q^{\ve}),\]
where two operators are defined as, respectively,
\begin{align}
    \CJ(Q)=&-aQ-bQ^2+c|Q|^2Q+\frac{b}{3}|Q|^2\II,\label{CJ-LP}\\
    (\CL(Q))_{kl}=&-\Big(L_1\Delta Q_{kl}+\frac{1}{2}(L_2+L_3)\Big(Q_{km,ml}+Q_{lm,mk}-\frac{2}{3}\delta_{kl}Q_{ij,ij}\Big)\Big).\label{CL-LP}
\end{align}

We denote by ${\chi}$ the characteristic function given by
\begin{align}\label{chi-fun}
    \chi(m)=\begin{cases}
        1,\quad m=0,\\
        0,\quad m\in\mathbb{Z}^+.
    \end{cases}
\end{align}
The relationship between coefficients in the Qian--Sheng model and those in the corresponding limit model can be expressed by
 \begin{equation}\label{coefficients}
       \left\{\begin{aligned}
           &\alpha_1=\beta_1s^2,\quad\alpha_2=\frac{1}{2}\mu_2s-\mu_1s^2,\\
           &\alpha_3=\frac{1}{2}\mu_2s+\mu_1s^2,\quad\alpha_4=\beta_4-\frac{s}{3}(\beta_5+\beta_6)+\frac{2}{9}\beta_7s^2,\\
           &\alpha_5=\beta_5s+\frac{1}{3}\beta_7s^2,\quad\alpha_6=\beta_6s+\frac{1}{3}\beta_7s^2,\\
           &\gamma_1=2\mu_1s^2,\quad\gamma_2=\mu_2s,\quad I={{\chi}}(m)2s^2J,\\
           &k_1=k_3=(2L_1+L_2+L_3)s^2,\quad k_2=2L_1s^2,\quad k_4=L_3s^2.
       \end{aligned}\right.
   \end{equation}
If $\chi(m)=1$, then (\ref{coefficients}) are defined as the coefficients in the full inertial Ericksen--Leslie model. If $\chi(m)=0$, i.e., the inertial constant $I=0$, then (\ref{coefficients}) becomes the coefficients in the noninertial Ericksen--Leslie model.

We now state the main results. Let $\CH_{\nn}$ be the Hessian of the bulk energy $f_b$ at its critical points $Q^{\ast}$, and the projection operators $\MP^{in}$ and $\MP^{out}$ be defined in (\ref{projection_in1})--(\ref{projection_out1}) concerning the kernel of $\CH_{\nn}$ in the Section \ref{crit-liner}.

\begin{theorem}\label{main theorem}
Let $m$ be a nonnegative integer, assume that $(\nn(t,\xx), \vv(t,\xx))$ is a smooth solution of the system {\em (\ref{EL-vv})--(\ref{EL-div})} on $[0, T]$ with coefficients defined by {\em (\ref{coefficients})}, and  satisfies
\begin{align}\label{initial-regularity}
     (\vv,\nabla \nn,{{\chi}}(m)\partial_t \nn)\in L^\infty([0, T];H^{\ell}),\quad {\ell\ge 20.}
\end{align}
    Let $Q_0(t,\xx)=s(\nn(t,\xx)\nn(t,\xx)-\frac{1}{3}\II)$ be the uniaxial minimizer of the bulk energy $f_b$, and the functions $(Q_1,Q_2,Q_3,\vv_1,\vv_2)$ are determined by Proposition \ref{prop Q(1)} for $m=0$ and by Proposition \ref{nu:prop Q(1)} for $m\in\mathbb{Z}^+$. Suppose that the initial data $(Q^\ve(0,\xx),\partial_tQ^\ve(0,\xx),\vv^\ve(0,\xx))$ takes the form
   \begin{align*}
&Q^\ve(0,\xx)=\sum_{k=0}^3\ve^kQ_k(0,\xx)+\ve^3Q_R(0,\xx),~~\vv^\ve(0,\xx)=\sum_{k=0}^2\ve^k\vv_k(0,\xx)+\ve^3\vv_R(0,\xx),\\
 &\partial_tQ^\ve(0,\xx)=\sum_{k=0}^3\ve^{k}\partial_tQ_k(0,\xx)+\ve^{3}\partial_tQ_R(0,\xx),
   \end{align*}
where $(Q_R(0,\xx),\partial_tQ_R(0,\xx),\vv_R(0,\xx))$ fulfills
\begin{align}
\|\vv_R(0,\xx)\|_{H^2}+\|Q_R(0,\xx)\|_{H^3}+\ve^{\frac{m}{2}}\|\partial_t Q_R(0,\xx)\|_{H^2}+\frac{1}{\ve}\big\|\MP^{out}(Q_R)(0,\xx)\big\|_{L^2}\leq E_0.\label{initial condition}
\end{align}
Then there exists $\ve_0$ and $E_1>0$ such that for  all $\ve<\ve_0$, the Qian--Sheng model {\em (\ref{eq:Q})--(\ref{eq:free div})} has a unique solution $(Q^\ve(t,\xx),\vv^\ve(t,\xx))$ on $[0,T]$ that possesses the following Hilbert expansion:
\[Q^\ve(t,\xx)=\sum_{k=0}^3\ve^kQ_k(t,\xx)+\ve^3Q_R(t,\xx),~~\vv^\ve(t,\xx)=\sum_{k=0}^2\ve^k\vv_k(t,\xx)+\ve^3\vv_R(t,\xx),\]
where, for any $t\in[0,T]$, the remainder $(Q_R,\vv_R)$ satisfies \[\widetilde{{\fE}}_m(Q_R,\vv_R)\leq E_1.\]
Here, the functional $\widetilde{{\fE}}_m(Q,\vv)$ is defined by
\begin{align}
    \widetilde{{\fE}}_{m}(Q,\vv)(t)\eqdefa\int_{\BR}&\bigg\{\Big(|\vv|^2+\ve^{m}J|\partial_tQ|^2+\frac{1}{\ve}\CH_{\nn}^\ve(Q):Q+|Q|^2\Big)\nonumber\\&+\ve^2\Big(|\nabla\vv|^2+\ve^{m} J|\partial_t\partial_iQ|^2+\frac{1}{\ve}\CH_\nn^\ve(\partial_iQ):\partial_iQ\Big)\nonumber\\&+\ve^4\Big(|\Delta\vv|^2+\ve^{{m}}J|\Delta\partial_tQ|^2+\frac{1}{\ve}\CH_\nn^\ve(\Delta Q):\Delta Q\Big)\bigg\}\ud\xx,\label{definition of widetile(fE)}
\end{align}
and $\CH_\nn^\ve(Q)=\CH_\nn(Q)+\ve\CL(Q)$, and the constant $E_1$ depends on $m$ but is independent of $\ve$.
\end{theorem}

\begin{remark}
We explain the meaning of $\chi(m)$. If $m=0$, then it means that we only consider the uniaxial limit that the small elastic coefficients tend to zero. The limit system, the equations obtained by the Hilbert expansion, and the remainder system are all hyperbolic. If $m\in\mathbb{Z}^+$, then it implies that we are concerned with the uniaxial limit that two small parameters (the elastic coefficients and the inertial coefficient) approach to zero simultaneously. This will bring about essential differences: except that the remainder system is hyperbolic, the limit model and the equations obtained by the Hilbert expansion are all parabolic. Therefore, Theorem \ref{main theorem} contains the results under two different cases.
\end{remark}

\begin{remark}
When $m\in\mathbb{Z}^+$, i.e., $\chi(m)=0$, the condition {\em(\ref{initial-regularity})} can be replaced by
    \begin{align*}
        (\vv,\nabla \nn)\in C([0, T];H^{\ell}),\quad {\ell\ge 20,}
    \end{align*}
since $(\vv,\nn)$ is a smooth solution of a parabolic system, i.e., the noninertial Ericksen--Leslie model {\em (\ref{EL-vv})--(\ref{EL-div})} where $I=0$.
\end{remark}

\begin{remark}
    If the damping coefficient $\mu_1$ and the viscosity coefficient $\mu_2$ satisfy $\mu_1=\mu_2=0$ when $m=0$, the same results as Theorem \ref{main theorem} still hold.
\end{remark}

To illustrate the idea of the proof Theorem \ref{main theorem}, we provide a short overview. To begin with, we make the Hilbert expansion of the solution $(Q^{\ve},\vv^\ve)$
with respect to the small parameter $\ve$:
\begin{align*}
    Q^{\ve}(t,\xx)=&Q_0(t,\xx)+\ve Q_1(t,\xx)+\ve^2 Q_2(t,\xx)+\ve^3 Q_3(t,\xx) +\ve^3 Q_R(t,\xx),\\
    \vv^{\ve}(t,\xx)=&\vv_0(t,\xx)+\ve\vv_1(t,\xx)+\ve^2\vv_2(t,\xx)+\ve^3\vv_R(t,\xx).
\end{align*}
Substituting the above expansions into the system (\ref{eq:Q})--(\ref{eq:free div}), we obtain a series of equations for $(Q_k,\vv_k;Q_3)(0\leq k\leq2)$(see Section \ref{Hilbert-expansion}). The O($\ve^{-1}$) equation gives $\CJ(Q_0)=0$, which implies from Proposition \ref{critical-points} that $Q_0=s(\nn\nn-\frac{1}{3}\II)$ is the uniaxial minimizer of the bulk energy $f_b$,
for some $\nn\in\BS$ and $s=\frac{b+\sqrt{b^2+24ac}}{4c}$.
The $O(1)$ system gives the uniaxial vectorial hydrodynamics. The uniaxial limit is the full inertial Ericksen--Leslie system for $m=0$, while the uniaxial limit is the noninertial version of the Ericksen--Leslie system for $m\in\mathbb{Z}^+$, see Proposition \ref{0toEL} and Proposition \ref{0toEL-m>0}, respectively.

The proof of Theorem \ref{main theorem} is based on mainly two ingredients: the existence of smooth solutions to the equations of $(Q_k,\vv_k;Q_3)(0\leq k\leq 2)$, and the uniform estimate for the remainder system.

The first ingredient relies on the local existence of smooth solution to the Ericksen--Leslie system on $[0,T]$, which has been established in \cite{WZZ1,WW} for the noninertial case and in \cite{JL2} for the inertial case. The basic approach to solve $(Q_1,\vv_1)$ is to cancel the non-leading term in the $O(\ve)$ equations by the kernel space $\text{Ker}\CH_{\nn}$, and to derive a linear system and then to show that such a system has a closed energy estimate. Thus, the existence of the smooth solutions $(Q_k,\vv_k;Q_3)(0\leq k\leq 2)$ can be guaranteed (see Proposition \ref{prop Q(1)} for the case of $m=0$ and Proposition \ref{nu:prop Q(1)} for the case of $m\in\mathbb{Z}^+$).

The second ingredient is to prove the uniform boundedness of the remainders $(Q_R,\PP,\vv_R)$. The remainder system has the following abstract form:
\begin{align*}
     &\ve^m J(\partial_t+\vv^\ve\cdot\nabla)\PP=-\frac{1}{\ve}\CH_\nn^\ve(Q_R)-\frac{\mu_2}{2}\DD_R-\mu_1(\PP-[\BOm_R,Q_0])+\cdots,\\
   & \partial_t{\vv}_R=-\nabla p_R+\nabla\cdot \Big\{\beta_4\DD_R
    +\frac{\mu_2}{2}(\PP-[\BOm_R,Q_0])+\mu_1\big[Q_0,(\PP-[\BOm_R,Q_0])\big]\Big\}+\cdots,
\end{align*}
where $\PP=(\partial_t+\tilde{\vv}\cdot\nabla)Q_R+\vv_R\cdot\nabla Q^{\ve}$ and $\tilde{\vv}=\sum^2_{k=0}\ve^k\vv_k$, and $m$ is a given nonnegative integer. It is not difficult to observe that $\PP$ is not fully expanded in the remainder system. The greatest benefit is that some higher-order derivative terms can be eliminated by the symmetric structure of the system, without having to deal with some tedious higher-order estimates as shown in \cite{LW}.

The main obstacle towards the uniform estimate comes from the singular term $\frac{1}{\ve}\CH^{\ve}_\nn(Q_R)$.
Since the remainder system is a hyperbolic system, the term
$\frac{1}{\ve}\left\langle\CH_\nn^\ve(Q_R),\PP\right\rangle$ will be brought into the energy. However, the linearized operator $\CH^{\ve}_{\nn}$ depends on the time $t$, and its time derivative will cause further the difficult terms contained in the energy. To handle these difficulties, a more delicate modified energy is introduced in the previous work \cite{LW}, but a rather strong condition of coefficients $\mu_1\gg J$ is imposed. Therefore, a new difficulty of this article is how to derive the uniform estimate for the remainder system when the strong assumption of coefficients is removed.

For this purpose, we introduce a suitable energy functional $\fE(t)$ in (\ref{energyE}), which includes a key additional term $\frac{1}{2}\int_{\mathbb{R}^3}M|Q_R|^2\ud\xx$ with $M$ a adjustable and sufficiently large constant. To control the singular term $\frac{1}{\ve}\left\langle\CH_\nn^\ve(Q_R),\PP\right\rangle$, the key estimate is given by
\begin{align*}
    -\frac{1}{\ve}\left\langle\CH_\nn^\ve(Q_R),\PP\right\rangle\leq\frac{\ud}{\ud t}\Big(-\frac{1}{2\ve}\langle\CH_\nn^\ve(Q_R),Q_R\rangle-\ve^m\CA(Q_R,\PP)\Big)+\cdots,
\end{align*}
where $\CA(Q_R,\PP)$ is defined by (\ref{CA-t}). The elimination of the singular terms relies also on the structure of the remainder system. We may refer to Lemma \ref{A key lemma} for $m=0$ and Lemma \ref{nu:A key lemma} for $m\in\mathbb{Z}^+$. In order to ensure the positive definiteness of the energy, we choose the suitable positive constant $M$ such that $\ve^m|\CA(Q_R,\PP)|\leq\frac{1}{2}\fE(t)$. Then we obtain the energy estimates of the remainder terms (see Proposition \ref{proposition-PP} and Proposition \ref{nu:proposition-PP}, respectively).

At last, we show that ${\fE}_{m}(t)$ defined by $(\ref{energyE})$ and $\widetilde{{\fE}}_{m}(t)$ defined by $(\ref{definition of widetile(fE)})$ are equivalent, and thus complete the proof of Theorem \ref{main theorem}.

\section{Critical points and the linearized operator}\label{crit-liner}

To study the limit $\ve\rightarrow 0$, we need to characterize the minimizer of the bulk energy $f_b$.
 A tensor $Q^{\ast}$ is called a critical point of $f_b(Q)$ if $\CJ(Q^{\ast}):=\frac{\partial f_b}{\partial Q}\big|_{Q=Q^{\ast}}=0$, where $\CJ(Q)$ is defined by (\ref{CJ-LP}). The characterization of critical points \cite{Maj, WZZ3} is given below.

\begin{proposition}\label{critical-points}
$\CJ(Q^{\ast})=0$ if and only if $Q^{\ast}=s(\nn\nn-\frac13\II)$ for some $\nn\in\mathbb{S}^2$,  where $s=0$ or a solution of {$2cs^2-bs-3a=0$},  that is,
\begin{align*}
s_{1}=\frac{b+\sqrt{b^2+24ac}}{4c} ~or~ s_{2}=\frac{b-\sqrt{b^2+24ac}}{4c}.
\end{align*}
Moreover,  the critical point $Q^{\ast}=s(\nn\nn-\frac13\II)$ is stable if $s=s_1$.
\end{proposition}
When the bulk energy $f_b$ is a fourth degree polynomial, it can be seen from Proposition \ref{critical-points} that the critical points of $f_b$ can only be isotropic or uniaxial. In this article, we only consider that the global minimizer of the bulk energy is uniaxial.

For a given critical point $Q^{\ast}=s(\nn\nn-\frac{1}{3}\II)$,  the linearized operator $\CH_{Q^{\ast}}$ of $\CJ(Q)$ around $Q^{\ast}$ can be defined as
\begin{align*}
\CH_{Q^{\ast}}(Q)=-aQ-b(Q^{\ast}\cdot Q+Q\cdot Q^{\ast})+c|Q^{\ast}|^2Q+2(Q^{\ast}:Q)\Big(cQ^{\ast}+\frac{b}{3}\II\Big).
\end{align*}
In other words, the linearized operator $\CH_{Q^{\ast}}$ is the Hessian of the bulk energy $f_b$ at $Q^{\ast}$.
For any $Q_i\in\mathbb{R}^{3\times3}(i=1, 2, 3)$, we define
\begin{align*}
\CB(Q_1, Q_2)\eqdefa&~Q_1\cdot Q_2+Q^T_2\cdot Q^T_1-\frac{1}{3}(Q_1:Q_2)\II, \\
\CC(Q_1, Q_2, Q_3)\eqdefa&~ Q_1(Q_2:Q_3)+Q_2(Q_1:Q_3)+Q_3(Q_1:Q_2).
\end{align*}
Then two operators $\CJ(Q)$ and $\CH_{Q^{\ast}}(Q)$ can be expressed as,  respectively,
\begin{align}
\CJ(Q)=&-aQ-\frac{b}{2}\CB(Q, Q)+\frac{c}{3}\CC(Q, Q, Q), \label{CJ=abc}\\
\CH_{Q^{\ast}}(Q)=&-aQ-b\CB(Q, Q^{\ast})+c\CC(Q, Q^{\ast}, Q^{\ast}).\label{CH=abc}
\end{align}
Since $Q^{\ast}=Q^{\ast}(\nn)$ can be viewed as a function of $\nn$, a direct calculation yields
\begin{align}\label{CH-nn}
\CH_{Q^{\ast}}(Q)=&~bs\Big(Q-(\nn\nn\cdot Q+Q\cdot\nn\nn)+\frac23(Q:\nn\nn)\II\Big)+2cs^2(Q:\nn\nn)\Big(\nn\nn-\frac13\II\Big)\nonumber\\
\eqdefa&~\CH_{\nn}(Q).
\end{align}

The kernel space of the linearized operator $\CH_{\nn}$,  being a two-dimensional subspace of $\mathbb{S}^3_0$,  can be defined by
\begin{align*}
{\rm Ker}\CH_{\nn}\eqdefa\{\nn\nn^{\perp}+\nn^{\perp}\nn\in\mathbb{S}^3_0: \nn^{\perp}\in\mathbb{V}_{\nn}\},
\end{align*}
for any given $\nn\in\mathbb{S}^2$,  where $\mathbb{V}_{\nn}\eqdefa\{\nn^{\perp}\in\BR: \nn^{\perp}\cdot\nn=0\}$. Let $\mathscr{P}^{in}$ be the projection operator from $\mathbb{S}^3_0$ to
$\text{Ker}\CH_{\nn}$ and $\MP^{out}$ the projection operator from $\mathbb{S}^3_0$ to
$(\text{Ker}\CH_{\nn})^{\perp}$, respectively.
Two projection operators are given by (see \cite{WZZ3} for details)
\begin{align}
 \MP^{in}(Q)
 =&~(\nn\nn\cdot Q+Q\cdot\nn\nn)-2(Q:\nn\nn)\nn\nn, \label{projection_in1}\\
 \MP^{out}(Q)
 =&~Q-(\nn\nn\cdot Q+Q\cdot\nn\nn)+2(Q:\nn\nn)\nn\nn.\label{projection_out1}
\end{align}

The properties of the linearized operator $\CH_{\nn}$ will play a key role in the analysis of the Hilbert expansion, which can be found in \cite{WZZ3}.
\begin{proposition}\label{linearized-oper-prop}
{\rm(i)} For any $\nn\in\BS$,  it follows that $\CH_{\nn}\mathrm{Ker}\CH_{\nn}=0$,  i.e.,  $\CH_{\nn}(Q)\in(\mathrm{Ker}\CH_{\nn})^{\perp}$.\\
{\rm(ii)} There exists a constant $C_0=C_0(a, b, c)>0$ such that for any $Q\in(\mathrm{Ker}\CH_{\nn})^{\perp}$,
\begin{align*}
\CH_{\nn}(Q):Q\geq C_0|Q|^2.
\end{align*}
{\rm(iii)} $\CH_{\nn}$ is a 1-1 map on $(\mathrm{Ker}\CH_{\nn})^{\perp}$ and its inverse $\CH^{-1}_{\nn}$ is given by
\begin{align}\label{HM-inverse}
\CH^{-1}_{\nn}(Q)=&\frac{1}{bs}\Big(Q-(\nn\nn\cdot Q+Q\cdot\nn\nn)+\frac23(Q:\nn\nn)\II\Big)\nonumber\\
&+\frac{4b+2cs}{bs(4cs-b)}(Q:\nn\nn)\Big(\nn\nn-\frac13\II\Big).
\end{align}
\end{proposition}

\section{The Hilbert expansion and the remainder system}\label{Hilbert-expansion}

This section focuses on performing the Hilbert expansion (also called the Chapman--Enskog expansion) of solutions with respect to the small parameter $\ve$, and deriving the system of the remainder terms.

\subsection{The Hilbert expansion}

Let ($Q^\ve,\vv^\ve$) be a solution of the system (\ref{eq:Q})--(\ref{eq:free div}). We make the following Hilbert expansion:
\begin{align}
    &Q^\ve=\sum^3_{k=0}\ve^kQ_k+\ve^3 Q_R\eqdefa \widetilde{Q}+\ve^3Q_R,\label{Q-ve}\\
    &\vv^\ve=\sum^2_{k=0}\ve^k\vv_k+\ve^3\vv_R\eqdefa \tilde{\vv}+\ve^3\vv_R,\label{v-ve}
\end{align}
where $Q_k(0\leq k\leq 3)$ and $\vv_k(0\leq k\leq2)$ are independent of $\ve$, and $(Q_R,\vv_R)$ represents the remainder term depending upon $\ve$. Armed with the expansions (\ref{Q-ve})--(\ref{v-ve}), we have
\begin{align}\label{dotQ-expansion}
 \dot{Q}^{\ve}=
        &(\partial_t+\vv^\ve\cdot\nabla)Q^\ve\nonumber\\=&\DP_0+\ve\DP_1+\ve^2(\DP_2+\vv_1\cdot \nabla Q_1)+\ve^3\bigg(\PP+\partial_tQ_3+\sum_{i+j\ge 3}\ve^{i+j-3}\vv_i\cdot\nabla Q_j\bigg),
\end{align}
where $\DP_k(0\leq k\leq2)$, independent of $\ve$, are given by, respectively,
\begin{align*}
    \DP_0=&(\partial_t+\vv_0\cdot\nabla)Q_0,\quad\DP_1=(\partial_t+\vv_0\cdot\nabla)Q_1+\vv_1\cdot\nabla Q_0,\\
    \DP_2=&(\partial_t+\vv_0\cdot\nabla)Q_2+\vv_2\cdot \nabla Q_0,
\end{align*}
and the remainder term $\PP$, depending on $\ve$, is expressed by
\begin{align*}
\PP=(\partial_t+\vv^\ve\cdot\nabla)Q_R+\vv_R\cdot\nabla{(Q^\ve-\ve^3Q_R)}.
\end{align*}
It is worth emphasizing that to maintain the intrinsic structure of the remainder system, the term $\PP$ in (\ref{dotQ-expansion}) is not fully expanded. This is because some terms related with $\PP$ can be cancelled in the estimates of the remainder term.

By the Taylor expansion and (\ref{Q-ve}), we obtain
\begin{align}\label{CJ-taylor}
\CJ(Q^{\ve})=&\CJ(Q_0)+\ve\CH_{\nn}(Q_1)+\ve^2(\CH_{\nn}(Q_2)+\JJ_1)+\ve^3(\CH_{\nn}(Q_3)+\JJ_2)\nonumber\\
&+\ve^3\CH_{\nn}(Q_R)+\ve^4\CJ_R^{\ve},
\end{align}
where $\JJ_1$ and $\JJ_2$ are given by, respectively,
\begin{align*}
\JJ_1=&-\frac{b}{2}\CB(Q_1,Q_1)+c\CC(Q_0,Q_1,Q_1),\\
\JJ_2=&-b\CB(Q_1,Q_2)+2c\CC(Q_0,Q_1,Q_2)+\frac{c}{3}\CC(Q_1,Q_1,Q_1).
\end{align*}
The fourth-order term $\CJ^{\ve}_R$ is given by
\begin{align}\label{CJ-veR}
\CJ^{\ve}_R=&\JJ^{\ve}-b\CB(\widehat{Q}^{\ve},Q_R)+2c\CC(Q_R,\widehat{Q}^{\ve},Q_0)+c\ve\CC(Q_R,\widehat{Q}^{\ve},\widehat{Q}^{\ve})\nonumber\\
&-\frac{b}{2}\ve^2\CB(Q_R,Q_R)+c\ve^2\CC(Q_R,Q_R,Q_0+\ve\widehat{Q}^{\ve})+\frac{c}{3}\ve^5\CC(Q_R,Q_R,Q_R),
\end{align}
where $\widehat{Q}^{\ve}=Q_1+\ve Q_2+\ve^2Q_3$, and $\JJ^{\ve}$ is expressed by
\begin{align*}
\JJ^{\ve}=&-\frac{b}{2}\sum\limits_{\mbox{\tiny$\begin{array}{c}
1\leq i,j\leq 3\\
i+j\geq4\end{array}$}}\ve^{i+j-4}\CB(Q_i,Q_j)\\
&+\frac{c}{3}\sum_{\mbox{\tiny$\begin{array}{c}
i+j+k\geq4\\
\text{at least two of}~i,j,k~\text{are not zero}\end{array}$}}
\ve^{i+j+k-4}\CC(Q_i,Q_j,Q_k).
\end{align*}
Armed with (\ref{CJ-taylor}), the molecular field $\HH^{\ve}$ can be expressed as
\begin{align*}
\HH(Q^{\ve})=&-\frac{1}{\ve}\CJ(Q^{\ve})-\CL(Q^{\ve})\\
=&-(\CH_{\nn}(Q_1)+\CL(Q_0))-\ve(\CH_{\nn}(Q_2)+\CL(Q_1))-\ve^2(\CH_{\nn}(Q_3)+\CL(Q_2))\\
&-\ve^2\CH^{\ve}_{\nn}(Q_R)-\ve^3\CJ^{\ve}_R,
\end{align*}
where $\CH^{\ve}_{\nn}(Q_R)=\CH_{\nn}(Q_R)+\ve\CL(Q_R)$ and $\CJ^{\ve}_R$ is defined by (\ref{CJ-veR}).

In this section, we also define a characteristic function $\widetilde{\chi}$ (different from $\chi$ defined in (\ref{chi-fun})) with the domain on $\mathbb{Z}$, i.e.,
\begin{align}
\label{chi-fun-extension}
\widetilde{\chi}(n)=\begin{cases}
        1,\quad n=0,\\
        0, \quad n\in\mathbb{Z}\setminus \{0\},
\end{cases}
\end{align}
where $\widetilde{\chi}|_{\mathbb{N}}=\chi|_{\mathbb{N}}$ with $\mathbb{N}$ a set of natural numbers.
We are now in a position to write down the expansion of the system  (\ref{eq:Q})--(\ref{eq:free div}) with the small parameter $\ve$ and collect the terms (independent of the remainder term $(Q_R,\vv_R)$) with the same order of $\ve$.
Meanwhile, we also derive the remainder system of $(Q_R,\vv_R)$. More specifically, we have the following:

$\bullet$ {\it The $O(\ve^{-1})$ system}:
\begin{align}\label{O(-1)}
    \CJ(Q_0)=0.
\end{align}

$\bullet$ {\it Zeroth order term in $\ve$}:
\begin{align}
\chi(m)J \ddot{Q}_0+\mu_{1} \dot{Q}_{0} & = -\CH_{\nn}(Q_1)-\CL(Q_0)-\frac{\mu_{2}}{2} \mathbf{D}_{0}+\mu_{1}\left[\boldsymbol{\Omega}_{0}, Q_{0}\right], \label{ve 0-Q}\\
(\partial_t+\mathbf{v}_{0} \cdot \nabla )\mathbf{v}_{0} & = -\nabla p_{0}+\nabla \cdot\Big(\beta_{1} Q_{0}(Q_{0}: \mathbf{D}_{0})+\beta_{4} \mathbf{D}_{0}+\beta_{5} \mathbf{D}_{0} \cdot Q_{0}\nonumber \\
&\quad +\beta_{6} Q_{0} \cdot \mathbf{D}_{0}+\beta_{7}\left(\mathbf{D}_{0} \cdot Q_{0}^{2}+Q_{0}^{2} \cdot \mathbf{D}_{0}\right)\nonumber\\
&\quad +\frac{\mu_{2}}{2} \mathcal{N}_{0}+\mu_{1}[Q_0,\CN_0]+\sigma^{d}\left(Q_{0}, Q_{0}\right)\Big), \label{ve 0-v}\\
\nabla \cdot \mathbf{v}_{0} & = 0,\label{ve 0-div}
\end{align}
where $\chi(m)$ is defined by (\ref{chi-fun}), and $\dot{Q}_0, \ddot{Q}_0, \CN_0$ are given by
\begin{align*}
\dot{Q}_0=(\partial_t+\vv_0\cdot\nabla)Q_0,\quad\ddot{Q}_0=(\partial_t+\vv_0\cdot\nabla)\dot{Q}_0,\quad\CN_0=\dot{Q}_0-[\BOm_0,Q_0].
\end{align*}

 $\bullet$ {\it First order term in $\ve$}:
\begin{align}
  \chi(m)J \dot{\DP}_1+\mu_1\DP_1=&-\CH_\nn{(Q_2)}-\CL(Q_{1})-\frac{\mu_2}{2}\DD_{1}+\mu_1[\BOm_1,Q_0]+\TT_1, \label{ve 1-Q}\\
    (\partial_t+\vv_0\cdot\nabla){\mathbf{v}}_{1}  =& -\mathbf{v}_{1} \cdot \nabla \mathbf{v}_{0}-\nabla p_{1}+\nabla \cdot\Big(\beta _ { 1 } \big(Q_{0}\left(Q_{0}: \mathbf{D}_{1}\right) \nonumber\\
& +Q_{0}\left(Q_{1}: \mathbf{D}_{0}\right)+Q_{1}(Q_{0}: \mathbf{D}_{0})\big)+\beta_{4} \mathbf{D}_{1} \nonumber\\
& +\beta_{5}\left(\mathbf{D}_{0} \cdot Q_{1}+\mathbf{D}_{1} \cdot Q_{0}\right)+\beta_{6}\left(Q_{0} \cdot \mathbf{D}_{1}+Q_{1} \cdot \mathbf{D}_{0}\right) \nonumber\\
& +\beta_{7}\left(\mathbf{D}_{1} \cdot Q_{0}^{2}+Q_{0}^{2} \cdot \mathbf{D}_{1}+\mathbf{D}_{0} \cdot Q_{1} \cdot Q_{0}+\mathbf{D}_{0} \cdot Q_{0} \cdot Q_{1}\right.\nonumber \\
& \left.+Q_{1} \cdot Q_{0} \cdot \mathbf{D}_{0}+Q_{0} \cdot Q_{1} \cdot \mathbf{D}_{0}\right)+\frac{\mu_{2}}{2} {\mathcal{N}}_{1} \nonumber\\
& +\mu_{1}\left(\left[Q_{1}, \mathcal{N}_{0}\right]+\left[Q_{0}, {\mathcal{N}}_{1}\right]\right)+\sigma^{d}\left(Q_{1}, Q_{0}\right)+\sigma^{d}\left(Q_{0}, Q_{1}\right)\Big), \label{ve 1-v}\\
    \nabla\cdot\vv_1=&~0, \label{ve 1-div}
\end{align}
 where $\dot{Q}_1, \DP_1, \dot{\DP}_1, \CN_1$ and $\TT_1$ are expressed as, respectively,
 \begin{align*}
     \dot{Q}_1=&(\partial_t+\vv_0\cdot\nabla)Q_1,\quad\DP_1=\dot{Q}_1+\vv_1\cdot\nabla Q_0,\\\dot{\DP}_1=&(\partial_t+\vv_0\cdot\nabla)\DP_1,\quad{\CN}_1=\DP_1-[\BOm_1,Q_0]-[\BOm_0,Q_1],\\
     \TT_1=&\mu_1[\BOm_0,Q_1]-\JJ_1-\chi(m)J\vv_1\cdot\nabla\dot{Q}_0-\widetilde{\chi}(m-1)J\ddot{Q}_0.
 \end{align*}
The term related with $\widetilde{\chi}(m-1)$ is retained only if $m=1$, while it disappears in the other cases.

$\bullet$ {\it Second order term in $\ve$}:
\begin{align}
     \chi(m)J \dot{\DP}_2+\mu_1\DP_2=&-\CH_{\nn}{(Q_3)}-\CL(Q_2)-\frac{\mu_2}{2}\DD_2+\mu_1[\BOm_2,Q_0]+\TT_2, \label{ve 2-Q}\\
     \partial_t\vv_2=&~\nabla\cdot\bigg\{\sum_{i+j+k=2}\Big[\beta_1 Q_i(Q_j:Q_k)+\beta_7\left(\DD_i\cdot Q_j\cdot Q_k\right.\nonumber\\
&\left.+Q_i\cdot Q_j\cdot\DD_k\right)\Big]+\sum_{i+j=2}\Big(\beta_5\DD_i\cdot Q_j\nonumber\\&+\beta_6Q_i\cdot \DD_j
+\sigma^d(Q_i, Q_j)+\mu_1[Q_i, {\CN}_j]\Big)\nonumber\\
    &+\beta_4\DD_2+\frac{\mu_2}{2}{\CN}_2\bigg\}-\nabla p_2-\sum_{0\leq i\leq 2}\vv_i\cdot\nabla \vv_{2-i}, \label{ve 2-v}\\
    \nabla\cdot\vv_2=&~0.\label{ve 2-div}
\end{align}
where $\dot{Q}_2, \DP_2, \dot{\DP}_2, \CN_2$ and $\TT_2$ are expressed as, respectively,
\begin{align*}
   \dot{Q}_2=&(\partial_t+\vv_0\cdot\nabla)Q_2,\quad\DP_2=\dot{Q}_2+\vv_2\cdot\nabla Q_0,\quad\dot{\DP}_2=(\partial_t+\vv_0\cdot\nabla)\DP_2,\\
     {\CN}_2=&\DP_2+\vv_1\cdot\nabla Q_1-[\BOm_0,Q_2]-[\BOm_1,Q_1]-[\BOm_2,Q_0],\\
     \TT_2=&\mu_1\big([\BOm_0,Q_2]+[\BOm_1,Q_1]\big)-\JJ_2-\widetilde{\chi}(m-1)J(\dot{\DP}_1+\vv_1\cdot\nabla \dot{Q}_0)\\
     &-\chi(m)J\big((\partial_t+\vv_0\cdot\nabla)(\vv_1\cdot\nabla Q_1)+\vv_1\cdot\nabla\DP_1+\vv_2\cdot\nabla\dot{Q}_0\big)\\&-\widetilde{\chi}(m-2)J\ddot{Q}_0.
\end{align*}
Similarly, the term related with $\widetilde{\chi}(m-2)$ is retained only if $m=2$, while it disappears in the other cases.

\subsection{The remainder system}

We now derive the remainder system. Using the expansions (\ref{Q-ve})--(\ref{v-ve}), we obtain
\begin{align}\label{QR-vvR}
    Q_R=\ve^{-3}(Q^{\ve}-\widetilde{Q}),\quad \vv_R=\ve^{-3}(\vv^{\ve}-\tilde{\vv}),
\end{align}
where $(Q_R,\vv_R)$ depends on the small parameter $\ve$, and $\widetilde{Q}=\sum^3_{k=0}\ve^kQ_k, \tilde{\vv}=\sum^2_{k=0}\ve^k\vv_k$. Combining the system (\ref{eq:Q})--(\ref{eq:free div}) with (\ref{O(-1)})--(\ref{QR-vvR}), we have the following:

$\bullet$ {\it The remainder system in $\ve$}:
\begin{align}
   \ve^m J(\partial_t+\vv^\ve\cdot\nabla)\PP=&-\frac{1}{\ve}\CH^{\ve}_\nn(Q_R)-\frac{\mu_2}{2}\DD_R-\mu_1\left(\PP-[\BOm_R,Q_0]\right)+\FF_R,\label{eq-QR}\\
    \partial_t{\vv}_R=&-\nabla p_R+\nabla\cdot \Big\{\beta_1Q_0(Q_0:\DD_R)+\beta_4\DD_R+\beta_5\DD_R\cdot Q_0\nonumber\\
    &+\beta_6 Q_0\cdot\DD_R+\beta_7(\DD_R\cdot Q_0^2+Q^2_0\cdot \DD_R)\nonumber\\
    &+\frac{\mu_2}{2}(\PP-[\BOm_R,Q_0])+\mu_1\big[Q_0,(\PP-[\BOm_R,Q_0])\big]\Big\}\nonumber\\
    &+\nabla\cdot \GG_R+\GG_R',\label{eq-vR}\\
    \nabla\cdot\vv_R=&~0,\label{eq-vR-div}
\end{align}
where $\PP$ is defined by
\begin{align*}
 \PP=(\partial_t+\tilde{\vv}\cdot\nabla)Q_R+\vv_R\cdot\nabla (\widetilde{Q}+\ve^3Q_R).
\end{align*}

In the equation (\ref{eq-QR}), the term $\FF_R$ is given by
\begin{align}\label{FFR-remaider-term}
    \FF_R=\FF_1+\FF_2+\FF_3,
\end{align}
where $\FF_1$ is independent of $(\vv_R,Q_R)$,
\begin{align*}
\mathbf{F}_{1}  =& -\ve^{m-3} J\bigg(\sum_{i+m\geq 3}\ve^i\partial_{t}^{2} Q_{i}+\sum_{i+j+m\ge 3}\ve^{i+j}\left(2\vv_i\cdot\nabla \partial_t Q_j+\partial_t\vv_i\cdot\nabla Q_j\right)\\&+\sum_{i+j+k+m \geq 3}\varepsilon^{i+j+k} \mathbf{v}_{i} \cdot \nabla\left(\mathbf{v}_{j} \cdot \nabla Q_{k}\right)\bigg)-\JJ^\ve-\CL(Q_{3})\\& -\mu_{1}\bigg(\partial_{t} Q_{3}+\sum_{{i+j\ge 3}}\ve^{i+j-3}\Big(\mathbf{v}_{i} \cdot \nabla Q_{j}-[\BOm_i,Q_j]\Big)\bigg),
\end{align*}
and $\FF_2$ linearly depends on $(\vv_R,Q_R)$,
\begin{align*}
\mathbf{F}_{2}  =&\mu_{1}[\widetilde{\BOm},Q_R]+\left(b \CB(\widehat{Q}^{\varepsilon}, Q_{R})-2c \CC(Q_{R}, \widehat{Q}^{\varepsilon}, Q_{0})-{c} \varepsilon \CC(Q_{R}, \widehat{Q}^{\varepsilon}, \widehat{Q}^{\varepsilon})\right)\\
&-\ve^m J \mathbf{v}_{R} \cdot \nabla\big(\partial_{t} \widetilde{Q}+\tilde{\mathbf{v}} \cdot \nabla \widetilde{Q}\big)+\ve\mu_{1}[\BOm_R,\widehat{Q}^\ve],
\end{align*}
where $\widetilde{\BOm}=\sum^2_{k=0}\ve^k\BOm_k$ and $\widehat{Q}^{\ve}=Q_1+\ve Q_2+\ve^2Q_3$. Again, the term $\FF_3$ nonlinearly depends on $(\vv_R,Q_R)$,
\begin{align*}
\mathbf{F}_3 =&-\left(-\frac{b}{2} \varepsilon^2 \CB(Q_R, Q_R)+c \varepsilon^2 \CC(Q_R, Q_R, \widetilde{Q})+\frac{c}{3} \varepsilon^5 {\CC}\left(Q_R, Q_R, Q_R\right)\right)\\&+\ve^3\mu_1 [\BOm_R,Q_R]
.
\end{align*}

In the equation (\ref{eq-vR}), the term $\GG_R'$ takes the form
\[ \GG_R'=-\vv_1\cdot\nabla\vv_2-\vv_2\cdot\nabla(\vv_1+\ve\vv_2)-\vv_R\cdot\nabla\tilde{\vv}-\tilde{\vv}\cdot\nabla\vv_R-\ve^3\vv_R\cdot\nabla\vv_R.\]
Similarly, $\GG_R$ can be written as
\[\GG_R=\GG_1+\GG_2+\GG_3,\]
where $\GG_1$ is given by
\begin{align*}
\mathbf{G}_{1}= &  \mu_{1}\left(\sum_{i+j \geq 3} \varepsilon^{i+j-3}\left[Q_{i}, \partial_{t} Q_{j}\right]+\sum_{i+j+k \geq 3} \varepsilon^{i+j+k-3}\Big[Q_{i},\left(\mathbf{v}_{j} \cdot \nabla Q_{k}-[\BOm_{j}, Q_{k}]\right)\Big]\right)\\
&+\sum_{i+j+k \geq 3} \varepsilon^{i+j+k-3}\Big(\beta_{1} Q_{i}\left(Q_{j}: \mathbf{D}_{k}\right)+\beta_{7}\left(\mathbf{D}_{i} \cdot Q_{j} \cdot Q_{k}+Q_{i} \cdot Q_{j} \cdot \mathbf{D}_{k}\right)\Big) \\
& +\sum_{i+j \geq 3} \varepsilon^{i+j-3}\Big(\beta_{5} \mathbf{D}_{i} \cdot Q_{j}+\beta_{6} Q_{i} \cdot \mathbf{D}_{j}+\sigma^{d}\left(Q_{i}, Q_{j}\right)\Big)\\&+\frac{\mu_{2}}{2}\left(\partial_tQ_3+\sum_{i+j\ge 3}\left(\ve^{i+j-3}\vv_i\cdot \nabla Q_j-[\BOm_i,Q_j]\right)\right) ,
\end{align*}
and $\GG_2$, $\GG_3$ are given by
\begin{align*}
\mathbf{G}_2= &~ \beta_1\left(\widetilde{Q}(Q_R: \widetilde{\mathbf{D}})+Q_R(\widetilde{Q}: \widetilde{\mathbf{D}})+\varepsilon Q_0(\widehat{Q}^{\varepsilon}: \mathbf{D}_R)+\varepsilon \widehat{Q}^{\varepsilon}(\widetilde{Q}: \mathbf{D}_R)\right) \\
& +\beta_5\left(\widetilde{\mathbf{D}} \cdot Q_R+\varepsilon \mathbf{D}_R \cdot \widehat{Q}^{\varepsilon}\right)+\beta_6\left(\varepsilon \widehat{Q}^{\varepsilon} \cdot \mathbf{D}_R+Q_R \cdot \widetilde{\mathbf{D}}\right) \\
& +\beta_7\left(\widetilde{\mathbf{D}} \cdot Q_R \cdot \widetilde{Q}+\widetilde{\mathbf{D}} \cdot \widetilde{Q} \cdot Q_R+\varepsilon \mathbf{D}_R \cdot \widehat{Q}^{\varepsilon} \cdot \widetilde{Q}+\varepsilon \mathbf{D}_R \cdot Q_0 \cdot \widehat{Q}^{\varepsilon}\right) \\
& +\beta_7\left(\widetilde{Q} \cdot Q_R \cdot \widetilde{\mathbf{D}}+Q_R \cdot \widetilde{Q} \cdot \widetilde{\mathbf{D}}+\varepsilon \widehat{Q}^{\varepsilon} \cdot Q_0 \cdot \mathbf{D}_R+\varepsilon \widetilde{Q} \cdot \widehat{Q}^{\varepsilon} \cdot \mathbf{D}_R\right) \\
& -\frac{\mu_2}{2}\big(\ve[\BOm_R, \widehat{Q}^\ve]+[\widetilde{\BOm},Q_R]\big)+\mu_1\left[Q_R,(\partial_t \widetilde{Q}+\tilde{\mathbf{v}} \cdot \nabla \widetilde{Q}-[\widetilde{\BOm}, \widetilde{Q}])\right] \\
& -\mu_1\left(\left[\widetilde{Q},[\widetilde{\BOm}, Q_R]\right]-\left[\varepsilon \widehat{Q}^{\varepsilon},(\PP-[\BOm_R, Q_0])\right] +\left[\widetilde{Q},[\boldsymbol{\Omega}_R, \varepsilon \widehat{Q}^{\varepsilon}]\right]\right)\\
& +\sigma^d(\widetilde{Q}, Q_R)+\sigma^d(Q_R, \widetilde{Q}),\\
\mathbf{G}_3= &~ \varepsilon^3\bigg(\beta_1\left(\widetilde{Q}(Q_R: \mathbf{D}_R)+Q_R(\widetilde{Q}: \mathbf{D}_R)+Q_R(Q_R: \widetilde{\mathbf{D}})+\varepsilon^3 Q_R(Q_R: \mathbf{D}_R)\right) \\
& +\beta_7\left(\widetilde{\mathbf{D}} \cdot Q_R \cdot Q_R+\mathbf{D}_R \cdot \widetilde{Q} \cdot Q_R+\mathbf{D}_R \cdot Q_R \cdot \widetilde{Q}+\varepsilon^3 \mathbf{D}_R \cdot Q_R \cdot Q_R\right) \\
& +\beta_7\left(\widetilde{Q} \cdot Q_R \cdot \mathbf{D}_R+Q_R \cdot \widetilde{Q} \cdot \mathbf{D}_R+Q_R \cdot Q_R \cdot \widetilde{\mathbf{D}}+\varepsilon^3 Q_R \cdot Q_R \cdot \mathbf{D}_R\right) \\
& +\beta_5 \mathbf{D}_R \cdot Q_R+\beta_6 Q_R \cdot \mathbf{D}_R-\frac{\mu_2}{2}[\BOm_R,Q_R] -\mu_1\left[\widetilde{Q},\left[\BOm_R, Q_R\right]\right] \\
& +\mu_1\left[Q_R,\big(\PP-[\widetilde{\Omega}, Q_R]-[\Omega_R, \widetilde{Q}]-\ve^3[\BOm_R,Q_R]\big)\right] +\mu_1 \varepsilon^3\sigma^d\left(Q_R, Q_R\right)\bigg),
\end{align*}
where $\widetilde{\DD}=\sum^2_{k=0}\ve^k\DD_k$, and $\widetilde{Q}, \widetilde{\BOm}$ and $\widehat{Q}^{\ve}$ are just as we expressed above.

Under a unified case (including two small parameters), we have derived a series of equations for $(Q_k,\vv_k;Q_3)(0\leq k\leq2)$ by the Hilbert expansion, and the remainder system is also presented subsequently. The $O(\ve^{-1})$ system requires that $\CJ(Q_0)=0$, implying from Proposition \ref{critical-points} that $Q_0$ is the critical point of the bulk energy $f_b$, which will be taken as the uniaxial global minimum with the following form:
\begin{align}
    Q_0(t,\xx)=s\Big(\nn(t,\xx)\nn(t,\xx)-\frac{1}{3}\II\Big), \label{Q0}
\end{align}
for some $\nn(t, \xx)\in\BS$ and $s=s_1$.

The first task becomes how to solve $(Q_k,\vv_k;Q_3)(0\leq k\leq2)$ from the above system (\ref{ve 0-Q})--(\ref{ve 2-div}). It can be easily observed that the equations of order $O(\ve^k)(k=0,1,2)$ is not closed, since the system of the leading terms $Q_k(k=0,1,2)$ contains the non-leading terms $Q_{k+1}$. However, we will take advantage of the zero-eigenvalue subspace $\mathrm{Ker}\CH_{\nn}$ of the linearized operator $\CH_{\nn}$ to cancel the non-leading terms, and thus closing the system of the leading order.
The second task is how to estimate the remainder system (\ref{eq-QR})--(\ref{eq-vR-div}). In order to control nonlinear remainder terms, it is natural to introduce the following energy functionals  as follow:
\begin{align}
   {\fE}_m(t)=&~\frac{1}{2}\int_{\mathbb{R}^3}\left(\left|\mathbf{v}_R\right|^2+
\ve^mJ\left|\PP\right|^2+M\left|Q_R\right|^2+\varepsilon^{-1} \CH_{\mathbf{n}}^{\varepsilon}\left(Q_R\right): Q_R\right) \nonumber\\
&~~~+\varepsilon^2\left(\left|\nabla \mathbf{v}_R\right|^2+  \ve^m J\left|\partial_i \PP\right|^2+\varepsilon^{-1} \CH_\nn^{\varepsilon}\left(\partial_i Q_R\right): \partial_i Q_R\right) \nonumber\\
&~~~+\varepsilon^4\left(\left|\Delta \mathbf{v}_R\right|^2+ \ve^m J\left|\Delta\PP\right|^2+\varepsilon^{-1} \CH_\nn^{\varepsilon}\left(\Delta Q_R\right): \Delta Q_R\right) \mathrm{d} \mathbf{x}\nonumber\\
\eqdefa&~{\fE}_{m,0}(t)+{\fE}_{m,1}(t)+{\fE}_{m,2}(t),\label{energyE}\\
\fF(t) =& \int_{\mathbb{R}^3} \Big(|\nabla\vv_R|_{L^2}^2+\mu_1|{\mathbf{U}}_0|^2\Big)
+\ve^2\Big(\left|\nabla \partial
_i\mathbf{v}_R\right|^2+\mu_1|{\mathbf{U}}_1|^2\Big)\nonumber\\
&+\ve^4\Big(\left|\nabla \Delta\mathbf{v}_R\right|^2+\mu_1|
{\mathbf{U}}_2|^2\Big)\mathrm{d} \mathbf{x}\nonumber\\
\eqdefa&~\fF_0(t)+\fF_1(t)+\fF_2(t),\label{energyF}
\end{align}
where $\CH^{\ve}_{\nn}(Q_R)=\CH_{\nn}(Q_R)+\ve\CL(Q_R),\PP=(\partial_t+\tilde{\vv}\cdot\nabla)Q_R+\vv^\ve\cdot\nabla Q_R$, and $M\geq 1$ is a sufficiently large constant to be determined later, and $\mathbf{U}_k(k=0,1,2)$ are defined as, respectively,
\begin{align}
{\mathbf{U}}_0=&~\PP-[\BOm_R,Q_0]+\frac{\mu_2}{2\mu_1}\DD_R, \label{UU0}\\
{\mathbf{U}}_1=&~\partial_i \PP-\left[\partial_i \boldsymbol{\Omega}_R, Q_0\right]+\frac{\mu_2}{2 \mu_1} \partial_i \mathbf{D}_R, \label{UU1}\\
{\mathbf{U}}_2=&~\Delta \PP-\left[\Delta\boldsymbol{\Omega}_R, Q_0\right]+\frac{\mu_2}{2 \mu_1} \Delta \mathbf{D}_R. \label{UU2}
\end{align}
We would remark that the terms $\FF_R, \GG_R$ and $\GG'_R$ in the equations (\ref{eq-QR})--(\ref{eq-vR}) are all {\it good terms}, which are easily estimated with the aid of the functionals (\ref{energyE})--(\ref{energyF}) (see Lemmas \ref{norm:FR}--\ref{norm:vR} and Lemma \ref{nu:norm:FR}). In the next two sections, we will accomplish the above two tasks under two different cases.

\section{From the Qian--Sheng model to the full inertial Ericksen--Leslie model}\label{m0-section}

This section will be devoted to rigorously justify the uniaxial limit from the inertial Qian--Sheng model to the full inertial Ericksen--Leslie model. In this section, only a small elastic parameter $\ve$ is considered, while the inertial density $J$ is a given constant ({\it not a small parameter}).  Since $\chi(m)=1$ when $m=0$, the Qian--Sheng system (\ref{eq:Q})--(\ref{eq:free div}) is a hyperbolic system with the small parameter $\ve$. The uniaxial limit model is just the hyperbolic Ericksen--Leslie system. In a word, in the framework of smooth solutions, we will prove that when the elastic coefficients tend to zero (called the uniaxial limit), the solution to the Qian--Sheng model converges to the solution to the full inertial Ericksen--Leslie model.

The evolution system of $(\nn,\vv_0)$ is determined by the $O(1)$ system (\ref{ve 0-Q})--(\ref{ve 0-div}). By letting $Q_0$ take (\ref{Q0}), the uniaxial limit model, i.e., the full inertial Ericksen--Leslie system can be deduced, which is exactly what has been done in \cite{LW}.

\begin{proposition}[see \cite{LW}]\label{0toEL}
   Under the condition of $m=0$, i.e., $\chi(m)=1$,
   if $(Q_0, \vv_0)$ is a smooth solution of the $O(1)$ system {\em (\ref{ve 0-Q})--(\ref{ve 0-div})},  then $(\nn, \vv_0)$ must be a solution of the full inertial Ericksen--Leslie system {\em (\ref{EL-vv})--(\ref{EL-div})},  where the coefficients are determined by {\rm(\ref{coefficients})}.
\end{proposition}

\subsection{Existence of the Hilbert expansion}\label{J-existence-He}

Under the condition of $m=0$, i.e., $\chi(m)=1$, the main task of this subsection is to show the existence of the Hilbert expansion, that is, how to solve $(Q_k,\vv_k)(k=1,2)$ and $Q_3$ from the system (\ref{ve 1-Q})--(\ref{ve 2-div}) and derive the corresponding estimates.  Specifically, we show the following proposition.

\begin{proposition}{\label{prop Q(1)}}
Let $(\nn, \vv_0)$ be a smooth solution to the full inertial Ericksen--Leslie system {\em (\ref{EL-vv})--(\ref{EL-div})} derived from the zeroth-order inertial equations {\em (\ref{ve 0-Q})--(\ref{ve 0-div})} on $[0, T]$, satisfying
    \[(\vv_0,\nabla \nn,\partial_t \nn)\in L^\infty([0, T];H^{\ell}),\quad {\ell \ge 20.} \]
    Then, there exists the solutions $(Q_k, \vv_k)(k=0,1,2)$ and $Q_3\in ({\rm Ker}\CH_\nn)^\perp$ of the system {\em (\ref{ve 1-Q})--(\ref{ve 2-div})} satisfying
    \begin{equation}
    \begin{aligned}
        &(\vv_k,\nabla Q_k,\partial_tQ_k)\in L^\infty([0, T];H^{\ell-4k}) (k=0,1,2),\quad Q_3\in L^\infty([0, T];H^{\ell-11}).
    \end{aligned}\label{def Q1}
    \end{equation}
\end{proposition}

As mentioned in Section \ref{Hilbert-expansion}, the equations (\ref{ve 1-Q})--(\ref{ve 1-div}) is not a closed system, due to the appearance of the unknown term related with $Q_2$. The closure of the system can be obtained by projecting the equations (\ref{ve 1-Q})--(\ref{ve 1-div}) into the subspace ${\rm Ker}\CH_\nn$. Let us decompose $Q_1$ in the light of ${\rm Ker}\CH_\nn$, i.e., $Q_1=Q^\top_1+Q_1^\perp$ with $Q_1^\top\in {\rm Ker}\CH_\nn$ and $Q_1^\perp\in({\rm Ker}\CH_\nn)^\perp$. We assume that there exists a smooth solution $(Q_0,\vv)$, where $Q_0=Q_0(\nn)$ is a function of $\nn\in\BS$.  Before proving Proposition \ref{prop Q(1)}, we present a lemma about the terms $\TT_1, \DP_1$ and the material derivatives of $Q_1$ and $\DP_1$. In this section, $L(\cdot)$ represents the linear function with the coefficients belonging to $L^\infty([0, T];H^{\ell-1})$ and $R\in L^\infty([0, T]; H^{{{\ell-3}}})$ some function depending only on $\nn, \vv_0$ and $Q_1^\perp$.

\begin{lemma}\label{MP-Q1}
 Let two projection operators $ \MP^{in}$ and $\MP^{out}$ be defined by \eqref{projection_in1} and \eqref{projection_out1}, respectively. Then it follows that
   \begin{align*}
    \MP^{out}(\Dot{Q}_1)=&~L(Q_1^\top)+R, \quad
\MP^{in}(\dot{Q}_1)={\dot{Q}_1^ \top}+L({Q}_1^\top)+R,\\
\MP^{out}({\DP}_1)=&~L(Q_1^\top,\vv_1)+R, \quad\MP^{in}({\DP}_1)=\dot {Q}_1^\top+L(Q_1^\top,\vv_1)+R,
\\
\MP^{in}(\dot{\DP}_1)=&~{\dot{\DP}_1^ \top}+L(\DP_1^\top,{Q}^\top_1,\vv_1)+R,\quad\MP^{in}(\TT_1)=L(Q_1^\top,\vv_1)+R,
\end{align*}
where the notations $\dot{f}=(\partial_t+\vv_0\cdot\nabla)f,~ \dot{f}^\top=(\partial_t+\vv_0\cdot\nabla)f^{\top}$, and the terms $\DP_1$ and $\TT_1$ are given by, respectively,
\begin{align}
\DP_1=&~\dot{Q}_1+\vv_1\cdot\nabla Q_0,\label{DP_1}\\
\TT_1=&~\mu_1[\BOm_0,Q_1]+\frac{b}{2}\CB(Q_1,Q_1)-c\CC(Q_1,Q_1,Q_0)-J\vv_1\cdot\nabla\dot{Q}_0.\label{TT_1}
\end{align}
\end{lemma}

\begin{proof}
By the definition of $\MP^{in}$,
 for any $Q\in\mathbb{S}_0^3$, we have
 \begin{align}
    (\partial_t+\vv_0\cdot\nabla)\MP^{in}(Q)=&~(\partial_t+\vv_0\cdot\nabla)\Big((\nn\nn\cdot Q+Q\cdot\nn\nn)-2(Q:\nn\nn)\nn\nn\Big)\nonumber\\
     =&~\Big((\nn\nn\cdot \dot{Q}+\dot{Q}\cdot\nn\nn)-2(\dot{Q}:\nn\nn)\nn\nn\Big)\nonumber\\
     &+\Dot{\overline{\nn\nn}}\cdot Q+Q\cdot \Dot{\overline{\nn\nn}}-2(Q:\Dot{\overline{\nn\nn}})\nn\nn-2(Q:\nn\nn)\Dot{\overline{\nn\nn}}\nonumber\\
     =&~\MP^{in}(\dot{Q})+L(Q),\label{part-in-Q}
 \end{align}
where $\Dot{\overline{\nn\nn}}=(\partial_t+\vv_0\cdot\nabla)(\nn\nn)$. Armed with the identity mapping $i_d=\MP^{in}+\MP^{out}$, we deduce
\begin{align}
   \nonumber (\partial_t+\vv_0\cdot\nabla)\MP^{out}({Q})=&\dot{Q}-(\partial_t+\vv_0\cdot\nabla)\MP^{in}(Q)\nonumber\\
     =&\dot{Q}-\MP^{in}(\dot{Q})-L(Q)\nonumber\\
     =&\MP^{out}(\dot{Q})-L(Q).\label{part-out-Q}
\end{align}
Taking $Q=Q_1$ in (\ref{part-in-Q}) and (\ref{part-out-Q}), respectively,  we obtain
\begin{align*}
    \MP^{in}(\dot{Q}_1)=&~(\partial_t+\vv_0\cdot\nabla)\MP^{in}(Q_1)-L(Q_1^\top+Q_1^\perp)\\
    =&~\dot{Q}_1^\top+L(Q_1^\top)+R,\\
     \MP^{out}(\dot{Q}_1)=&~(\partial_t+\vv_0\cdot\nabla)\MP^{out}(Q_1)+L(Q_1^\top+Q_1^\perp)\\
     =&~L(Q_1^\top)+R.
\end{align*}
Acting the projections $\MP^{in}$ and $\MP^{out}$ on $\DP_1=\dot{Q}_1+\vv_1\cdot\nabla Q_0$, respectively, and using the above two equalities, we have
\begin{align*}
     \MP^{in}(\DP_1)=&\MP^{in}(\dot{Q}_1)+\MP^{in}(\vv_1\cdot\nabla Q_0)=\dot{Q}_1^\top+L(Q_1^\top,\vv_1)+R,\\
     \MP^{out}(\DP_1)=&\MP^{out}(\dot{Q}_1)+\MP^{out}(\vv_1\cdot\nabla Q_0)=L(Q_1^\top,\vv_1)+R.
\end{align*}
Similarly, taking $Q=\DP_1$ in (\ref{part-in-Q}) yields
\begin{align*}
     \MP^{in}(\dot{\DP}_1)=&~(\partial_t+\vv_0\cdot\nabla)\MP^{in}(\DP_1)-L(\DP_1^\top+\DP_1^\perp)\\
     =&~\dot{\DP}_1^\top+L(\DP_1^\top,Q_1^\top,\vv_1)+R.
\end{align*}
Using the expression of $\TT_1$ and the decomposition $Q_1=Q^\top_1+Q_1^\perp$, there holds
\begin{align*}
    \TT_1
    =\frac{b}{2}\CB(Q_1^\top,Q_1^\top)-c\CC(Q_1^\top,Q_1^\top,Q_0)+L(Q_1^\top,\vv_1)+R.
\end{align*}
Further, a simple calculation leads to
\begin{align*}
 \frac{b}{2}\CB(Q_1^\top,Q_1^\top)-c\CC(Q_1^\top,Q_1^\top,Q_0)\in ({\rm Ker}\CH_\nn)^\perp.
\end{align*}
We thus deduce that $\MP^{in }(\TT_1)=L(Q_1^\top,\vv_1)+R$.
 \end{proof}

The following lemma provides a sufficient and necessary condition on the Leslie coefficients, which guarantees that the energy is dissipated.
\begin{lemma}[see \cite{WZZ1}]\label{nn, D}
The dissipation relation \[\hat{\beta}_{1}|\mathbf{n n}: \mathbf{D}|^{2}+\hat{\beta}_{2}|\mathbf{D}|^{2}+\hat{\beta}_{3}|\mathbf{n} \cdot \mathbf{D}|^{2} \geq 0\] holds for any $\DD\in \mathbb{S}^3_0$ and unit vector $\nn$, if and only if
\begin{align}\label{beta rela} \hat{\beta}_{2} \geq 0,  \quad 2 \hat{\beta}_{2}+\hat{\beta}_{3} \geq 0,  \quad \frac{3}{2} \hat{\beta}_{2}+\hat{\beta}_{3}+\hat{\beta}_{1} \geq 0.
\end{align}
\end{lemma}

We now turn into the proof of Proposition \ref{prop Q(1)}.

 \begin{proof}[Proof of Proposition \ref{prop Q(1)}]
Assume that $(\nn,\vv_0)$ is a solution to the full inertial Ericksen--Leslie system \eqref{EL-vv}--\eqref{EL-div} on $[0,T]$ and satisfies
\begin{align*}
    (\vv_0,\nabla \nn,\partial_t \nn)\in L^\infty([0, T];H^{\ell}),\quad {\ell \ge 20.}
\end{align*}
Since $Q_0=Q_0(\nn)$ is a function of the unit vector $\nn$ and takes the form (\ref{Q0}), it holds that $Q_0\in L^{\infty}([0,T],H^{\ell+1})$.

We solve $Q_1^\perp$ from the equation $(\ref{ve 0-Q})$, i.e.,
\begin{align*}
    Q_1^\perp=\CH_\nn^{-1}\Big(-J\ddot{Q}_0-\mu_1\dot{Q}_0-\CL(Q_0)-\frac{\mu_2}{2}\DD_0+\mu_1[\BOm_0,Q_0]\Big)\in L^\infty([0,T];H^{\ell-1}).
\end{align*}
Here, the inverse $\CH^{-1}_{\nn}$ is well-defined within $(\mathrm{Ker}\CH_{\nn})^{\perp}$ by reason of Proposition \ref{linearized-oper-prop}. Therefore, the existence of $(Q_1,\vv_1)$ can be reduced to solving $(Q^{\top}_1,\vv_1)$. The key observation is that $(\DP_1^\top, Q^{\top}_1,\vv_1)$ fulfils a linear dissipative system, where $\DP_1^\top$ corresponds to $\dot{Q}_1^\top$.

Taking the projection $\MP^{in}$ on both sides of ($\ref{ve 1-Q}$), and noting $\CH_{\nn}(Q_2)\in (\mathrm{Ker}\CH_{\nn})^{\perp}$, we derive from Lemma \ref{MP-Q1} the closed linear system of $(\DP_1^\top, Q^{\top}_1,\vv_1)$,
\begin{align}
   J\dot{\DP}^\top_1+\mu_1{\DP}_1^\top=&~\MP^{in}\Big(-\CL(Q_1^\top)-\frac{\mu_2}{2}\DD_1+\mu_1[\BOm_1,Q_0] \Big)+L(\DP_1^\top,Q_1^\top,\vv_1)+R,\label{L Q1}\\
    (\partial_t+\vv_0\cdot\nabla){\vv}_1=&-\vv_1\cdot\nabla \vv_0-\nabla p_1+\nabla\cdot\Big(\beta_1Q_0(Q_0:\DD_1)+\beta_4\DD_1\nonumber\\
    &+\beta_5\DD_1\cdot Q_0+\beta_6Q_0\cdot Q_1+\beta_7(\DD_1\cdot Q_0^2+Q_0^2\cdot \DD_1)\nonumber\\
   & +\frac{\mu_2}{2}({\DP}_1^\top-[\BOm_1,Q_0])+\mu_1\big[Q_0,({\DP}_1^\top-[\BOm_1,Q_0])\big]\nonumber\\
    &+\sigma^d({Q}_1^\top,Q_0)+\sigma^d(Q_0,Q_1^\top)+L(Q_1^\top,\vv_1)+R\Big),\label{L v1}\\
    \nabla\cdot \vv_1=&~0.\label{L 1div}
\end{align}
To obtain the unique solvability of the linear system(\ref{L Q1})--(\ref{L 1div}),  we establish a priori estimate for the energy
\begin{align*}
    \CE_{\ell}(t)\eqdefa&\sum_{k=0}^{\ell-4}\frac{1}{2}\Big(\Vert\partial_i^k\vv_1\Vert^2_{L^2}+\langle \partial_i^k Q_1^{\top}, \CL(\partial_i^k Q_1^{\top})\rangle+J\|\partial_i^k\DP_1^\top\|_{L^2}^2\Big)+\frac{1}{2}\Vert Q_1^\top\Vert^2_{L^2},\\
    \CF_\ell(t)\eqdefa&\sum_{k=0}^{\ell-4}\Big(\|\nabla\partial_i^k\vv_1\|_{L^2}^2+\mu_1\Big\|\partial_i^k\DP_1^\top-[\partial_i^k\BOm_1,Q_0]+\frac{\mu_2}{2\mu_1}\partial_i^k\DD_1\Big\|_{L^2}^2\Big).
\end{align*}
Specifically, we show that there exists positive constants $C$ and $c_1$, such that
\begin{align*}
    \frac{\ud }{\ud t}\CE_{\ell}(t)+c_1\CF_\ell(t)\leq C\left({\CE}_\ell(t)+\|R(t)\|^2_{H^{\ell-3}}\right).
\end{align*}
Here, $(\DP_1^\top, Q^{\top}_1,\vv_1)$ satisfies $(\vv_1,\nabla Q^{\top}_1,\DP_1^\top)\in L^{\infty}([0,T];H^{\ell-4})$, which implies
\[\begin{aligned}
    \|\partial_tQ_1^\top\|_{H^{\ell-4}}=&~\|\DP_1^\top-\vv_0\cdot\nabla Q_1^\top+L(\vv_1,Q_1^\top)+R\|_{H^{\ell-4}}\\\leq& ~C\|(\vv_1,Q_1^\top,\nabla Q_1^\top,\DP_1^\top,R)\|_{H^{\ell-4}}\in L^\infty[0,T],
\end{aligned}\]
that is, $\partial_tQ_1^\top\in L^\infty([0,T];H^{\ell-4})$.

It suffices to show the case of $k=0$ because of similar arguments for the general case.
The associated energy functionals are given by
\begin{align*}
\CE(t)=&~\frac{1}{2}\left(\left\|\mathbf{v}_1\right\|_{L^2}^2+ J\left\|\DP_1^\top\right\|_{L^2}^2+\left\|Q_1^\top\right\|_{L^2}^2+\langle\CL(Q_1^\top),Q_1^\top\rangle\right),\\
\CF(t)=&~\|\nabla\vv_1\|_{L^2}^2+\mu_1\Big\|\DP_1^\top-[\BOm_1,Q_0]+\frac{\mu_2}{2\mu_1}\DD_1\Big\|_{L^2}^2.
\end{align*}

To begin with, from Lemma \ref{MP-Q1} and $\nabla\cdot\vv_0=0$, we infer that
\begin{align}\label{Q1top-L2}
    \frac{1}{2}\frac{\ud}{\ud t}\|Q_1^\top\|_{L^2}^2=&~\langle\partial_t{Q}_1^\top,Q_1^\top\rangle=\langle\dot{Q}_1^\top,Q^{\top}_1\rangle\nonumber\\
    =&~\big\langle\DP_1^\top-L(Q_1^\top,\vv_1)-R,Q_1^\top\big\rangle\nonumber\\
    \leq&~ C  (\CE+\|R\|_{L^2}^2).
\end{align}
Using the system (\ref{L Q1})--(\ref{L 1div}), the energy rate $\frac{1}{2}\frac{\ud}{\ud t}(\|\vv_1\|_{L^2}^2+J\|\DP_1^\top\|_{L^2}^2)$ can be estimated as
\begin{align} \label{DP1-Q1-top-v1}J\langle&\partial_t{\DP}_1^\top,\DP_1^\top\rangle+\langle\partial_t{\vv}_1,\vv_1\rangle\nonumber\\
&\leq -\langle\beta_1Q_0(Q_0:\DD_1)+\beta_4\DD_1+\beta_5\DD_1\cdot Q_0+\beta_6Q_0\cdot \DD_1,\nabla \vv_1\rangle\nonumber\\
    &\quad-\langle\beta_7(\DD_1\cdot Q_0^2+Q_0^2\cdot\DD_1),\nabla\vv_1\rangle-\frac{\mu_2}{2}\langle\DP_1^\top-[\BOm_1,Q_0],\nabla\vv_1\rangle\nonumber\\
    &\quad-\mu_1\Big\langle\big[Q_0,(\DP_1^\top-[\BOm_1,Q_0])\big],\nabla\vv_1\Big\rangle-\frac{\mu_2}{2}\langle\DD_1,\DP_1^\top\rangle\nonumber\\
   &\quad-\mu_1\langle\DP_1^\top-[\BOm_1,Q_0],\DP_1^\top\rangle-\langle\CL(Q^\top_1),\DP_1^\top\rangle+C(\|R\|_{H^1}^2+{\CE}+{\CE}^{\frac{1}{2}}{\CF}^{\frac{1}{2}})\nonumber\\
&\eqdefa \sum^7_{k=1}I_k+C(\|R\|_{H^1}^2+{\CE}+{\CE}^{\frac{1}{2}}{\CF}^{\frac{1}{2}}).
\end{align}
We estimate $I_k(k=1,\cdots,7)$ term by term as follows.
Using Lemma \ref{v*Q,L(Q)} and integrating by parts, the term $I_7$ is calculated as
\begin{align*}
    -\langle\CL(Q_1^\top),\DP_1^\top\rangle=&-\Big\langle\CL(Q_1^\top), \partial_t Q_1^\top+\vv_0\cdot\nabla{Q}_1^\top+L(Q_1^\top,\vv_1)+R\Big\rangle \\\leq&-\frac{1}{2}\frac{\ud}{\ud t}\langle\CL(Q_1^\top),Q_1^\top\rangle+C(\|R\|_{H^1}^2+{\CE}+{\CE}^{\frac{1}{2}}{\CF}^{\frac{1}{2}}).
\end{align*}
For the terms $I_1$ and $I_2$, from the relation $\mu_2=\beta_6-\beta_5$ and $Q_0=s(\nn\nn-\frac{1}{3}\II)$, we infer that
\begin{align}
I_1+I_2= & -\left\langle\beta_1 Q_0\left(Q_0: \mathbf{D}_1\right)+\beta_4 \mathbf{D}_1+\frac{\beta_5+\beta_6}{2}\left(Q_0 \cdot \mathbf{D}_1+\mathbf{D}_1 \cdot Q_0\right), \mathbf{D}_1\right\rangle \nonumber\\ & -\left\langle\beta_7\left(\mathbf{D}_1 \cdot Q_0^2+Q_0^2 \cdot \mathbf{D}_1\right), \mathbf{D}_1\right\rangle\nonumber \\ & +\left\langle\left(\frac{\beta_5+\beta_6}{2}-\beta_5\right) \mathbf{D}_1 \cdot Q_0+\left(\frac{\beta_5+\beta_6}{2}-\beta_6\right) Q_0 \cdot \mathbf{D}_1, \mathbf{D}_1+\boldsymbol{\Omega}_1\right\rangle \nonumber\\ = & -\beta_1 s^2\left\|\mathbf{n n}: \mathbf{D}_1\right\|_{L^2}^2-\left(\beta_4-\frac{s\left(\beta_5+\beta_6\right)}{3}+\frac{2}{9} \beta_7 s^2\right)\left\|\mathbf{D}_1\right\|_{L^2}^2\nonumber \\ & -\left(s\left(\beta_5+\beta_6\right)+\frac{2}{3} \beta_7 s^2\right)\left\|\mathbf{n} \cdot \mathbf{D}_1\right\|_{L^2}^2+\underbrace{\frac{\mu_2}{2}\left\langle\left[\mathbf{D}_1, Q_0\right], \boldsymbol{\Omega}_1\right\rangle}_{{I}_1^{\prime}} .
\label{I1+I2}
\end{align}
Armed with the symmetry of $[\BOm_1,Q_0]$, we get
\begin{align*}
I_1^{\prime}+I_3+I_5 =&~\frac{\mu_2}{2}\left\langle\left[\mathbf{D}_1, Q_0\right], \boldsymbol{\Omega}_1\right\rangle-\frac{\mu_2}{2}\big\langle\DP_1^\top-\left[\boldsymbol{\Omega}_1, Q_0\right], \mathbf{D}_1\big\rangle-\frac{\mu_2}{2}\langle\mathbf{D}_1,\DP_1^\top\rangle \\
=&-\mu_2\langle\DP_1^\top-\left[\boldsymbol{\Omega}_1, Q_0\right], \mathbf{D}_1\rangle .
\end{align*}
By a simple calculation, we obtain
\begin{align*}
    I_4+I_6=&-\mu_1\big\langle\big[Q_0,(\DP_1^\top-[\BOm_1,Q_0])\big],\BOm_1\big\rangle-\mu_1\big\langle\DP_1^\top-[\BOm_1,Q_0],\DP_1^\top\big\rangle\\
    =&-\mu_1\big\langle[Q_0,\BOm_1],\DP_1^\top-[\BOm_1,Q_0]\big\rangle-\mu_1\big\langle\DP_1^\top-[\BOm_1,Q_0],\DP_1^\top\big\rangle\\
    =&-\mu_1\|\DP_1^\top-[\BOm_1,Q_0]\|_{L^2}^2
\end{align*}
and
\begin{align*}
    I_1'+I_3+I_4+I_5+I_6=-\mu_1\|\DP_1^\top-[\BOm_1,Q_0]+\frac{\mu_2}{2\mu_1}\DD_1\|_{L^2}^2+\frac{\mu_2^2}{4\mu_1}\|\DD_1\|_{L^2}^2.
\end{align*}
Then, combining the above equalities yields
\begin{align}
   \sum^6_{k=1}I_k=&-{\mu_1}\Big\|\DP_1^\top-[\BOm_1,Q_0]+\frac{\mu_2}{2\mu_1}\DD_1\Big\|^2_{L^2}-{4\delta}\|\DD_1\|^2_{L^2}\nonumber\\
    &-\Big(\overline{\beta}_{1} \left\|\mathbf{n n}: \mathbf{D}_1\right\|_{L^{2}}^{2}+\overline{\beta}_2\left\|\mathbf{D}_1\right\|_{L^{2}}^{2}  +\overline{\beta}_3\left\|\mathbf{n} \cdot \mathbf{D}_1\right\|_{L^{2}}^{2}\Big)
  ,\label{c1}
\end{align}
where the coefficients $\overline{\beta}_i(i=1, 2, 3)$ are given by
\begin{align}\label{relation overlinebeta}
\left\{\begin{array}{l}
\overline{\beta}_{1}= \beta_{1} s^{2},  \quad \overline{\beta}_{2}  = \beta_{4}-4\delta-\frac{s\left(\beta_{5}+\beta_{6}\right)}{3}+\frac{2}{9} \beta_{7} s^{2}-\frac{\mu_{2}^{2}}{4 \mu_{1}},  \\
\overline{\beta}_{3}  = s\left(\beta_{5}+\beta_{6}\right)+\frac{2}{3} \beta_{7} s^{2},
\end{array}\right.
\end{align}
and $\delta>0$ is small enough so that $\overline{\beta}_i(i=1, 2, 3)$ satisfy the relation (\ref{beta rela}). Therefore,  from (\ref{Q1top-L2})--(\ref{DP1-Q1-top-v1}) and the above calculations, we deduce the following estimate:
\begin{align*}
    &\frac{\ud}{\ud t}{\CE(t)}+2\delta\|\nabla\vv_1\|_{L^2}^2+\mu_1\Big\|\DP_1^\top-[\BOm_1,Q_0]+\frac{\mu_2}{2\mu_1}\DD_1\Big\|_{L^2}^2\\
    &\quad\leq C(\|R\|_{H^1}^2+{\CE}+{\CE}^{\frac{1}{2}}{\CF}^{\frac{1}{2}}),
\end{align*}
which implies that
\begin{align*}
\frac{\ud}{\ud t}{\CE(t)}+c_1{\CF(t)}\leq C(\CE(t)+\|R\|_{H^1}^2)\nonumber,
\end{align*}
where the constant $c_1=\min\{\delta,1\}>0$.

Thus, the solution $(\vv_1,Q_1)$ can be uniquely determined. Further, it follows that
\[\DP_1=(\partial_t+\vv_0\cdot\nabla)Q_1+\vv_1\cdot\nabla Q_0 \in L^\infty([0,T];H^{\ell-4}),\quad\dot{\DP}_1\in L^\infty([0,T];H^{\ell-5}).\]
By the equation (\ref{ve 1-Q}), $Q_2^\perp$ is determined by
\begin{align*}
    Q_2^\perp=\CH_\nn^{-1}\Big( \dot{\DP}_1+\mu_1\DP_1+\CL(Q_{1})+\frac{\mu_2}{2}\DD_{1}-\mu_1[\BOm_1,Q_0]-\TT_1\Big) \in L^\infty([0,T];H^{\ell-5}).
\end{align*}
In a similar argument, we can show the existence of $(Q_2,\vv_2)$ and solve $Q_3$ by (\ref{ve 2-Q}) due to $Q_3\in ({\rm Ker}\CH_\nn)^\perp$. Here we omit the details.

\end{proof}

\subsection{Uniform estimates for the remainder system}

The aim of this subsection is to derive the uniform estimate for the remainder system (\ref{eq-QR})--(\ref{eq-vR-div}) with $m=0$. Since the singular term $\ve^{-1}\CH_{\nn}(Q_R)$ in $\ve$ is contained in the remainder system, it is crucial to construct a suitable energy functional, which is defined by (\ref{energyE}).  It is worth emphasizing that the term $M|Q_R|^2$ in (\ref{energyE}) will ensure the positive definiteness of the energy functional, by modulating the positive constant $M$ to be determined later.

For notational simplicity, we use $\fE(t)$ to denote $\fE_0(t)$ in (\ref{energyE}). Based on Proposition \ref{prop Q(1)}, the solutions $(Q_k,\vv_k)(k=0,1,2)$ and $Q_3$ will be handled as known functions. We denote by $C$ a positive constant depending on $\sum_{k=0}^2\|\vv_k\|_{L^\infty([0,T],H^{\ell-4k})}$ and $\sum_{k=0}^3\|Q_k\|_{L^\infty([0,T],H^{\ell+1-4k})}$, and independent of $\ve$ and $M$. The a priori estimate for the remainder $(\DP_R,Q_R,\vv_R)$ is stated as follows.

\begin{proposition}\label{proposition-PP}
Assume that $(\DP_R,Q_R,\vv_R)$ is a smooth solution to the remainder system {\rm (\ref{eq-QR})--(\ref{eq-vR-div})} {\rm (}with $m=0${\rm)} on $[0,T]$. Then for any $t\in[0,T]$, there exists a constant $c_1>0$ such that
\begin{align*}
{\fE}(t)&-3\fE(0)+{c_1}\int_0^t\fF(\tau)\ud \tau\\&\leq CM^2\int_0^t\Big(\big(1+{\fE}(\tau)+\ve^2{\fE}^2(\tau)\big)+\fF(\tau)\big(\ve+\ve{\fE}^{\frac{1}{2}}(\tau)+\ve^4{\fE}(\tau)\big)\Big)\ud \tau.
\end{align*}

\end{proposition}

To prove the Proposition \ref{proposition-PP}, we need the following lemmas.
We first present the following inequality:
\begin{equation}
\|fg\|_{H^k}\leq C\|f\|_{H^2}\|g\|_{H^k},~k=0,1,2,\label{fg Hk}\end{equation}
which will be frequently used.
Armed with (\ref{fg Hk}) and the definitions of $\fE(t)$ and $\fF(t)$, we immediately obtain the following lemma.

\begin{lemma} \label{norm:QR,vR}
There exists a positive constant $C$, such that
\begin{align*}
 \|(\PP,\ve\nabla\PP,\ve^2\nabla^2\PP)\|_{L^2}+\|(\vv_R,\ve\nabla\vv_R,\ve^2\nabla^2\vv_R)\|_{L^2}\leq C{\fE}^{\frac{1}{2}}&,\\
    \|Q_R\|_{H^1}+\|(\ve\nabla^2Q_R,\ve^2\nabla^3Q_R)\|_{L^2}\leq C{\fE}^{\frac{1}{2}}&,\\
    \|(\nabla\vv_R,\ve\nabla^2\vv_R,\ve^2\nabla^3\vv_R)\|_{L^2}\leq C\fF^{\frac{1}{2}}&.
\end{align*}
Further, for $\dot{Q}_R=(\partial_t+\tilde{\vv}\cdot\nabla) Q_R$, it follows that
\begin{align*}
   \|(\dot{Q}_R,\ve\nabla\dot{Q}_R,\ve^2\nabla^2\dot{Q}_R)\|_{L^2}\leq C(\|\ve^k\PP\|_{H^k}+\|\ve^k\vv_R\|_{H^k}\|\nabla Q^\ve\|_{H^2})\leq C({\fE}^{\frac{1}{2}}+\ve{\fE}).
\end{align*}
\end{lemma}

\begin{lemma}\label{norm:FR}
For the remainder term $\FF_R$, there exists a positive constant $C$, such that
    \begin{align*}
        \|(\FF_R,\ve\nabla\FF_R,\ve^2\Delta\FF_R)\|_{L^2}\leq C\Big(1+{\fE}^{\frac{1}{2}}+\ve{\fE}+\ve^3{\fE}^{\frac{3}{2}}+\ve\fF^{\frac{1}{2}}+\ve{\fE}^{\frac{1}{2}}\fF^{\frac{1}{2}}
    \Big).
    \end{align*}
\end{lemma}

\begin{proof}

Applying Lemma \ref{norm:QR,vR} and (\ref{fg Hk}), we infer that
    \begin{align*}
       &\|(\FF_1,\ve\nabla\FF_1,\ve^2\Delta\FF_1)\|_{L^2}\leq C,\\
        &\|(\FF_2,\ve\nabla\FF_2,\ve^2\Delta\FF_2)\|_{L^2}\leq C({\fE}^{\frac{1}{2}}+\ve\fF^{\frac{1}{2}}).
    \end{align*}
By a direct calculation, we have
    \begin{align*}
         \|\ve^k\FF_3\|_{H^k}&\leq C\ve\|\ve^kQ_R\|_{H^k}(\ve\| Q_R\|_{H^2}+\ve^4\|Q_R\|_{H^2}^2+\ve^2\|\BOm_R\|_{H^2})\\&\leq C(\ve{\fE}+\ve^3{\fE}^{\frac{3}{2}}+\ve{\fE}^{\frac{1}{2}}\fF^{\frac{1}{2}}),
    \end{align*}
which together with the above two inequalities yields the completion of the lemma.
\end{proof}

\begin{lemma}\label{norm:vR}
For the remainder terms $\GG_R$ and $\GG_R'$, there exists a positive constant $C$, such that
    \begin{align*}
\|(\GG_R,\ve\nabla\GG_R,\ve^2\Delta\GG_R)\|_{L^2}\leq&~ C(1+{\fE}^{\frac{1}{2}}+\ve^2{\fE}+\ve\fF^{\frac{1}{2}}+\ve^2{\fE}^{\frac{1}{2}}\fF^{\frac{1}{2}}+\ve^4{\fE}\fF^{\frac{1}{2}}),\\
\|(\GG_R',\ve\nabla\GG_R',\ve^2\Delta\GG_R')\|_{L^2}\leq& ~C(1+{\fE}^{\frac{1}{2}}+\fF^{\frac{1}{2}}+\ve{\fE}^{\frac{1}{2}}\fF^{\frac{1}{2}}).
    \end{align*}
\end{lemma}

\begin{proof}
We derive from Lemma \ref{norm:QR,vR} that
    \begin{align*}        \|(\GG_1,\ve\nabla\GG_1,\ve^2\Delta\GG_1)\|_{L^2}\leq&~C,\\\|(\GG_2,\ve\nabla\GG_2,\ve^2\Delta\GG_2)\|_{L^2}\leq&~C({\fE}^{\frac{1}{2}}+\ve\fF^{\frac{1}{2}}).
    \end{align*}
    By the inequality (\ref{fg Hk}), we obtain
    \begin{align*}
           \|\ve^k\GG_3\|_{H^k}\leq&~C\ve^3\|Q_R\|_{H^2}\Big(\|\ve^k\nabla\vv_R\|_{H^k}+\|\ve^kQ_R\|_{H^k}+\ve^3\|Q_R\|_{H^2}\|\ve^k\nabla\vv_R\|_{H^k}\\
           &+\|\ve^k\PP\|_{H^k}\Big)+\ve^6\|\nabla Q_R\|_{H^2}\|\ve^k\nabla Q_R\|_{H^k}\\\leq&~C\ve^2{\fE}^{\frac{1}{2}}(\fF^{\frac{1}{2}}+\ve^2{\fE}^{\frac{1}{2}}\fF^{\frac{1}{2}}+{\fE}^{\frac{1}{2}}),\\
\|\ve^k\GG_R'\|_{H^k}\leq&~C(1+\|\ve^k\vv_R\|_{H^k}+\|\ve^k\nabla\vv_R\|_{H^k}+\ve\|\ve^k\vv_R\|_{H^k}\|\ve^2\nabla\vv_R\|_{H^2})\\
        \leq&~C(1+{\fE}^{\frac{1}{2}}+\fF^{\frac{1}{2}}+\ve{\fE}^{\frac{1}{2}}\fF^{\frac{1}{2}}).
    \end{align*}
Consequently, combining the above estimates leads to the proof of the lemma.
\end{proof}

To ensure the closure of energy estimates, we need to control the singular term $\ve^{-1}\langle\CH_{\nn}^\ve({Q}_R),\PP\rangle$. However, the linearized operator $\CH^{\ve}_{\nn}$ depends on $t$, and its time derivative will generate the difficult term contained in the energy. To overcome this difficulty, the key estimates are given as follows.

\begin{lemma}[A key lemma]\label{A key lemma}
Let $\PP=\dot{Q}_R+\vv_R\cdot\nabla Q^{\ve}$ with $\dot{Q}_R=(\partial_t+\tilde{\vv}\cdot\nabla)Q_R$.
Assume that $(\vv_R,Q_R,\PP)$ is a smooth solution to the remainder system {\rm (\ref{eq-QR})--(\ref{eq-vR-div})} with $m=0$. Then there exists a positive constant $C$ depending on $\nn,\nabla_{t,x}\nn,\tilde{\vv}$ and $\widetilde{Q}$, such that
\begin{align}
- \frac{1}{\ve}\langle\CH_{\nn}^\ve({Q}_R),\PP\rangle\leq &~\frac{\ud}{\ud t}\Big(-\frac{1}{2\ve}\langle\CH_\nn^\ve(Q_R),Q_R\rangle-\CA(Q_R,\PP)\Big)\nonumber\\&+C({\fE}+\ve{\fE}^{\frac{3}{2}}+{\fE}^{\frac{1}{2}}\fF^{\frac{1}{2}}+\ve{\fE}\fF^{\frac{1}{2}}+{\fE}^{\frac{1}{2}}\|\FF_R\|_{L^2}),\label{CH,Q0}\\
    -\ve\langle\partial_i\CH_{\nn}^{\ve}(Q_R),\partial_i\PP\rangle
    \leq&-\frac{\ve}{2}\frac{\ud}{\ud t}\big\langle\CH_{\nn}^\ve(\partial_iQ_R),\partial_i{Q}_R\big\rangle+C({\fE}+\ve{\fE}^{\frac{3}{2}}+{\fE}^{\frac{1}{2}}\fF^{\frac{1}{2}}+\ve{\fE}\fF^{\frac{1}{2}}),\label{CH,Q1}\\
     -\ve^3\langle\Delta\CH_{\nn}^{\ve}(Q_R),\Delta\PP\rangle
     \leq&-\frac{\ve^3}{2}\frac{\ud}{\ud t}\big\langle\CH_{\nn}^\ve(\Delta Q_R),\Delta Q_R\big\rangle+C\big({\fE}+\ve{\fE}^{\frac{3}{2}}+{\fE}^{\frac{1}{2}}\fF^{\frac{1}{2}}+\ve{\fE}\fF^{\frac{1}{2}}\big),\label{CH,Q2}
\end{align}
where
\begin{align}\label{CA-t}
  \CA(Q_R,\PP)\eqdefa J\Big\langle\CH_\nn^{-1}\Big(2bs\dot{\overline{\nn\nn}}\cdot Q^\top_R-2cs^2(\dot{\overline{\nn\nn}}:Q_R^\top) \Big(\nn\nn-\frac{1}{3}\II\Big)\Big),\PP\Big\rangle,
\end{align}
 and $\dot{\overline{\nn\nn}}=(\partial_t+\tilde{\vv}\cdot\nabla)(\nn\nn)$, and $\CH^{-1}_{\nn}$ is defined by {\em (\ref{HM-inverse})} and $\FF_R$ is the remainder term given by {\em (\ref{FFR-remaider-term})}.
\end{lemma}
\begin{proof}
    We only provide here the arguments of (\ref{CH,Q0}). The proof of (\ref{CH,Q1})--(\ref{CH,Q2}) will be left in the Appendix \ref{CH,vv*Q}.
According to the definition of $\PP$, we get
\begin{align}\label{CHve-nn-PP}
    \Big\langle\frac{1}{\ve}\CH_{\nn}^\ve({Q}_R),\PP\Big\rangle=\Big\langle\frac{1}{\ve}\CH_{\nn}^\ve({Q}_R),\Dot{Q}_R+\vv_R\cdot\nabla Q^\ve\Big\rangle.
\end{align}
 Armed with the fact $\vv_R\cdot\nabla Q_0\in \text{Ker}{\CH_\nn}$ and Lemma \ref{v*Q,L(Q)}, we deduce that
 \begin{align}
    -\Big\langle\vv_R\cdot\nabla Q^\ve, \frac{1}{\ve}\CH_\nn^\ve(Q_R)\Big\rangle=& -\langle\vv_R\cdot\nabla(\ve^2Q_R+ \widehat{Q}^\ve), \CH_\nn(Q_R)\rangle-\langle\vv_R\cdot\nabla Q^\ve, \CL(Q_R)\rangle\nonumber\\
    \leq&~C\|Q_R\|_{L^2}(\|\vv_R\|_{L^2}+\|\nabla Q_R\|_{L^2}\|\ve^2\vv_R\|_{H^2})\nonumber\\&+C(\ve^3\|\nabla\vv_R\|_{H^2}\|Q_R\|_{H^1}^2+\|\vv_R\|_{H^1}\|\nabla Q_R\|_{L^2})\nonumber\\
    \leq&~C({\fE}+\ve^2{\fE}^{\frac{3}{2}}+{\fE}^{\frac{1}{2}}\fF^{\frac{1}{2}}+\ve{\fE}\fF^{\frac{1}{2}}),\label{vQ1}
\end{align}
where $\widehat{Q}^\ve=Q_1+\ve Q_2+\ve^2Q_3$. We assume the decomposition $Q_R=Q^{\top}_R+Q^{\perp}_R$ with $Q^{\top}_R\in \text{Ker}{\CH_\nn}$ and $Q^{\perp}_R\in (\text{Ker}{\CH_\nn})^{\perp}$.
From the definition of $\CH_{\nn}$ in (\ref{CH-nn}) and Lemma \ref{v*Q,L(Q)}, we derive that
\begin{align}
  -\Big\langle \frac{1}{\ve}\CH_\nn^\ve(Q_R),\Dot{Q}_R\Big\rangle=&-\frac{1}{2\ve}\frac{\ud }{\ud t}\langle\CH_\nn^\ve(Q_R),{Q}_R\rangle+\frac{1}{2\ve}\langle[\partial_t+\Tilde{\vv}\cdot\nabla,\CH_\nn]Q_R,Q_R\rangle\nonumber\\&-\langle\CL(Q_R),\tilde{\vv}\cdot\nabla Q_R\rangle\nonumber\\=&-\frac{1}{2\ve}\frac{\ud }{\ud t}\langle\CH_\nn^\ve(Q_R),{Q}_R\rangle+\Big\langle \frac{1}{\ve}\mathbf{A}(Q_R^\top),Q_R^\perp\Big\rangle\nonumber\\
  &-\frac{1}{\ve}\langle bs\dot{\overline{\nn\nn}}\cdot Q_R^\perp,Q_R^\perp\rangle-\langle\CL(Q_R),\tilde{\vv}\cdot\nabla Q_R\rangle\nonumber\\
  \leq&-\frac{1}{2\ve}\frac{\ud }{\ud t}\langle\CH_\nn^\ve(Q_R),{Q}_R\rangle+ \frac{1}{\ve}\langle \mathbf{A}(Q_R^\top),Q_R^\perp\rangle+C\fE,\label{Q,CH1}
\end{align}
where $\mathbf{A}(Q)\eqdefa 2bs \dot{\overline{\nn\nn}} \cdot Q-2cs^2(\dot{\overline{\nn\nn}}: Q)(\mathbf{n n}-\frac{1}{3}\II)$.
Using Proposition \ref{linearized-oper-prop} and the equation (\ref{eq-QR}), it follows that
\begin{align}
    &\frac{1}{\ve}\langle \mathbf{A}(Q_R^\top),Q_R^\perp\rangle=\frac{1}{\ve}\langle\CH_\nn^{-1}(\mathbf{A}(Q_R^\top)),\CH_\nn(Q_R)\rangle\nonumber\\
    &\quad=-\Big\langle\CH_\nn^{-1}(\mathbf{A}(Q_R^\top)),J(\partial_t+\vv^\ve\cdot\nabla)\PP+{\mu_1}\underbrace{(\PP-[\BOm_R,Q_0]+\frac{\mu_2}{2\mu_1}\DD_R)}_{{\UU}_0}\Big\rangle\nonumber\\
    &\qquad-\Big\langle\CH_\nn^{-1}(\mathbf{A}(Q_R^\top)),\CL(Q_R)-\FF_R\Big\rangle\nonumber\\
    &\quad=-\frac{\ud}{\ud t}\underbrace{{\langle\CH_\nn^{-1}(\mathbf{A}(Q_R^\top)),J\PP\rangle}}_{\CA(Q_R,\PP)}+J\langle(\partial_t+{\vv}^\ve\cdot\nabla)\CH_\nn^{-1}(\mathbf{A}(Q_R^\top)),\PP\rangle\nonumber\\
    &\qquad-\Big\langle\CH_\nn^{-1}(\mathbf{A}(Q_R^\top)),\mu_1{\UU}_0+\CL(Q_R)-\FF_R\Big\rangle\nonumber\\
    &\quad\leq-\frac{\ud}{\ud t}\CA(Q_R,\PP)+C\|\PP\|_{L^2}\big(\|Q_R\|_{L^2}+\|\dot{Q}_R\|_{L^2}+\ve^3\|\vv_R\|_{H^2}\|Q_R\|_{H^1}\big)\nonumber\\
    &\qquad+C\|Q_R\|_{H^1}\big(\|{\UU}_0\|_{L^2}+\|\nabla Q_R\|_{L^2}+\|\FF_R\|_{L^2}\big)\nonumber\\
    &\quad\leq-\frac{\ud}{\ud t}\CA(Q_R,\PP)+C({\fE}+\ve{\fE}^{\frac{3}{2}}+{\fE}^{\frac{1}{2}}\fF^{\frac{1}{2}}+\ve{\fE}\fF^{\frac{1}{2}}+{\fE}^{\frac{1}{2}}\|\FF_R\|_{L^2}).\label{1HQ1}
\end{align}
From (\ref{CHve-nn-PP})--(\ref{1HQ1}), we obtain the following estimate:
\begin{align*}
   -\Big\langle\frac{1}{\ve}\CH_{\nn}^\ve({Q}_R),\PP\Big\rangle\leq &~\frac{\ud}{\ud t}\Big(-\frac{1}{2\ve}\langle\CH_\nn^\ve(Q_R),Q_R\rangle-\CA(Q_R,\PP)\Big)\\&+C({\fE}+\ve{\fE}^{\frac{3}{2}}+{\fE}^{\frac{1}{2}}\fF^{\frac{1}{2}}+\ve{\fE}\fF^{\frac{1}{2}}+{\fE}^{\frac{1}{2}}\|\FF_R\|_{L^2}).
\end{align*}
This completes the proof of Lemma \ref{A key lemma}.
\end{proof}

We now turn to the argument of Proposition \ref{proposition-PP}. Based on Lemmas \ref{norm:QR,vR}--\ref{A key lemma}, the proof will be divided into three steps.
\begin{proof}[Proof of Proposition \ref{proposition-PP}]

{\it Step 1. $L^2$-estimate.}
Multiplying (\ref{eq-QR}) by $\PP$ and (\ref{eq-vR}) by $\vv_R$, respectively, integrating by parts over the space $\BR$, we deduce that
\begin{align}
    \langle\vv_R,&\partial_t\vv_R\rangle+ J\langle\PP,\partial_t\PP\rangle\nonumber\\
    =&\underbrace{-\langle\beta_1Q_0(Q_0:\DD_R)+\beta_4\DD_R+\beta_5\DD_R\cdot Q_0+\beta_6Q_0\cdot\DD_R,\nabla\vv_R\rangle}_{J_1}\nonumber\\
    &\underbrace{-\langle\beta_7(\DD_R\cdot Q_0^2+Q_0^2\cdot \DD_R),\nabla\vv_R\rangle}_{J_2}\underbrace{-\frac{\mu_2}{2}\langle\PP-[\BOm_R,Q_0],\nabla\vv_R\rangle}_{J_3}\nonumber\\
    &\underbrace{-\mu_1\Big\langle\big[Q_0,(\PP-[\BOm_R,Q_0])\big],\nabla\vv_R\Big\rangle}_{J_4}\underbrace{-\frac{\mu_2}{2}\langle\DD_R,\PP\rangle}_{J_5}\nonumber\\
    &\underbrace{-\mu_1\langle\PP-[\BOm_R,Q_0],\PP\rangle}_{J_6}\underbrace{-\langle\frac{1}{\ve}\CH_\nn^\ve(Q_R),\PP\rangle}_{J_7}\nonumber\\
    &+\langle\nabla\cdot \GG_R+\GG'_R,\vv_R\rangle+\langle\FF_R,\PP\rangle.\label{L^2}
\end{align}
Similar to  the derivation of (\ref{c1}) , we also have
\begin{align*}
    \sum^6_{k=1}J_k=&-{\mu_1}\Big\|\PP-[\BOm_R,Q_0]+\frac{\mu_2}{2\mu_1}\DD_R\Big\|^2_{L^2}-{2\delta}\|\nabla\vv_R\|^2_{L^2}\\
    &-\Big(\overline{\beta}_{1} \left\|\mathbf{n n}: \mathbf{D}_R\right\|_{L^{2}}^{2}+\overline{\beta}_2\left\|\mathbf{D}_R\right\|_{L^{2}}^{2}  +\overline{\beta}_3\left\|\mathbf{n} \cdot \mathbf{D}_R\right\|_{L^{2}}^{2}\Big)
    \\
    \leq &-c_1\fF_0,
\end{align*}
where the constant $c_1=\min\{\delta,1\}$, and $\delta$ represents a small positive constant that had been determined in (\ref{relation overlinebeta}). Moreover, for the last two terms in (\ref{L^2}), it follows that
\begin{align*}
\langle\nabla\cdot \GG_R+\GG'_R,\vv_R\rangle+\langle\FF_R,\PP\rangle\leq \|(\GG'_R,\FF_R)\|_{L^2}{\fE}^{\frac{1}{2}}+\|\GG_R\|_{L^2}\fF^{\frac{1}{2}}.
\end{align*}
Using $\nabla\cdot\tilde{\vv}=0$ and Lemma \ref{norm:FR}, we derive that
\begin{align*}
    \frac{M}{2}\frac{\ud}{\ud t}\|Q_R\|_{L^2}^2=M\langle Q_R,\dot{Q}_R\rangle\leq CM({\fE}+\ve{\fE}^{\frac{3}{2}}),
\end{align*}
where $M\geq 1$ is a sufficiently large constant to be determined later.
Combining (\ref{L^2}) with (\ref{CH,Q0})  and  the above estimates, we obtain
\begin{align}\label{L2-estimate}
    &\frac{\ud}{\ud t}\big({\fE}_{0,0}(t)+\CA(Q_R,\PP)\big)+c_1\fF_0(t)\nonumber\\
    &\quad\leq CM({\fE}+\ve{\fE}^{\frac{3}{2}}+{\fE}^{\frac{1}{2}}\fF^{\frac{1}{2}}+\ve{\fE}\fF^{\frac{1}{2}})+\|(\GG'_R,\FF_R)\|_{L^2}{\fE}^{\frac{1}{2}}+\|\GG_R\|_{L^2}\fF^{\frac{1}{2}},
\end{align}
where $\fE_{0,0}(t), \fF_0(t)$ are defined by (\ref{energyE}) and (\ref{energyF}), respectively, and $\CA(Q_R,\PP)$ is expressed by (\ref{CA-t}).

{\it{Step 2. $H^1$-estimate}.}
We first apply the derivative $\partial_i$ on (\ref{eq-QR}) and take the $L^2$-inner product with $\partial_i\PP$, and then act $\partial_i$ on (\ref{eq-vR}) and
take the $L^2$-inner product with $\partial_i\vv_R$, it follows that
\begin{align}
{\ve^2}\langle\partial_i&\vv_R,\partial_t(\partial_i\vv_R)\rangle+\ve^{2}J\langle\partial_i\PP,\partial_t(\partial_i\PP)\rangle\nonumber\\
   = &\underbrace{-\ve^2\Big\langle\partial_i\Big(\beta_1Q_0(Q_0:\DD_R)+\beta_4\DD_R+\beta_5\DD_R\cdot Q_0+\beta_6Q_0\cdot\DD_R\Big),\nabla\partial_i\vv_R\Big\rangle}_{K_1}\nonumber\\
    &\underbrace{-\langle\ve^2\beta_7\partial_i(\DD_R\cdot Q_0^2+Q_0^2\cdot \DD_R),\nabla\partial_i\vv_R\rangle}_{K_2}\underbrace{-\ve^2\frac{\mu_2}{2}\langle\partial_i(\PP-[\BOm_R,Q_0]),\nabla\partial_i\vv_R\rangle}_{K_3}\nonumber\\
    &\underbrace{-\ve^2\mu_1\Big\langle\partial_i\big[Q_0,(\PP-[\BOm_R,Q_0])\big],\nabla\partial_i\vv_R\Big\rangle}_{K_4}\underbrace{-\ve^2\frac{\mu_2}{2}\langle\partial_i\DD_R,\partial_i\PP\rangle}_{K_5}\nonumber\\
    &\underbrace{-\ve^2\mu_1\langle\partial_i(\PP-[\BOm_R,Q_0]),\partial_i\PP\rangle}_{K_6}\underbrace{-\langle{\ve}\partial_i\CH_\nn^\ve(Q_R),\partial_i\PP\rangle}_{K_7}+\ve^{2}J\langle\partial_i\vv^\ve\cdot\nabla\PP,\partial_i\PP\rangle\nonumber\\
    &-\langle\ve^2\partial_i \GG_R,\nabla\partial_i\vv_R\rangle+\langle\ve^2\partial_i\GG_R',\partial_i\vv_R\rangle+\ve^2\langle\partial_i\FF_R,\partial_i\PP\rangle\nonumber\\
    \leq &~\sum^7_{i=1}K_i+C(1+\ve\fF^{\frac{1}{2}}){\fE}+\ve\|(\partial_i\GG'_R,\partial_i\FF_R)\|_{L^2}{\fE}^{\frac{1}{2}}+\ve\|\partial_i\GG_R\|_{L^2}\fF^{\frac{1}{2}}.\label{H^1}
\end{align}
Now we estimate (\ref{H^1}) term by term.
Remembering the relation $\beta_6-\beta_5=\mu_2$, the terms $K_1$ and $K_2$ can be estimated as
\begin{align}
K_{1}+K_{2} \leq & -\varepsilon^{2}\left\langle\beta_{1} Q_{0}\left(Q_{0}: \partial_{i} \mathbf{D}_{R}\right)+\beta_{4} \partial_{i} \mathbf{D}_{R}, \nabla \partial_{i} \mathbf{v}_{R}\right\rangle \nonumber\\&-\ve^2\langle\beta_{5} \partial_{i} \mathbf{D}_{R} \cdot Q_{0}+\beta_{6} Q_{0} \cdot \partial_{i} \mathbf{D}_{R}, \nabla \partial_{i} \mathbf{v}_{R}\rangle\nonumber\\
& -\varepsilon^{2}\left\langle\beta_{7}\left(\partial_{i} \mathbf{D}_{R} \cdot Q_{0}^{2}+Q_{0}^{2} \cdot \partial_{i} \mathbf{D}_{R}\right), \nabla \partial_{i} \mathbf{v}_{R}\right\rangle\nonumber\\&
+C\left\|\varepsilon \nabla \mathbf{v}_{R}\right\|_{L^{2}}\left\|\varepsilon \nabla \partial_{i} \mathbf{v}_{R}\right\|_{L^{2}} \nonumber\\
\leq & -\varepsilon^{2}\big\langle\beta_{1} Q_{0}\left(Q_{0}: \partial_{i} \mathbf{D}_{R}\right)+\beta_{4} \partial_{i} \mathbf{D}_{R}, \partial_{i} \mathbf{D}_{R}\big\rangle\nonumber \\&
-\ve^2\langle\frac{\beta_{5}+\beta_{6}}{2}\left(Q_{0} \cdot \partial_{i} \mathbf{D}_{R}+\partial_{i} \mathbf{D}_{R} \cdot Q_{0}\right), \partial_{i} \mathbf{D}_{R}\big\rangle\nonumber\\
& -\varepsilon^{2}\left\langle\beta_{7}\left(\partial_{i} \mathbf{D}_{R} \cdot Q_{0}^{2}+Q_{0}^{2} \cdot \partial_{i} \mathbf{D}_{R}\right), \partial_{i} \mathbf{D}_{R}\right\rangle\nonumber\\&
+\underbrace{\ve^2\frac{\mu_2}{2}\left\langle\left[\partial_{i} \mathbf{D}_{R}, Q_{0}\right], \nabla \partial_{i} \mathbf{v}_{R}\right\rangle}_{K_{1}^{\prime}}+C \mathfrak{E}^{\frac{1}{2}} \mathfrak{F}^{\frac{1}{2}}.\label{T12}
\end{align}
Noting the symmetry of $[\partial_i\BOm_R,Q_0]$, we infer that
\begin{align}
K_{1}^{\prime}+K_{3}+K_{5} \leq & ~\varepsilon^{2} \frac{\mu_{2}}{2}\left\langle\left[\partial_{i} \mathbf{D}_{R}, Q_{0}\right], \partial_{i} \boldsymbol{\Omega}_{R}\right\rangle-\varepsilon^{2} \mu_{2}\left\langle\partial_{i} \mathbf{D}_{R}, \partial_{i} \PP\right\rangle\nonumber \\
& +\varepsilon^{2} \frac{\mu_{2}}{2}\left\langle\left[\partial_{i} \boldsymbol{\Omega}_{R}, Q_{0}\right], \partial_{i} \mathbf{D}_{R}\right\rangle+C\left\|\varepsilon \nabla \mathbf{v}_{R}\right\|_{L^{2}}\left\|\varepsilon \nabla \partial_{i} \mathbf{v}_{R}\right\|_{L^{2}} \nonumber\\
\leq & -\varepsilon^{2} \mu_{2}\left\langle\partial_{i} \PP-\left[\partial_{i} \boldsymbol{\Omega}_{R}, Q_{0}\right], \partial_{i} \mathbf{D}_{R}\right\rangle+C \mathfrak{E}^{\frac{1}{2}} \mathfrak{F}^{\frac{1}{2}}.\label{T135}
\end{align}
At the same time, the terms $K_4$ and $K_6$ can be calculated as
\begin{align}
K_4+K_6 \leq & -\varepsilon^2 \mu_1\left\langle\left[Q_0,\left(\partial_i \PP-\left[\partial_i \boldsymbol{\Omega}_R, Q_0\right]\right)\right], \nabla \partial_i \mathbf{v}_R\right\rangle\nonumber \\
& -\varepsilon^2 \mu_1\left\langle\partial_i \PP-\left[\partial_i \boldsymbol{\Omega}_R, Q_0\right], \partial_i \PP\right\rangle \nonumber\\
& +C\big(\|\varepsilon(\PP-\left[\boldsymbol{\Omega}_R, Q_0\right])\|_{L^2}+\left\|\varepsilon\left[\boldsymbol{\Omega}_R, \partial_i Q_0\right]\right\|_{L^2}\big)\left\|\varepsilon \nabla \partial_i \mathbf{v}_R\right\|_{L^2} \nonumber\\
& +C\left\|\varepsilon \nabla \mathbf{v}_R\right\|_{L^2}\left\|\varepsilon \partial_i \PP\right\|_{L^2}\nonumber \\
\leq & -\varepsilon^2 \mu_1\left\|\partial_i \PP-\left[\partial_i \boldsymbol{\Omega}_R, Q_0\right]\right\|_{L^2}^2+C(\mathfrak{E}^{\frac{1}{2}} \mathfrak{F}^{\frac{1}{2}}+{\fE}).\label{T46}
\end{align}
It can be observed that
\begin{align*}
&-\varepsilon^2 \mu_1\left\|\partial_i \PP-\left[\partial_i \boldsymbol{\Omega}_R, Q_0\right]\right\|_{L^2}^2-\varepsilon^2 \mu_2\left\langle\partial_i \PP-\left[\partial_i \boldsymbol{\Omega}_R, Q_0\right], \partial_i \mathbf{D}_R\right\rangle \nonumber\\
&\qquad=-\varepsilon^2 \mu_1\big\|\underbrace{\partial_i \PP-\left[\partial_i \boldsymbol{\Omega}_R, Q_0\right]+\frac{\mu_2}{2 \mu_1} \partial_i \mathbf{D}_R}_{{\UU}_1}\big\|_{L^2}^2+\frac{\mu_2^2}{4 \mu_1}\left\|\partial_i \mathbf{D}_R\right\|_{L^2}^2.
\end{align*}
Similar to the argument of (\ref{c1}), we have
\begin{align}
\sum^6_{i=1}K_i\leq & -\varepsilon^2 \beta_1 s^2\left\|\mathbf{n n}: \partial_i \mathbf{D}_R\right\|_{L^2}^2-\varepsilon^2\left(\beta_4-\frac{s\left(\beta_5+\beta_6\right)}{3}-\frac{\mu_2^2}{4\mu_1}-4\delta\right)\left\|\partial_i \mathbf{D}_R\right\|_{L^2}^2\nonumber \\ & -\varepsilon^2 s\left(\beta_5+\beta_6\right)\left\|\mathbf{n} \cdot \partial_i \mathbf{D}_R\right\|_{L^2}^2-\varepsilon^2 \mu_1\left\|{\UU}_1\right\|_{L^2}^2 -2\ve^2\delta\|\nabla\partial_i\vv_R\|_{L^2}^2\nonumber\\ & +C(\mathfrak{E}^{\frac{1}{2}} \mathfrak{F}^{\frac{1}{2}}+{\fE})\nonumber\\ \leq & -c_1\fF_1+C ({\fE}+{\fE}^{\frac{1}{2}}\mathfrak{F}^{\frac{1}{2}}).\label{T1-6}
\end{align}
Summarizing (\ref{H^1}), (\ref{T1-6}) and (\ref{CH,Q1}), we deduce that
\begin{align}
    \frac{\ud}{\ud t}{\fE}_{0,1}(t)+c_1\fF_1(t)\leq&~ C\big( {\fE}+ \ve{\fE}^{\frac{3}{2}}+{\fE}^{\frac{1}{2}}\fF^{\frac{1}{2}}+\ve{\fE}\fF^{\frac{1}{2}}\big)\nonumber\\
    &~+\ve\|(\partial_i\GG_R',\partial_i\FF_R)\|_{L^2}{\fE}^{\frac{1}{2}}+\|\ve\partial_i\GG_R\|_{L^2}\fF^{\frac{1}{2}}.\label{H1-estimate}
\end{align}

{\it {Step 3. $H^2$-estimate.}}
Similarly, from the system (\ref{eq-QR})--(\ref{eq-vR-div}), we derive that
\begin{align}
    \frac{\ve^4}{2}\frac{\ud}{\ud t}&\Big\{\|\Delta\vv_R\|_{L^2}^2+J\|\Delta\PP\|_{L^2}^2\Big\}\nonumber\\
   = &\underbrace{-\ve^4\Big\langle\Delta\Big(\beta_1Q_0(Q_0:\DD_R)+\beta_4\DD_R+\beta_5\DD_R\cdot Q_0+\beta_6Q_0\cdot\DD_R\Big),\nabla\Delta\vv_R\Big\rangle}_{V_1}\nonumber\\
    &\underbrace{-\ve^4\beta_7\langle\Delta(\DD_R\cdot Q_0^2+Q_0^2\cdot \DD_R),\nabla\Delta\vv_R\rangle}_{V_2}\underbrace{-\ve^4\frac{\mu_2}{2}\langle\Delta(\PP-[\BOm_R,Q_0]),\nabla\Delta\vv_R\rangle}_{V_3}\nonumber\\
    &\underbrace{-\ve^4\mu_1\Big\langle\Delta\big[Q_0,(\PP-[\BOm_R,Q_0])\big],\nabla\Delta\vv_R\Big\rangle}_{V_4}\underbrace{-\ve^4\frac{\mu_2}{2}\langle\Delta\DD_R,\Delta\PP\rangle}_{V_5}\nonumber\\
    &\underbrace{-\ve^4\mu_1\langle\Delta(\PP-[\BOm_R,Q_0]),\Delta\PP\rangle}_{V_6}\underbrace{-\ve^3\langle\Delta\CH_\nn^\ve(Q_R),\Delta\PP\rangle}_{V_7}\nonumber\\
    &-\langle\ve^4\Delta \GG_R,\nabla\Delta\vv_R\rangle+\langle\ve^4\Delta\GG_R',\Delta\vv_R\rangle+\ve^4\langle\Delta\FF_R,\Delta\PP\rangle\nonumber\\&-\ve^{4}J\langle\Delta(\vv^\ve\cdot\nabla\PP),\Delta \PP\rangle\nonumber\\
    \leq&~\sum^7_{k=1}V_k-\ve^{4}J\langle\Delta(\vv^\ve\cdot\nabla\PP),\Delta \PP\rangle\nonumber\\
    &~+C\left(\ve^2\|(\Delta\GG_R',\Delta\FF_R)\|_{L^2}{\fE}^{\frac{1}{2}}+\ve^2\|\Delta\GG_R\|_{L^2}\fF^{\frac{1}{2}}\right).\label{H^2}
\end{align}
Due to the incompressible condition $\nabla\cdot\vv^\ve=0$, it holds that
\begin{align*}
    -\ve^{4}J&\langle\Delta(\vv^\ve\cdot\nabla\PP),\Delta \PP\rangle\\=&-\ve^{4}J\big\langle[\Delta,\vv^\ve\cdot\nabla]\PP,\Delta \PP\big\rangle\\
    =&-\ve^{4}J\langle\Delta\vv^\ve\cdot\nabla\PP,\Delta \PP\rangle-\ve^{4}J\langle\partial_i\vv^\ve\cdot\nabla\partial_i\PP),\Delta \PP\rangle\\
    \leq&~C\ve^{4}\big(\|\Delta\vv^\ve\|_{H^1}\|\PP\|_{H^2}\|\Delta\PP\|_{L^2}+\|\partial_i\vv^\ve\|_{H^2}\|\nabla^2\PP\|_{L^2}^2\big)\\
    \leq&~C({\fE}+\ve{\fE}\fF^{\frac{1}{2}}).
\end{align*}
With regards to the terms $V_1$ and $V_2$, we estimate them as follows:
\begin{align}
   V_1+V_2
    \leq&{-\ve^4\langle\beta_1Q_0(Q_0:\Delta\DD_R)+\beta_4\Delta\DD_R+\beta_5\Delta\DD_R\cdot Q_0+\beta_6Q_0\cdot\Delta\DD_R,\nabla\Delta\vv_R\rangle}\nonumber\\
    &{-\ve^4\beta_7\langle\Delta\DD_R\cdot Q_0^2+Q_0^2\cdot \Delta\DD_R,\nabla\Delta\vv_R\rangle}+\ve^4\|\nabla\vv_R\|_{H^1}\|\nabla\Delta\vv_R\|_{L^2}\nonumber\\
    \leq&{-\ve^4\langle\beta_1Q_0(Q_0:\Delta\DD_R)+\beta_4\Delta\DD_R+\frac{\beta_5+\beta_6}{2}(\Delta\DD_R\cdot Q_0+Q_0\cdot\Delta\DD_R),\Delta\DD_R\rangle}\nonumber\\
    &{-\ve^4\beta_7\langle\Delta\DD_R\cdot Q_0^2+Q_0^2\cdot \Delta\DD_R,\Delta\DD_R\rangle}+\underbrace{\ve^4\frac{\mu_2}{2}\langle[\Delta\DD_R,Q_0],\nabla\Delta\vv_R\rangle}_{V_1'}+C{\fE}^{\frac{1}{2}}\fF^{\frac{1}{2}}.\label{MT12}
\end{align}
Further, a direct calculation enable us to get
\begin{align}
    V_1'+V_3+V_5\leq&~\ve^4\frac{\mu_2}{2}\langle[\Delta\DD_R,Q_0],\Delta\BOm_R\rangle-\ve^4\langle\Delta\DD_R,\Delta\PP\rangle\nonumber\\
   & +\ve^4\frac{\mu_2}{2}\langle[\Delta\BOm_R,Q_0],\Delta\DD_R\rangle+C\ve^4\|\nabla\vv_R\|_{H^1}\|\nabla\Delta\vv_R\|_{L^2}\nonumber\\
    \leq&-\ve^4\mu_2\langle\Delta\PP-[\Delta\BOm_R,Q_0],\Delta\DD_R\rangle+C{\fE}^{\frac{1}{2}}\fF^{\frac{1}{2}}.\label{MT135}
\end{align}
For the estimates of $V_4$ and $V_6$, it holds that
\begin{align}
V_4+V_6\leq & -\varepsilon^4 \mu_1\left\langle\big[Q_0,(\Delta \PP-\left[\Delta \boldsymbol{\Omega}_R, Q_0\right])\big], \nabla \Delta \mathbf{v}_R\right\rangle \nonumber\\
& -\varepsilon^4 \mu_1\left\langle\Delta \PP-\left[\Delta \boldsymbol{\Omega}_R, Q_0\right], \Delta \PP\right\rangle\nonumber \\
& +C\big(\|\varepsilon^2(\PP-\left[\boldsymbol{\Omega}_R, Q_0\right])\|_{H^1}+\left\|\varepsilon^2\left[\boldsymbol{\Omega}_R, \partial_i Q_0\right]\right\|_{H^1}\big)\left\|\varepsilon ^2\nabla \Delta \mathbf{v}_R\right\|_{L^2} \nonumber\\
& +C\left\|\varepsilon^2 \nabla \mathbf{v}_R\right\|_{H^1}\left\|\varepsilon^2 \Delta \PP\right\|_{L^2} \nonumber\\
\leq & -\varepsilon^4 \mu_1\left\|\Delta \PP-\left[\Delta\boldsymbol{\Omega}_R, Q_0\right]\right\|_{L^2}^2+C(\mathfrak{E}^{\frac{1}{2}} \mathfrak{F}^{\frac{1}{2}}+{\fE}).\label{MT46}
\end{align}
A simple calculation leads to
\begin{align}
-\varepsilon^4 \mu_1&\left\|\Delta \PP-\left[\Delta \boldsymbol{\Omega}_R, Q_0\right]\right\|_{L^2}^2-\varepsilon^4 \mu_2\left\langle\Delta \PP-\left[\Delta \boldsymbol{\Omega}_R, Q_0\right], \Delta \mathbf{D}_R\right\rangle \nonumber\\
&=-\varepsilon^2 \mu_1\big\|\underbrace{\Delta \PP-\left[\Delta\boldsymbol{\Omega}_R, Q_0\right]+\frac{\mu_2}{2 \mu_1} \Delta \mathbf{D}_R}_{{\UU}_2}\big\|_{L^2}^2+\frac{\mu_2^2}{4 \mu_1}\left\|\Delta \mathbf{D}_R\right\|_{L^2}^2.\label{U2}
\end{align}
Similar to the derivation of (\ref{c1}), it follows that
\begin{align}
\sum^6_{k=1}V_k \leq & -\ve^4 \beta_1 s^4\left\|\mathbf{n n}: \Delta \mathbf{D}_R\right\|_{L^2}^2-\varepsilon^4\left(\beta_4-\frac{s\left(\beta_5+\beta_6\right)}{3}-\frac{\mu_2^2}{4\mu_1}-4\delta\right)\left\|\Delta \mathbf{D}_R\right\|_{L^2}^2 \nonumber\\ & -\varepsilon^4 s\left(\beta_5+\beta_6\right)\left\|\mathbf{n} \cdot \Delta \mathbf{D}_R\right\|_{L^2}^2-\varepsilon^4 \mu_1\left\|{\UU}_2\right\|_{L^2}^2 -2\ve^4\delta\|\nabla\Delta\vv_R\|_{L^2}^2\nonumber\\ & +C(\mathfrak{E}^{\frac{1}{2}} \mathfrak{F}^{\frac{1}{2}}+{\fE})\nonumber \\ \leq & - c_1 \fF_2+C ({\fE}+{\fE}^{\frac{1}{2}}\mathfrak{F}^{\frac{1}{2}}).\label{MT1-6}
\end{align}
Consequently, using Lemma \ref{A key lemma} and the above estimates, we obtain
\begin{align}
    \frac{\ud}{\ud t}{\fE}_{0,2}(t)+c_1\fF_2(t)\leq&~C\Big({\fE}+\ve{\fE}^{\frac{3}{2}}+{\fE}^{\frac{1}{2}}\fF^{\frac{1}{2}}+\ve{\fE}\fF^{\frac{1}{2}}\nonumber\\&+\ve^2\|(\Delta\GG_R',\Delta\FF_R\|_{L^2}{\fE}^{\frac{1}{2}}
    +\ve^2\fF^{\frac{1}{2}}\|\Delta\GG_R\|_{L^2}\Big).\label{H2-estimate}
\end{align}

Therefore, combining (\ref{L2-estimate}), (\ref{H1-estimate}) and (\ref{H2-estimate}), we deduce that
\begin{align}\label{fE-CA-energy}
    &\frac{\ud}{\ud t}\big({\fE}(t)+\CA(Q_R,\PP)\big)+c_1\fF(t)\nonumber\\
    &\quad\leq CM\Big({\fE}+\ve{\fE}^{\frac{3}{2}}+{\fE}^{\frac{1}{2}}\fF^{\frac{1}{2}}+\ve{\fE}\fF^{\frac{1}{2}}\Big)\nonumber\\
    &\qquad+C\sum_{k=0,1,2}\ve^{2k}\Big({\fE}^{\frac{1}{2}}\|(\partial_i^k\GG_R',\partial_i^k\FF_R)\|_{L^2}+\fF^{\frac{1}{2}}\|\partial_i^k\GG_R\|_{L^2}\Big)\nonumber\\
    &\quad\leq CM\Big(({\fE}^{\frac{1}{2}}+{\fE}+\ve{\fE}^{\frac{3}{2}}+\ve^3{\fE}^2)+\fF^{\frac{1}{2}}(1+{\fE}^{\frac{1}{2}}+\ve{\fE})+\fF(\ve+\ve{\fE}^{\frac{1}{2}}+\ve^4{\fE})\Big)\nonumber\\
    &\quad\leq CM^2(1+\fE+\ve^2\fE^2)+CM\fF(\ve+\ve\fE^{\frac{1}{2}}+\ve^4\fE)+\frac{c_1}{2}\fF.
\end{align}
Applying the definition of $\CA(Q_R,\PP)$, it can be estimated as
\begin{align*}
    |\CA(Q_R,\PP)|=&~J\left|\int_{\BR}\Big(\CH_\nn^{-1}\Big(2bs\dot{\overline{\nn\nn}}\cdot Q_R^\top-2cs^2(\dot{\overline{\nn\nn}}:Q_R^\top)\Big(\nn\nn-\frac{1}{3}\II\Big)\Big):\PP\Big)\ud\xx\right|\\
    \leq&~C(\nn,\tilde{\vv})\|\PP\|_{L^2}\|Q_R\|_{L^2}\leq \frac{1}{4}\Big(J\|\PP\|_{L^2}^2+C_1(\nn,\tilde{\vv})\|Q_R\|_{L^2}^2\Big).
\end{align*}
Then, it is sufficient to take $M\geq \max\{1,C_1(\nn,\tilde{\vv})\}$ so that $|\CA(Q_R,\PP)|\leq \frac{1}{2} \fE(t)$. Hence, from (\ref{fE-CA-energy}), there exists a constant $C>0$ independent of $(\ve,M)$, such that
\begin{align*}
    {\fE}(t)&-3\fE(0)+{c_1}\int_0^t\fF(\tau)\ud \tau\\&\leq CM^2\int_0^t\Big(\big(1+{\fE}(\tau)+\ve^2{\fE}^2(\tau)\big)+\fF(\tau)\big(\ve+\ve{\fE}^{\frac{1}{2}}(\tau)+\ve^4{\fE}(\tau)\big)\Big)\ud \tau.
\end{align*}
This completes the proof of Proposition \ref{proposition-PP}.
\end{proof}

\section{From the Qian--Sheng model to the noninertial Ericksen--Leslie model}\label{mneq0-section}

In this section, we present a rigorous derivation of the noninertial Ericksen--Leslie model starting from the Qian--Sheng model. Two small parameters derived from the elastic coefficients $L_i(i=1,2,3)$ and the inertial constant $J$ are considered. For this purpose, the small parameters $\ve$ and $\ve^m$ are introduced simultaneously, where $m\in \mathbb{Z}^+$ and $\chi(m)=0$. The Qian--Sheng model is still a hyperbolic system with the small parameter $\ve$, while the limit model is the parabolic Ericksen--Leslie system. The structure of the limit model has changed fundamentally, since the small inertial parameter $\ve^m(m\geq1)$ goes to zero. In brief, by the Hilbert expansion, we will rigorously show that the smooth solution to the inertial Qian--Sheng model converges to the solution to the noninertial Ericksen--Leslie model.

Similar to Proposition \ref{0toEL}, we also have the following proposition.

\begin{proposition}\label{0toEL-m>0}
   Under the condition of $m\in\mathbb{Z}^+$, i.e., $\chi(m)=0$,
   if $(Q_0, \vv_0)$ is a smooth solution of the $O(1)$ system {\em (\ref{ve 0-Q})--(\ref{ve 0-div})},  then $(\nn, \vv_0)$ must be a solution of the noninertial Ericksen--Leslie system {\em(\ref{EL-vv})--(\ref{EL-div})} with $I=0$,  where the coefficients are determined by {\rm(\ref{coefficients})}.
\end{proposition}

\subsection{Existence of the Hilbert expansion}

In this subsection, let us solve $(Q_k,\vv_k)(k=1,2)$ and $Q_3$ from the system (\ref{ve 1-Q})--(\ref{ve 2-div}) with $\chi(m)=0$. Note that the system (\ref{ve 1-Q})--(\ref{ve 1-div}) and the system (\ref{ve 2-Q})--(\ref{ve 2-div}) are all parabolic. Then their solutions all have better continuity. This is an essential difference from the hyperbolic-parabolic system in Subsection \ref{J-existence-He}. Thus, it is relatively easy to prove the existence of the Hilbert expansion. The corresponding existence result can be stated as follows.

\begin{proposition}\label{nu:prop Q(1)}
Assume that $Q_0=s(\nn\nn-\frac{1}{3}\II)$. Let $(\nn, \vv_0)$ be a smooth solution of the noninertial Ericksen--Leslie system {\em(\ref{EL-vv})--(\ref{EL-div})} {\em(}with the inertial constant $I=0${\em)} on $[0, T]$ and satisfy
    \[(\vv_0,\nabla \nn)\in C([0, T];H^\ell),\quad {\ell\ge 20.} \]
Then there exists the solutions $(Q_k, \vv_k)(k=1,2)$ and $Q_3\in(\text{Ker}\CH_{\nn})^{\perp}$ of the $O(\ve^k)$ system {\em (\ref{ve 1-Q})--(\ref{ve 2-div})} with $\chi(m)=0$ satisfying
\begin{align*}
 (\vv_k,\nabla Q_k)\in C([0, T];H^{\ell-4k})(k=0,1,2),\quad Q_3\in C([0, T];H^{\ell-11}).
\end{align*}
\end{proposition}

Similar to the argument of Proposition \ref{prop Q(1)}, we will only prove the existence of the solution $(Q_1,\vv_1)$, then derive a linear system and a closed energy estimate. The existence of the solutions $(Q_2,\vv_2)$ and $Q_3$ can be shown by a similar argument.

The parabolic system (\ref{ve 1-Q})--(\ref{ve 1-div}) is not closed since it involves $Q_2$ which is unknown. Thus, we need to project this system into $\text{Ker}\CH_{\nn}$ to remove the non-leading term. In what follows, we denote by $\widetilde{L}(Q_1^\top,\vv_1)$ the terms only linearly depending on $(Q_1^\top,\vv_1)$ (not have their derivatives) with the coefficients belonging to $C([0,T];H^{\ell-1})$. The function $\widetilde{R}\in C([0,T];H^{\ell-3})$ is denoted as the terms depending only on $\nn,\vv_0$ and $Q^{\perp}_1$. Mimicking
the derivation of Lemma \ref{MP-Q1}, it follows that
\begin{align*}
    &\MP^{in}(\DP_1)=\dot{Q}_1^\top+\widetilde{L}(Q_1^\top,\vv_1)+\widetilde{R},\quad \MP^{out}(\DP_1)=\widetilde{L}(Q_1^\top,\vv_1)+\widetilde{R},\\
    &\MP^{in}(\TT_1)=\widetilde{L}(Q_1^\top)+\chi(m)\widetilde{L}(\vv_1)+\widetilde{R}=\widetilde{L}(Q_1^\top)+\widetilde{R},
\end{align*}
where the notation $\dot{f}^\top=(\partial_t+\vv_0\cdot\nabla)f^{\top}$, and the terms $\DP_1$ and $\TT_1$ are defined by (\ref{DP_1}) and (\ref{TT_1}), respectively.

 We next prove Proposition \ref{nu:prop Q(1)}.

 \begin{proof}[Proof of Proposition \ref{nu:prop Q(1)}]

Let $(\nn,\vv_0)$ be a smooth solution to the noninertial Ericksen--Leslie system (\ref{EL-vv})--(\ref{EL-div}) on $[0,T]$ with $I=0$, such that
\begin{align*}
    (\nabla\nn,\vv_0)\in C([0,T];H^{\ell}),~~\ell\geq 20.
\end{align*}
Since $Q_0=s(\nn\nn-\frac{1}{3}\II)$, we have $Q_0\in C([0,T];H^{\ell+1})$.
Let $Q_1=Q^{\top}_1+Q^{\perp}_1$ with $Q^{\top}_1\in {\text{Ker}\CH_{\nn}}$ and $Q^{\perp}_1\in(\text{Ker}\CH_{\nn})^{\perp}$. From the equation $(\ref{ve 0-Q})$,
$Q_1^\perp$ can be determined by
\begin{align*}
    Q_1^\perp=-\CH_\nn^{-1}\Big(\mu_1\dot{Q}_0+\CL(Q_0)+\frac{\mu_2}{2}\DD_0-\mu_1[\BOm_0,Q_0]\Big)\in C([0,T];H^{\ell-1}).
\end{align*}

Taking the projection $\MP^{in}$ on both sides of ($\ref{ve 1-Q}$), we derive a linear system of $(Q_1^\top,\vv_1)$, which is given by
\begin{align}
   \mu_1\dot{Q}_1^\top=&~\MP^{in}\Big(-\CL(Q_1^\top)-\frac{\mu_2}{2}\DD_1+\mu_1[\BOm_1,Q_0] \Big)+\widetilde{L}(Q_1^\top,\vv_1)+\widetilde{R},\label{nu:L Q1}\\
    \dot{\vv}_1=&-\vv_1\cdot\nabla\vv_0-\nabla p_1+\nabla\cdot\Big(\beta_1Q_0(Q_0:\DD_1)+\beta_4\DD_1\nonumber\\&+\beta_5\DD_1\cdot Q_0+\beta_6Q_0\cdot Q_1+\beta_7(\DD_1\cdot Q_0^2+Q_0^2\cdot \DD_1)\nonumber\\&+\frac{\mu_2}{2}\big(\dot{Q}_1^\top-[\BOm_1,Q_0]\big)+\mu_1\big[Q_0,(\dot{Q}_1^\top-[\BOm_1,Q_0])\big]\nonumber\\
   & +\sigma^d({Q}_1^\top,Q_0)+\sigma^d(Q_0,Q_1^\top)+\widetilde{L}(Q_1^\top,\vv_1)+\widetilde{R}\Big),\label{nu:L v1}\\
    \nabla\cdot \vv_1=&~0.\label{nu:L 1div}
\end{align}
The next task is to show the following energy estimate:
\begin{equation}
    \frac{\ud}{\ud t}\widehat{\CE}_\ell(t)+c_1\widehat{\CF}_\ell(t)\leq C(\widehat{\CE}_\ell+\|\widetilde{R}\|_{H^{\ell-3}}^2), \label{CEk}
\end{equation}
where the energy functionals are defined by, respectively,
\begin{align*}
\widehat{\CE}_\ell(t) \stackrel{\text { def }}{=} &\frac{1}{2}\sum_{k=0}^{\ell-4}\left(\left\|\partial_i^k\mathbf{v}_1\right\|_{L^2}^2+\langle\CL(\partial_i^kQ_1^\top),\partial_i^kQ_1^\top\rangle\right)+\frac{\mu_1}{2}\|Q_1^\top\|_{L^2}^2,\\
 \widehat{\CF}_\ell(t) \stackrel{\text { def}}{=}& ~\sum_{k=0}^{\ell-4}\|\nabla\partial_i^k\vv_1\|_{L^2}^2.
\end{align*}
The energy inequality (\ref{CEk}) will ensure that the linear system (\ref{nu:L Q1})--(\ref{nu:L 1div}) has a unique solution $(\vv_1,Q_1^\top)\in C([0,T];H^{\ell-4}\times H^{\ell-3})$.
Without loss of generality, we only show (\ref{CEk}) for the case of $k=0$ and the proof is similar for the general case.
We set
\begin{align*}
\widehat{\CE}(t){=} \frac{1}{2}\int_{\mathbb{R}^3}\left(\left|\mathbf{v}_1\right|^2+\CL(Q_1^\top):Q_1^\top+\mu_1|Q_1^\top|^2\right)\mathrm{d} \mathbf{x},\quad
 \widehat{\CF}(t) {=}\int_{\mathbb{R}^3} |\nabla\vv_1|^2\ud\xx.
\end{align*}
Firstly, using equation (\ref{L Q1}) we derive that
\begin{align}
\frac{\mu_1}{2}\frac{\ud}{\ud t}\|Q_1^\top\|_{L^2}^2=&\nonumber\Big\langle\MP^{in}\Big(-\CL(Q_1^\top)-\frac{\mu_2}{2}\DD_1+\mu_1[\BOm_1,Q_0]\Big)+\widetilde{L}(Q_1^\top,\vv_1)+\widetilde{R},Q_1^\top \Big\rangle\\=&\nonumber\Big\langle-\CL(Q_1^\top)-\frac{\mu_2}{2}\DD_1+\mu_1[\BOm_1,Q_0]+\widetilde{L}(Q_1^\top,\vv_1)+\widetilde{R},Q_1^\top \Big\rangle\\\leq&~C(\widehat{\CE}+\widehat{\CE}^{\frac{1}{2}}\widehat{\CF}^{\frac{1}{2}}+\|\widetilde{R}\|_{L^2}^2).\label{nu:dt-Q1}
\end{align}
Multiplying (\ref{nu:L Q1}) by $\dot{Q}_1^\top$ and integrating over $\BR$, we have
\begin{align}
    \mu_1\|\dot{Q}_1^\top\|_{L^2}^2=\left\langle\MP^{in}\Big(-\CL(Q_1^\top)-\frac{\mu_2}{2}\DD_1+\mu_1[\BOm_1,Q_0]\Big)+\widetilde{L}(Q_1^\top,\vv_1)+\widetilde{R},\dot{Q}_1^\top\right\rangle.\label{nu:mu1-Q1*2}
\end{align}
Thanks to $\MP^{in}(\dot{Q}_1^\top)=\dot{Q}_1^\top+\widetilde{L}(Q_1^\top)$ and Lemma \ref{v*Q,L(Q)}, we get
\begin{align}
    &\left\langle\MP^{in}\Big(-\CL({Q^\top_1})-\frac{\mu_2}{2}\DD_1+\mu_1[\BOm_1,Q_0]\Big), \dot{Q}_1^{\top}\right\rangle\nonumber\\
    =&-\langle\CL({Q^\top_1})+\frac{\mu_2}{2}\DD_1-\mu_1[\BOm_1,Q_0], \MP^{in}(\dot{ Q}_1^\top)\rangle\nonumber\\
   \nonumber=&-\left\langle\CL({Q^\top_1})+\frac{\mu_2}{2}\DD_1-\mu_1[\BOm_1,Q_0], \partial_t {Q}_1^{\top}+\vv_0\cdot\nabla Q_1^\top+\widetilde{L}({Q^\top_1})\right\rangle
   \\ \leq&-\frac{1}{2}\frac{\ud}{\ud t}\left\langle\CL({Q^\top_1}), {Q}_1^{\top}\right\rangle-\langle\frac{\mu_2}{2}\DD_1-\mu_1[\BOm_1,Q_0],\dot{Q}_1^\top\rangle\nonumber\\&+C\big(\Vert {Q^\top_1}\Vert^2_{H^1}+\|Q_1^\top\|_{L^2}\|\nabla\vv_1\|_{L^2}\big).\label{nu:MP;in;CL}
\end{align}
Using the equation (\ref{nu:L Q1}), it holds that
\begin{align}
    &\langle \widetilde{L}(Q_1^\top,\vv_1)+\widetilde{R},\dot{Q}_1^\top\rangle\nonumber\\
    &\quad=\frac{1}{\mu_1}\left\langle \widetilde{L}(Q_1^\top,\vv_1)+\widetilde{R}, \MP^{in}\Big(-\CL(Q_1^\top)-\frac{\mu_2}{2}\DD_1+\mu_1[\BOm_1,Q_0]\Big)+\widetilde{L}(Q_1^\top,\vv_1)+\widetilde{R} \right\rangle\nonumber\\
    &\quad=\frac{1}{\mu_1}\Big\langle \MP^{in} \Big(\widetilde{L}(Q_1^\top,\vv_1)+\widetilde{R}\Big), -\CL(Q_1^\top)-\frac{\mu_2}{2}\DD_1+\mu_1[\BOm_1,Q_0]\Big\rangle+\frac{1}{\mu_1}\|\widetilde{L}(Q_1^\top,\vv_1)+\widetilde{R}\|_{L^2}^2\nonumber\\
    &\quad\leq C\Big(\|Q_1^\top\|_{H^1}^2+\|\vv_1\|^2_{L^2}+(\|Q_1^\top\|_{H^1}+\|\vv_1\|_{L^2})\|\vv\|_{H^1}+\|\widetilde{R}\|_{H^1}^2\Big).\label{nu:L+R}
\end{align}
From (\ref{nu:mu1-Q1*2})--(\ref{nu:L+R}), we deduce that
\begin{align}
    \frac{1}{2}\frac{\ud}{\ud t}\left\langle\CL({Q^\top_1}), {Q}_1^{\top}\right\rangle\leq& -\left\langle\frac{\mu_2}{2}\DD_1+\mu_1\dot{Q}_1^\top-\mu_1[\BOm_1,Q_0],\dot{Q}_1^\top\right\rangle\nonumber\\
    &+C(\widehat{\CE}+\widehat{\CE}^{\frac{1}{2}}\widehat{\CF}^{\frac{1}{2}}+\|\widetilde{R}\|^2_{H^1}).\label{nu:dt-CL}
\end{align}
Next, we need to estimate $\frac{1}{2}\frac{\ud}{\ud t}\Vert\vv_{1}\Vert^2_{L^2}$, it can be calculated as
\begin{align}
   \langle\vv_{1},\partial_t\vv_{1}\rangle
=&-
\Big\langle  \frac{\mu_{2}}{2} (\dot{Q}_1^{\top}-[{\BOm}_{1}, Q_{0}])
+\mu_1\big[Q_{0}, (\dot{Q}_1^{\top}-[{\BOm}_{1}, {Q}_0])\big],\nabla\vv_{1}\Big\rangle\nonumber\\&
-\Big\langle\sigma^{d}\big(Q^{\top}_1 ,  Q_{0}\big)+\sigma^{d}\big(Q^{\top}_1 ,  Q_{0}\big)
+\widetilde{L} ({Q^\top_1}, \vv_{1})
+\widetilde{R},\nabla\vv_{1}\Big\rangle\nonumber\\&-\Big\langle
    \Big\{\beta_1 Q_{0}(Q_{0}:\DD_{1})+\beta_7[(Q_{0})^2\cdot\DD_{1}+\DD_{1}\cdot (Q_{0})^2]+\beta_4\DD_{1}\nonumber\\&\quad+\beta_{5}\mathbf{D}_{1}\cdot Q_{0}+\beta_{6}Q_{0}\cdot\mathbf{D}_{1}\Big\}, \nabla\vv_{1}\Big\rangle
 -\left\langle\vv_{1}\cdot\nabla\vv_{0}, \vv_{1}\right\rangle\nonumber\\
\eqdefa&~\widehat{I}_1+\widehat{I}_2+\widehat{I}_3+\widehat{I}_4.\label{nu:dt-v1}
\end{align}
With respect to terms containing $\partial_t$ in (\ref{nu:dt-CL}) and (\ref{nu:dt-v1}), summing them leads to
\begin{align*}
&\widehat{I}_1-\left\langle\frac{\mu_2}{2}\DD_1+\mu_1\dot{Q}_1^\top-\mu_1[\BOm_1,Q_0],\dot{Q}_1^\top\right\rangle\\
&\quad=-\frac{\mu_{2}}{2}\langle\mathbf{D}_{1},  \dot{Q}_{1}^{\top}\rangle
-\frac{\mu_{2}}{2}\left\langle\dot {Q}_1^{\top}-\left[\boldsymbol{\Omega}_{1},  Q_{0}\right],  \mathbf{D}_{1}\right\rangle\\
&\qquad -\mu_{1}\left\langle\dot {Q}_1^{\top}-\left[\boldsymbol{\Omega}_{1},  Q_{0}\right],  \dot{Q}_1^{\top}\right\rangle-\mu_{1}\left\langle\left[Q_{0},  \boldsymbol{\Omega}_{1}\right],\dot{Q}_1^{\top}-\left[\boldsymbol{\Omega}_{1},  Q_{0}\right]\right\rangle \\
&\quad= -\mu_{2}\left\langle\dot {Q}_1^{\top}-\left[\boldsymbol{\Omega}_{1},  Q_{0}\right],  \mathbf{D}_{1}\right\rangle-\frac{\mu_{2}}{2}\left\langle\left[\boldsymbol{\Omega}_{1},  Q_{0}\right],  \mathbf{D}_{1}\right\rangle\\
&\qquad-\mu_{1}\left\|\dot{Q}_1^{\top}-\left[\boldsymbol{\Omega}_{1},  Q_{0}\right]\right\|_{L^{2}}^{2}  \\
&\quad= -\mu_{1}\left\|\dot{Q}_1^{\top}-\left[\boldsymbol{\Omega}_{1},  Q_{0}\right]+\frac{\mu_{2}}{2 \mu_{1}} \mathbf{D}_{1}\right\|_{L^{2}}^{2}+\frac{\mu_{2}^{2}}{4 \mu_{1}}\left\|\mathbf{D}_{1}\right\|_{L^{2}}^{2}\\
&\qquad-\frac{\mu_{2}}{2}\left\langle\left[\boldsymbol{\Omega}_{1},  Q_{0}\right],  \mathbf{D}_{1}\right\rangle  .
\end{align*}
For the terms $\widehat{I}_2$ and $\widehat{I}_4$, it can be handled as
\begin{align*}
   \widehat{I}_2+\widehat{I}_4\leq C\|\vv_{1}\|_{H^1}\big(\|\vv_{1}\|_{L^2}+\|Q^{\top}_{1}\|_{H^1}\big)+C\|\vv_1\|_{L^2}\|\widetilde{R}\|_{H^1}\leq C(\widehat{\CE}+\widehat{\CE}^{\frac{1}{2}}\widehat{\CF}^{\frac{1}{2}}+\|\widetilde{R}\|^2_{H^1}).
\end{align*}
Recalling the definition of the term $I_1+I_2$ in (\ref{I1+I2}), we can infer that
\begin{align*}
    \widehat{I}_3=&~I_1+I_2\\= &-\beta_{1} s^{2}\left\|\mathbf{n n}: \mathbf{D}_{1}\right\|_{L^{2}}^{2}-\left(\beta_{4}-\frac{s\left(\beta_{5}+\beta_{6}\right)}{3}+\frac{2}{9} \beta_{7} s^{2}\right)\left\|\mathbf{D}_{1}\right\|_{L^{2}}^{2} \\
& -\left(s\left(\beta_{5}+\beta_{6}\right)+\frac{2}{3} \beta_{7} s^{2}\right)\left\|\mathbf{n} \cdot \mathbf{D}_{1}\right\|_{L^{2}}^{2}+\frac{\mu_{2}}{2}\left\langle\left[\boldsymbol{\Omega}_{1},  Q_{0}\right],  \mathbf{D}_{1}\right\rangle\\
=&-2\delta\|\nabla\vv_1\|_{L^2}^2-\overline{\beta}_{1} \left\|\mathbf{n n}: \mathbf{D}_{1}\right\|_{L^{2}}^{2}-\Big(\overline{\beta}_2+\frac{\mu^2_2}{4\mu_1}\Big)\left\|\mathbf{D}_{1}\right\|_{L^{2}}^{2}\\  &-\overline{\beta}_3\left\|\mathbf{n} \cdot \mathbf{D}_{1}\right\|_{L^{2}}^{2}+\frac{\mu_{2}}{2}\left\langle\left[\boldsymbol{\Omega}_{1},  Q_{0}\right],  \mathbf{D}_{1}\right\rangle.
\end{align*}
Here $\delta>0$ is a small enough constant that had been determined in last section, and the coefficients $\overline{\beta}_i(i=1, 2, 3)$ are defined by (\ref{relation overlinebeta}).
From (\ref{nu:dt-Q1}), (\ref{nu:dt-CL}) and (\ref{nu:dt-v1}), we obtain
\begin{align*}
    &\frac{\ud}{\ud t}\widehat{\CE}(t)+2\delta\|\nabla\vv_1\|_{L^2}^2+\mu_{1}\big\|\dot{Q}_1^{\top}-\left[\boldsymbol{\Omega}_{1},  Q_{0}\right]+\frac{\mu_{2}}{2 \mu_{1}} \mathbf{D}_{1}\big\|_{L^{2}}^{2}\\
    &\quad\leq C\big(\widehat{\CE}+\|\widetilde{R}\|^2_{H^1}+\widehat{\CE}^{\frac{1}{2}}\widehat{\CF}^{\frac{1}{2}}\big)-\overline{\beta}_{1} \left\|\mathbf{n n}: \mathbf{D}_{1}\right\|_{L^{2}}^{2}-\overline{\beta}_2\left\|\mathbf{D}_{1}\right\|_{L^{2}}^{2}  -\overline{\beta}_3\left\|\mathbf{n} \cdot \mathbf{D}_{1}\right\|_{L^{2}}^{2}.
\end{align*}
Hence, similar to the argument of Proposition \ref{prop Q(1)}, taking $c_1=\min\{\delta,1\}>0$, we obtain the following energy estimate:
\begin{align}
    \frac{\ud}{\ud t}\widehat{\CE}(t)+c_1\widehat{\CF}(t)\leq C\big(\widehat{\CE}(t)+\|\widetilde{R}\|^2_{H^1}\big),
\end{align}
which leads to the completion of Proposition \ref{nu:prop Q(1)}.
\end{proof}

\subsection{Uniform estimates for the remainder system}

When $m\in\mathbb{Z}^+$, the remainder system {\rm (\ref{eq-QR})--(\ref{eq-vR-div})} is still hyperbolic. The main object of this subsection is to derive the uniform energy estimate for the remainder system. The whole procedure is similar to the argument of  Proposition \ref{proposition-PP}. Throughout this subsection, we assume that $(\nabla Q_k,\vv_k)\in C([0,T];H^{\ell-4k})(k=0,1,2)$ and $Q_3\in C([0,T];H^{\ell-11})$. We denote by $C$ a constant depending on $\sum^2_{k=0}\sup_{t\in[0,T]}\|\vv_k(t)\|_{\ell-4k}$ and $\sum^3_{k=0}\sup_{t\in[0,T]}\|Q_k(t)\|_{H^{\ell+1-4k}}$, and independent of $\ve$. To obtain the uniform energy estimate, the energy functionals are defined by (\ref{energyE})--(\ref{energyF}) with $m\in\mathbb{Z}^+$. The a priori estimate for the remainder term $(\vv_R,Q_R,\PP)$ is stated as follows.

\begin{proposition}\label{nu:proposition-PP}
    Let $(\vv_R,Q_R,\PP)$ be a smooth solution of the remainder system {\rm (\ref{eq-QR})--(\ref{eq-vR-div})} with $m\in\mathbb{Z}^+$ on $[0,T]$. Then for any $t\in[0,T]$, it follows that
   \[\begin{aligned}
       &{\fE}_{m}(t)-3\fE(0)+{c_1}\int_0^t\fF(\tau)\ud\tau\\&\leq CM^2\int_0^t\Big(\big(1+{\fE}_{m}(\tau)+\ve^2{\fE}_{m}^2(\tau)+\ve^6{\fE}_{m}^3(\tau)\big)+\fF(\tau)\big(\ve+\ve{\fE}_{m}^{\frac{1}{2}}(\tau)+\ve^4{\fE}_{m}(\tau)\big)\Big)\ud\tau,
   \end{aligned}\]
   where $c_1=\min\{\delta,1\}>0$, and the constant $C>0$ is independent of $(\ve,M)$.
\end{proposition}
Similar to the proof of Lemmas \ref{norm:QR,vR}--\ref{norm:vR}, we also have the following Lemma \ref{nu:norm:QR,vR} and Lemma \ref{nu:norm:FR}.

\begin{lemma} \label{nu:norm:QR,vR}
There exists a positive constant $C$ independent of $(\ve,M)$, such that
\begin{align*}
  & \|Q_R\|_{H^1}+\|(\ve\nabla^2Q_R,\ve^2\nabla^3Q_R)\|_{L^2}+\|(\vv_R,\ve\nabla\vv_R,\ve^2\nabla^2\vv_R)\|_{L^2}\leq C{\fE}_{m}^{\frac{1}{2}},\\
    &\|(\nabla\vv_R,\ve\nabla^2\vv_R,\ve^2\nabla^3\vv_R\|_{L^2}+\|(\PP,\ve\nabla\PP,\ve^2\nabla^2\PP\|_{L^2}\leq C\fF^{\frac{1}{2}},\\
    &\|(\dot{Q}_R,\ve\nabla\dot{Q}_R,\ve^2\nabla^2\dot{Q}_R)\|_{L^2}\leq C({\fE}_{m}^{\frac{1}{2}}+\ve{\fE}_{m}+\fF^{\frac{1}{2}}),\\
    &\ve^{\frac{m}{2}}\|(\dot{Q}_R,\ve\nabla\dot{Q}_R,\ve^2\nabla^2\dot{Q}_R)\|_{L^2}\leq C({\fE}_{m}^{\frac{1}{2}}+\ve\fE_m),\\
&\ve^{\frac{m}{2}}\|(\PP,\ve\nabla\PP,\ve^2\nabla^2\PP)\|_{L^2}\leq C{\fE}_{m}^{\frac{1}{2}}.
\end{align*}
\end{lemma}

\begin{proof}
    Armed with (\ref{fg Hk}), it is easy to estimate that
    \begin{align*}
      \|(\PP,\ve\nabla\PP,\ve^2\Delta\PP)\|_{L^2}\leq& C(\|\ve^k{\UU}_k\|_{L^2}+\|\ve^k\nabla\vv_R\|_{H^k})\leq C\fF^{\frac{1}{2}},\\
\|(\dot{Q}_R,\ve\nabla\dot{Q}_R,\ve^2\Delta\dot{Q}_R)\|_{L^2}\leq &C(\|\ve^k\PP\|_{H^k}+\|\ve^k\vv_R\|_{H^k}\|\nabla Q^\ve\|_{H^2})\\
\leq& C({\fE}_{m}^{\frac{1}{2}}+\ve{\fE}_{m}+\fF^{\frac{1}{2}}),\\
\ve^{\frac{m}{2}}\|(\dot{Q}_R,\ve\nabla\dot{Q}_R,\ve^2\Delta\dot{Q}_R)\|_{L^2}\leq &C(\ve^{\frac{m}{2}}\|\ve^k\PP\|_{H^k}+\|\ve^k\vv_R\|_{H^k}\|\nabla Q^\ve\|_{H^2})\\
\leq& C({\fE}_{m}^{\frac{1}{2}}+\ve{\fE}_{m}),
    \end{align*}
which imply immediately that the lemma follows.
\end{proof}

\begin{lemma}\label{nu:norm:FR}
For the remainder terms $\FF_R, \GG_R$ and $\GG_R'$, it follows that
    \begin{align*}
        \|(\FF_R,\ve\nabla\FF_R,\ve^2\Delta\FF_R)\|_{L^2}\leq&~ C\Big(1+{\fE}_{m}^{\frac{1}{2}}+\ve{\fE}_{m}+\ve^3{\fE}_{m}^{\frac{3}{2}}+\ve\fF^{\frac{1}{2}}+\ve{\fE}_{m}^{\frac{1}{2}}\fF^{\frac{1}{2}}
    \Big),\\
\|(\GG_R,\ve\nabla\GG_R,\ve^2\Delta\GG_R)\|_{L^2}\leq&~C(1+{\fE}_{m}^{\frac{1}{2}}+\ve^2{\fE}_{m}+\ve\fF^{\frac{1}{2}}+\ve^2{\fE}_{m}^{\frac{1}{2}}\fF^{\frac{1}{2}}+\ve^4{\fE}_{m}\fF^{\frac{1}{2}}),\\
\|(\GG_R',\ve\nabla\GG_R',\ve^2\Delta\GG_R')\|_{L^2}\leq&~C(1+{\fE}_{m}^{\frac{1}{2}}+\fF^{\frac{1}{2}}+\ve{\fE}_{m}^{\frac{1}{2}}\fF^{\frac{1}{2}}).
    \end{align*}
\end{lemma}
Lemma \ref{nu:norm:FR} is an analogue of Lemmas \ref{norm:FR}--\ref{norm:vR}, whose proofs are based on (\ref{fg Hk}) and the definitions of functionals (\ref{energyE})--(\ref{energyF}).

To deal with the singular term $\ve^{-1}\langle\CH^{\ve}_{\nn}(Q_R),\PP\rangle$, where $\PP=\dot{Q}_R+\vv_R\cdot\nabla Q^\ve$, we also need to establish the following key estimates:

\begin{lemma}\label{nu:A key lemma}
Assume that $(\vv_R,Q_R,\PP)$ is a smooth solution of the remainder system {\rm (\ref{eq-QR})--(\ref{eq-vR-div})} with $m\in\mathbb{Z}^+$. Then there exists a positive constant $C$ depending on $\nn,\nabla_{t,x}\nn,\tilde{\vv}$ and $\widetilde{Q}$, such that
\begin{align}
- \langle\frac{1}{\ve}\CH_{\nn}^\ve({Q}_R),\PP\rangle\leq &~\frac{\ud}{\ud t}\Big(-\frac{1}{2\ve}\langle\CH_\nn^\ve(Q_R),Q_R\rangle-\ve^m\CA(Q_R,\PP)\Big)\nonumber\\&+C({\fE}_{m}+\ve{\fE}_{m}^{\frac{3}{2}}+{\fE}_{m}^{\frac{1}{2}}\fF^{\frac{1}{2}}+\ve{\fE}_{m}\fF^{\frac{1}{2}}+{\fE}_{m}^{\frac{1}{2}}\|\FF_R\|_{L^2}),\label{nu:CH,Q0}\\
    -\ve\langle\partial_i\CH_{\nn}^{\ve}(Q_R),\partial_i\PP\rangle
    \leq&-\frac{\ve}{2}\frac{\ud}{\ud t}\big\langle\CH_{\nn}^\ve(\partial_iQ_R),\partial_i{Q}_R\big\rangle+C({\fE}_{m}+\ve{\fE}_{m}^{\frac{3}{2}}+{\fE}_{m}^{\frac{1}{2}}\fF^{\frac{1}{2}}+\ve{\fE}_{m}\fF^{\frac{1}{2}}),\label{nu:CH,Q1}\\
     -\ve^3\langle\Delta\CH_{\nn}^{\ve}(Q_R),\Delta\PP\rangle
     \leq&-\frac{\ve^3}{2}\frac{\ud}{\ud t}\big\langle\CH_{\nn}^\ve(\Delta Q_R),\Delta Q_R\big\rangle+C\big({\fE}_{m}+\ve{\fE}_{m}^{\frac{3}{2}}+{\fE}_{m}^{\frac{1}{2}}\fF^{\frac{1}{2}}+\ve{\fE}_{m}\fF^{\frac{1}{2}}\big),\label{nu:CH,Q2}
\end{align}
where $\CA(Q_R,\PP)$ is defined by {\em(\ref{CA-t})} and  $\FF_R$ is given by {\em (\ref{FFR-remaider-term})}.
\end{lemma}

\begin{proof}
Applying the definition of $\PP$, we have
\begin{align*}
    -\Big\langle\frac{1}{\ve}\CH_{\nn}^\ve({Q}_R),\PP\Big\rangle=&~\Big\langle-\vv_R\cdot\nabla Q^\ve, \frac{1}{\ve}\CH_\nn^\ve(Q_R)\Big\rangle+\Big\langle-\dot{Q}_R, \frac{1}{\ve}\CH_\nn^\ve(Q_R)\Big\rangle\\
    \eqdefa&~W_0+S_0,\\
    -\ve\langle\partial_i\CH_{\nn}^\ve({Q}_R),\partial_i\PP\rangle=&~\ve\langle-\partial_i(\vv_R\cdot\nabla Q^\ve), \partial_i\CH_\nn^\ve(Q_R)\rangle+\ve\langle-\partial_i\dot{Q}_R, \partial_i\CH_\nn^\ve(Q_R)\rangle\\
    \eqdefa&~W_1+S_1,\\
    -\ve^3\langle\Delta\CH_{\nn}^\ve({Q}_R),\Delta\PP\rangle=&~\ve^3\langle-\Delta(\vv_R\cdot\nabla Q^\ve),\Delta\CH_\nn^\ve(Q_R)\rangle+\ve^3\langle-\Delta\dot{Q}_R,\Delta\CH_\nn^\ve(Q_R)\rangle\\
    \eqdefa&~W_2+S_2.
\end{align*}
When replacing $\fE$ by $\fE_m$ in (\ref{vQ1}), (\ref{vQ2}) and (\ref{vQ3}), respectively,
it follows that
\begin{align}
    &W_0\leq C\big(\fE_m+\fE_m^\frac{1}{2}\fF^{\frac{1}{2}}+\ve\fE_m\fF^{\frac{1}{2}}\big),\label{nu:SD0}\\
    &W_1\leq C\big(\fE_m+\ve\fE_m^{\frac{3}{2}}+\fE_m^\frac{1}{2}\fF^{\frac{1}{2}}+\ve\fE_m\fF^{\frac{1}{2}}\big),\label{nu:SD1}\\
    &W_2\leq C\big( \fE_m+\ve\fE_m^{\frac{3}{2}}+\fE_m^\frac{1}{2}\fF^{\frac{1}{2}}+\ve\fE_m\fF^{\frac{1}{2}}\big).\label{nu:SD2}
\end{align}
By means of the equation (\ref{eq-QR}) and the derivation of (\ref{Q,CH1}), the term $S_0$ can be controlled as
\begin{align}
  S_0 =&\nonumber-\ve^m\frac{\ud}{\ud t}\underbrace{{\langle\CH_\nn^{-1}(\mathbf{A}(Q_R^\top)),  J\PP\rangle}}_{\CA(Q_R,\PP)}+ \ve^m J\langle(\partial_t+{\vv}^\ve\cdot\nabla)\CH_\nn^{-1}(\mathbf{A}(Q_R^\top),\PP\rangle\\
    &\nonumber+\Big\langle\CH_\nn^{-1}(\mathbf{A}(Q_R^\top)),\mu_1{\UU}_0+\CL(Q_R)-\FF_R\Big\rangle-\langle\CL(Q_R),\tilde{\vv}\cdot\nabla Q_R\rangle\\
    &\nonumber-\frac{1}{\ve}\langle bs\dot{\overline{\nn\nn}}\cdot Q_R^\perp,Q_R^\perp\rangle-\frac{1}{2\ve}\frac{\ud}{\ud t}\langle\CH^{\ve}_{\nn}(Q_R),Q_R\rangle
    \\\leq&\nonumber-\frac{\ud}{\ud t}\Big(\ve^m\CA(Q_R,\PP)+\frac{1}{2\ve}\langle\CH^{\ve}_{\nn}(Q_R),Q_R\rangle\Big)+C{\ve}^{-1}\|Q_R^\perp\|_{L^2}^2\\
    &\nonumber+C  \ve^m\|\PP\|_{L^2}\big(\|Q_R\|_{L^2}+\|\dot{Q}_R\|_{L^2}+\ve^3\|\vv_R\|_{H^2}\|Q_R\|_{H^1}\big)\\&\nonumber+C\|Q_R\|_{H^1}\big(\|{\UU}_0\|_{L^2}+\| Q_R\|_{H^1}+\|\FF_R\|_{L^2}\big)\\
    \leq&\nonumber-\frac{\ud}{\ud t}\Big(\ve^m\CA(Q_R,\PP)+\frac{1}{2\ve}\langle\CH^{\ve}_{\nn}(Q_R),Q_R\rangle\Big)\\
    &+C(\fE_m+\ve{\fE}_{m}^{\frac{3}{2}}+\ve{\fE}_{m}^{\frac{1}{2}}\fF^{\frac{1}{2}}+{\fE}_{m}^{\frac{1}{2}}\|\FF_R\|_{L^2}),\label{nu:Q,CH1}
\end{align}
where $\mathbf{A}(Q)\eqdefa 2bs \dot{\overline{\nn\nn}} \cdot Q-2cs^2(\dot{\overline{\nn\nn}}: Q)(\mathbf{n n}-\frac{1}{3}\II)$ and $\dot{\overline{\nn\nn}}=(\partial_t+\tilde{\vv}\cdot\nabla)(\nn\nn)$, and we have used the fact that $C_0\ve^{-1}\|Q_R^\perp\|_{L^2}^2 \leq \ve^{-1}\langle\CH_{\nn}(Q_R),Q_R\rangle$ with $C_0=C_0(a,b,c)>0$.
Similar to the derivations of  (\ref{Q,CH2}) and (\ref{Q,CH3}), we have
\begin{align}
   S_1=&-\frac{\ve}{2}\frac{\ud}{\ud t}\Big\langle\CH_{\nn}^\ve(\partial_iQ_R),\partial_i{Q}_R\Big\rangle-\ve\Big\langle\CH_\nn^\ve(\partial_iQ_R),\partial_i\tilde{\vv}\cdot\nabla Q_R\Big\rangle\nonumber\\
   &-\ve^2\langle\CL(\partial_iQ),(\tilde{\vv}\cdot\nabla)\partial_iQ_R\rangle+\ve\Big\langle[\CH_\nn,\partial_i]Q_R,\partial_i\dot{Q}_R\Big\rangle\nonumber\\
   &+\frac{\ve}{2}\langle[\partial_t+\tilde{\vv}\cdot\nabla,\CH_\nn]\partial_iQ_R,\partial_i{Q}_R\rangle
\nonumber\\
\leq&-\frac{\ve}{2}\frac{\ud}{\ud t}\Big\langle\CH_{\nn}^\ve(\partial_iQ_R),\partial_i{Q}_R\Big\rangle+C({\fE}_{m}+\ve{\fE}_{m}^{\frac{3}{2}}+\fE_{m}^{\frac{1}{2}}\fF^{\frac{1}{2}}),\label{nu:Q,CH2}\end{align}
where $\tilde{\vv}=\vv_0+\ve\vv_1+\ve^2\vv_2$, and
\begin{align}
S_2
=&-\frac{\ve^3}{2}\frac{\ud}{\ud t}\Big\langle\CH_{\nn}^\ve(\Delta Q_R),\Delta {Q}_R\Big\rangle-\ve^3\Big\langle\CH_{\nn}(\Delta Q_R),[\Delta,\Tilde{\vv}\cdot\nabla ]Q_R\Big\rangle\nonumber\\
&+\ve^4\Big\langle\CL(\partial_i Q_R),\partial_i\big([\Delta,\Tilde{\vv}\cdot\nabla ]Q_R\big)\Big\rangle-\ve^4\langle\CL(\Delta Q_R),\tilde{\vv}\cdot\nabla\Delta Q_R\rangle\nonumber\\
&+\ve^3\Big\langle[\CH_\nn,\Delta]Q_R,\Delta\Dot{Q}_R\Big\rangle+\frac{\ve^3}{2}\langle[\partial_t+\Tilde{\vv}\cdot\nabla,\CH_\nn]\Delta Q_R,\Delta Q_R\rangle\nonumber\\
\leq&-\frac{\ve^3}{2}\frac{\ud}{\ud t}\Big\langle\CH_{\nn}^\ve(\Delta Q_R),\Delta{Q}_R\Big\rangle+C({\fE}_{m}+\ve^2{\fE}_{m}^{\frac{3}{2}}+\ve\fE^{\frac{1}{2}}_{m}\fF^{\frac{1}{2}}).\label{nu:Q,CH3}
\end{align}
Hence, combining (\ref{nu:SD0})--(\ref{nu:SD2}) with (\ref{nu:Q,CH1})--(\ref{nu:Q,CH3}), we complete the proof of the lemma.
\end{proof}

Next, we prove Proposition \ref{nu:proposition-PP}, which is split into three steps.

\begin{proof}[Proof of Proposition \ref{nu:proposition-PP}]
{\it Step 1. $L^2$-estimate.}
Multiplying (\ref{eq-QR}) by $\PP$ and (\ref{eq-vR}) by $\vv_R$, respectively, and integrating by parts over the space $\BR$, we then have
\begin{align*}
    \langle\vv_R,\partial_t\vv_R\rangle+\ve^m   J\langle\PP,\partial_t\PP\rangle
=&\sum_{k=1}^7J_k+\langle\nabla\cdot \GG_R+\GG'_R,\vv_R\rangle+\langle\FF_R,\PP\rangle\\
    \leq&-c_1\fF+\|(\GG_R,\FF_R)\|_{L^2}\fF^{\frac{1}{2}}+\|\GG'_R\|_{L^2}{\fE}_{m}^{\frac{1}{2}},
\end{align*}
where the terms $J_k(k=0,\cdots,7)$ have been defined in (\ref{L^2}).
Using $\nabla\cdot\vv^\ve=0$, there holds
\begin{align*}
    \frac{M}{2}\frac{\ud}{\ud t}\|Q_R\|_{L^2}^2=&M\langle Q_R,(\PP-\vv_R\cdot\nabla \widetilde{Q})\rangle-M\langle Q_R,\vv^\ve\cdot\nabla Q_R\rangle\\
    \leq& CM({\fE}_{m}+{\fE}_{m}^{\frac{1}{2}}\fF^{\frac{1}{2}}),
\end{align*}
where $\widetilde{Q}=\sum\limits^3_{k=0}\ve^kQ_k$.
Combining (\ref{nu:CH,Q0}) with the above estimates, we get
\begin{align}
    &\frac{\ud}{\ud t}\big({\fE}_{m,0}(t)+\ve^m\CA(Q_R,\PP)\big)+c_1\fF_0\nonumber\\
    &\quad\leq CM({\fE}+\ve{\fE}^{\frac{3}{2}}+{\fE}^{\frac{1}{2}}\fF^{\frac{1}{2}}+\ve{\fE}\fF^{\frac{1}{2}})\nonumber\\
    &\qquad+\|(\GG_R,\FF_R)\|_{L^2}\fF^{\frac{1}{2}}+\|\GG_R'\|_{L^2}{\fE}_{m}^{\frac{1}{2}}+C\|\FF_R\|_{L^2}{\fE}_{m}^{\frac{1}{2}}.\label{nu:L2-estimate}
\end{align}

{\it{Step 2. $H^1$-estimate}.}
We apply the derivative $\partial_i$ on (\ref{eq-QR}) and take the $L^2$-inner product with $\partial_i\PP$. Again by acting $\partial_i$ on (\ref{eq-vR}) and
taking the $L^2$-inner product with $\partial_i\vv_R$, we obtain
\begin{align}
{\ve^2}\langle\partial_i&\vv_R,\partial_t(\partial_i\vv_R)\rangle+\ve^{2+m}J\langle\partial_i\PP,\partial_t(\partial_i\PP)\rangle\nonumber\\
   = &~\sum_{i=1}^7K_i+\ve^{2+m}J\langle\partial_i\vv^\ve\cdot\nabla\PP,\partial_i\PP\rangle-\langle\ve^2\partial_i \GG_R,\nabla\partial_i\vv_R\rangle\nonumber\\
&+\langle\ve^2\partial_i\GG_R',\partial_i\vv_R\rangle+\ve^2\langle\partial_i\FF_R,\partial_i\PP\rangle\nonumber\\
    \leq &~\sum_{i=1}^7K_i+C\Big((1+\ve\fF^{\frac{1}{2}}){\fE}_{m}+\ve\|(\partial_i\GG_R,\partial_i\FF_R)\|_{L^2}\fF^{\frac{1}{2}}+\ve\|\partial_i\GG_R'\|_{L^2}{\fE}_{m}^{\frac{1}{2}}\Big),\label{nu:H^1}
\end{align}
where the terms $K_i(i=1,\cdots,7)$ have been defined in (\ref{H^1}).
Similar to the derivation of (\ref{T12})--(\ref{T1-6}), it follows that
\begin{align*}
\sum_{i=1}^6 K_i\leq & ~C\big(\|\ve\nabla\vv_R\|_{L^2}+\ve\|\PP-[\BOm_R,Q_0]\|_{L^2}+\|\ve[\BOm_R,\partial_iQ_0]\|_{L^2}\big)\|\ve\nabla\partial_i\vv_R\|_{L^2}\\&+C\|\ve\nabla\vv_R\|_{L^2}\|\ve\partial_i\PP\|_{L^2}-c_1\fF_1\\\leq&-c_1\fF_1+C(\ve\fF+{\fE}_{m}^{\frac{1}{2}}\fF^{\frac{1}{2}}),
\end{align*}
which together with (\ref{nu:CH,Q1}) and (\ref{nu:H^1}) yields
\begin{align}
    \frac{\ud}{\ud t}{\fE}_{m,1}(t)+c_1\fF_1(t)\leq &C\Big( {\fE}_{m}+ \ve{\fE}_{m}^{\frac{3}{2}}+\ve\fF+{\fE}_{m}^{\frac{1}{2}}\fF^{\frac{1}{2}}+\ve{\fE}_{m}\fF^{\frac{1}{2}}\nonumber\\
    &\quad+\ve\|(\partial_i\GG_R,\partial_i\FF_R)\|_{L^2}\fF^{\frac{1}{2}}+\|\ve\partial_i\GG_R'\|_{L^2}{\fE}_{m}^{\frac{1}{2}}\Big).\label{nu:H1-estimate}
\end{align}

{\it {Step 3. $H^2$-estimate.}}
Similarly, from the remainder system {\rm (\ref{eq-QR})--(\ref{eq-vR-div})} with $m\in\mathbb{Z}^+$, one can deduce that
\begin{align*}
    \frac{\ve^4}{2}\frac{\ud}{\ud t}\Big\{\|\Delta\vv_R\|_{L^2}^2+\ve^m   J\|\Delta\PP\|_{L^2}^2\Big\}
    \leq &\sum_{k=1}^7V_k-\ve^{4+m}J\langle\Delta(\vv^\ve\cdot\nabla\PP),\Delta \PP\rangle\\&+C\left(\ve^2\|(\Delta\GG_R,\Delta\FF_R)\|_{L^2}\fF^{\frac{1}{2}}+\ve^2\|\Delta\GG_R'\|_{L^2}{\fE}_{m}^{\frac{1}{2}}\right),
\end{align*}
where the terms $V_k(k=1,\cdots,7)$ have been given by (\ref{H^2}).
Due to $\nabla\cdot\vv^\ve=0$ with $\vv^\ve=(v^\ve_1,v^\ve_2,v^\ve_3)$, it follows that
\begin{align*}
    -\ve^{4+m}J&\langle\Delta(\vv^\ve\cdot\nabla\PP),\Delta \PP\rangle\\=&-\ve^{4+m}J\big\langle[\Delta,\vv^\ve\cdot\nabla]\PP,\Delta \PP\big\rangle\\
    =&-\ve^{4+m}J\langle\Delta\vv^\ve\cdot\nabla\PP,\Delta \PP\rangle-\ve^{4+m}J\langle(\partial_i\vv^\ve\cdot\nabla\partial_i\PP),\Delta \PP\rangle\\
    =&-\ve^{4+m}J\langle\partial_i(\Delta v^\ve_i\PP),\Delta \PP\rangle-\ve^{4+m}J\langle(\partial_i\vv^\ve\cdot\nabla\partial_i\PP),\Delta \PP\rangle\\
\leq&~C\ve^{4+m}\big(\|\Delta\vv^\ve\|_{H^1}\|\PP\|_{H^2}+\|\partial_i\vv^\ve\|_{H^2}\|\nabla^2\PP\|_{L^2}\big)\|\Delta\PP\|_{L^2}\\
\leq&~C({\fE}_{m}+\ve{\fE}_{m}\fF^{\frac{1}{2}}).
\end{align*}
Recalling the derivation of (\ref{MT1-6}), it can be seen that
\begin{align*}
    \sum_{k=1}^6 V_k\leq&- c_1\fF_2+C\ve^4\Big(\|\nabla\Delta\vv_R\|_{L^2}\big(\|\nabla\vv_R\|_{H^1}+\|\PP-[\BOm_R,Q_0]\|_{H^1}\\&+\|[\BOm_R,\partial_iQ_0]\|_{H^1}\big)+\|\nabla\vv_R\|_{H^1}\|\Delta\PP\|_{L^2}\Big)\\
    \leq&-c_1\fF_2+C(\ve\fF+{\fE}_{m}^{\frac{1}{2}}\fF^{\frac{1}{2}}).
\end{align*}
Using Lemma \ref{nu:A key lemma} and the above estimates, we have
\begin{align}
    \frac{\ud}{\ud t}{\fE}_{m,2}(t)+c_1\fF_2(t)\leq&~C\Big({\fE}_{m}+\ve{\fE}_{m}^{\frac{3}{2}}+\ve\fF+{\fE}_{m}^{\frac{1}{2}}\fF^{\frac{1}{2}}+\ve{\fE}_{m}\fF^{\frac{1}{2}}\nonumber\\&+\ve^2\|(\Delta\GG_R,\Delta\FF_R\|_{L^2}\fF^{\frac{1}{2}}
    +\ve^2{\fE}_{m}^{\frac{1}{2}}\|\Delta\GG_R'\|_{L^2}\Big).\label{nu:H2-estimate}
\end{align}
From (\ref{nu:L2-estimate}), (\ref{nu:H1-estimate}) and (\ref{nu:H2-estimate}), we infer that
\begin{align}\label{ve-m-ef-CA}
    &\frac{\ud}{\ud t}\big({\fE}_{m}(t)+\ve^m\CA(Q_R,\PP)\big)+{c_1}\fF(t)\nonumber\\
    &\quad\leq CM\Big({\fE}_{m}+\ve{\fE}_{m}^{\frac{3}{2}}+\ve\fF+{\fE}_{m}^{\frac{1}{2}}\fF^{\frac{1}{2}}+\ve{\fE}_{m}\fF^{\frac{1}{2}}\Big)
    +C{\fE}_{m}^{\frac{1}{2}}\|\FF_R\|_{L^2}\nonumber\\
    &\qquad+C\sum_{k=0,1,2}\ve^{2k}\Big(\fF^{\frac{1}{2}}\|(\partial_i^k\GG_R,\partial_i^k\FF_R)\|_{L^2}+{\fE}_{m}^{\frac{1}{2}}\|\partial_i^k\GG_R'\|_{L^2}\Big)\nonumber\\
    &\quad\leq CM\Big((1+{\fE}_{m}+\ve{\fE}_{m}^{\frac{3}{2}}+\ve^3{\fE}_{m}^2)+\fF^{\frac{1}{2}}(1+{\fE}^{\frac{1}{2}}+\ve{\fE}_{m}+\ve^3{\fE}_{m}^{\frac{3}{2}})\nonumber\\
    &\qquad+\fF(\ve+\ve{\fE}_{m}^{\frac{1}{2}}+\ve^4{\fE}_{m})\Big)\nonumber\\
    &\quad\leq CM^2\Big((1+{\fE}_{m}+\ve^2{\fE}_{m}^2+\ve^6{\fE}_{m}^3)+\fF(\ve+\ve{\fE}_{m}^{\frac{1}{2}}+\ve^4{\fE}_{m})\Big)+\frac{c_1}{2}\fF.
\end{align}
When the positive constant $M\geq 1$ is large enough, the term $\ve^m|\CA(Q_R,\PP)|$ can be controlled as
\begin{align*}
    \ve^m|\CA(Q_R,\PP)|\leq& \ve^m C(\nn,\tilde{\vv})   \|\PP\|_{L^2}\|Q_R\|_{L^2}\\
    \leq& \frac{1}{4}(\ve^m J\|\PP\|_{L^2}^2+C_2(\nn,\tilde{\vv})\|Q_R\|_{L^2}^2)\leq \frac{1}{2}{\fE}_{m,0}(t).
\end{align*}
Therefore, there exists a positive constant $C$ not depending on $(\ve, M)$, such that (\ref{ve-m-ef-CA}) becomes
  \[\begin{aligned}
       &{\fE}_{m}(t)-3\fE(0)+{c_1}\int_0^t\fF(\tau)\ud\tau\\&\leq CM^2\int_0^t\Big(\big(1+{\fE}_{m}(\tau)+\ve^2{\fE}_{m}^2(\tau)+\ve^6{\fE}_{m}^3(\tau)\big)+\fF(\tau)\big(\ve+\ve{\fE}_{m}^{\frac{1}{2}}(\tau)+\ve^4{\fE}_{m}(\tau)\big)\Big)\ud\tau,
   \end{aligned}\]
where $M\geq \max\{1,C_2(\nn,\tilde{\vv})\}$. Then the proposition follows.
\end{proof}

\section{Proof of Theorem \ref{main theorem}}

This section concentrates on the proof Theorem \ref{main theorem}. Although Theorem \ref{main theorem} implies two different results, it can be proved in a unified way. To control the singular terms of the remainder system and close the energy estimates, two functionals have been introduced. Before proving the Theorem \ref{main theorem}, we need to show that ${\fE}_{m}(t)$ defined by $(\ref{energyE})$ and $\widetilde{{\fE}}_{m}(t)$ defined by $(\ref{definition of widetile(fE)})$ can be controlled by each other, where $m$ is a nonnegative integer. Specifically, we have the following lemma.

\begin{lemma}\label{fE-tilde(fE)}
     There exists constants $C'>0$ and $C''>0$ such that
    \begin{align*}
        {\fE}_{m}(t)\leq C'(\widetilde{{\fE}}_{m}(t)+\ve^2\widetilde{{\fE}}_{m}^2(t)),~~\widetilde{{\fE}}_{m}(t)\leq C''({\fE}_{m}(t)+\ve^2{\fE}_{m}^2(t)).
    \end{align*}
\end{lemma}
\begin{proof}
It suffices to estimate the different terms related to the inertial constant $J$ in the functionals ${\fE}_{m}(t)$ and $\widetilde{{\fE}}_{m}(t)$. Using the definition of $\PP$, it follows that
\begin{align*}
    \sum_{k=0}^2\ve^{k+\frac{m}{2}}\|\partial_i^k\PP\|_{L^2} =& \sum_{k=0}^2\ve^{k+\frac{m}{2}}\|\partial_i^k(\partial_t Q_R+\tilde{\vv}\cdot\nabla Q_R+\vv_R\cdot\nabla Q^\ve)\|_{L^2}\\\leq& ~C \sum_{m=0}^2\ve^{k+\frac{m}{2}}\Big(\|\partial_tQ_R\|_{H^k}+\|\nabla Q_R\|_{H^k}+\|\vv_R\|_{H^k}\|\nabla Q^\ve\|_{H^2}\Big)\\
    \leq &C(\widetilde{{\fE}}_{m}^{\frac{1}{2}}+\ve\widetilde{{\fE}}_{m}),\\
     \sum_{k=0}^2\ve^{k+\frac{m}{2}}\|\partial_i^k\partial_t Q_R\|_{L^2} =& \sum_{k=0}^2\ve^{k+\frac{m}{2}}\|\partial_i^k(\PP-\tilde{\vv}\cdot\nabla Q_R-\vv_R\cdot\nabla Q^\ve)\|_{L^2}
     \\\leq&~ C({\fE}_{m}^{\frac{1}{2}}+\ve{\fE}_{m}),
\end{align*}
which together with the mean inequality leads to the proof of the lemma.
\end{proof}

It remains to complete the proof of Theorem \ref{main theorem}. We will show that ${\fE}_{m}(t)$ is uniformly bounded in $[0,T]$, and thus $\widetilde{{\fE}}_{m}(t)$ is uniformly bounded in $[0,T]$.

\begin{proof}[Proof of Theorem \ref{main theorem}]

Recalling the initial condition (\ref{initial condition}) and Lemma \ref{fE-tilde(fE)}, we get
\[{\fE}_{m}(0)\leq C'\big(\widetilde{{\fE}}_{m}(0)+\ve^2
\widetilde{{\fE}}_{m}^2(0)\big)\leq C'C(\nn)(E_0^2+E_0^4)\eqdefa \widetilde{E}_{m},\]
where $C(\nn)$ only depends on $\nn$, and $m$ is a nonnegative integer. It can be showed by the energy method in \cite{DZ} that there exists a maximal time $T_\ve>0$  dependent of $(\ve,m)$ and a unique solution $(\vv^\ve,Q^\ve)$ of the system (\ref{eq:Q})--(\ref{eq:free div}) such that \[(\nabla Q^\ve,\partial_tQ^\ve)\in L^\infty([0,T_\ve);H^2)\cap L^2(0,T_\ve;H^2),\quad\vv^\ve\in L^\infty([0,T_\ve);H^2)\cap L^2(0,T_\ve;H^3).\]
Now we prove that $T\leq T_{\ve}$. Suppose it is not true. From Proposition \ref{proposition-PP} and Proposition \ref{nu:proposition-PP}, we infer that
\[\begin{aligned}
   &{\fE}_{m}(t)-3\fE_m(0)+c_1\int^t_0\fF(\tau)\ud\tau\\
   &\quad\leq C\int_0^t\Big(\big(1+{\fE}_{m}(\tau)+\ve^2{\fE}_{m}^2(\tau)+\ve^6{\fE}_{m}^3(\tau)\big)+\fF(\tau)\big(\ve+\ve{\fE}_{m}^{\frac{1}{2}}(\tau)+\ve^4{\fE}_{m}(\tau)\big)\Big)\ud\tau,
\end{aligned}\]
for any $t\in [0,T_\ve]$ and $m\in \mathbb{N}$ .
Let $\widehat{E}=(1+3\widetilde{E}_m)(1+2CT e^{2CT})>{{\fE}_m}(0)$, and
\[T_1=\sup\{t\in[0,T_\ve]:{\fE}_m(t)\leq \widehat{E}\}.\]
If we take $\ve_0$ small enough such that
\[\ve_0^2\widehat{E}\leq 1,~~C(\ve_0+\ve_0\widehat{E}^\frac{1}{2}+\ve_0^4\widehat{E}^4)\leq \frac{c_1}{2},\]
then for $t\leq T_1$, it holds that
\[{\fE}_{m}(t)\leq 2C\int_0^t(1+{\fE}_{m}(\tau))\ud\tau+3\fE_m(0).\]
If $T_\ve<T$, then the Gronwall inequality yields that for $t \leq T_1$,
\[{\fE}_{m}(t)\leq (1+3\widetilde{E})(1+2CT e^{2CT})-1<\widehat{E},\]
which implies $T_1=T_\ve$ and ${\fE}_{m}(T_\ve)<+\infty
($\text{i.e.}$,\widetilde{{\fE}}_{m}(T_\ve)<+\infty)$. From the definition of $\fE_m$, we obtain $(\ve^3\vv_R,\ve^3Q_R)\in H^2\times H^3$ in time $T_\ve$. Using Proposition \ref{prop Q(1)} for $m=0$ (or Proposition \ref{nu:prop Q(1)} for $m\in\mathbb{Z}^+$), we get $(\vv^\ve(T_\ve,\cdot),Q^\ve(T_\ve,\cdot))\in H^2\times H^3$, which contradicts our assumption. Thus $T\leq T_\ve$, and ${\fE}_{m}(t)\leq \widehat{E}$ for $t\in[0,T]$. Using Lemma \ref{fE-tilde(fE)} again, it holds that
\[
\widetilde{{\fE}}_{m}(t)\leq C''({\fE}_{m}(t)+\ve^2{\fE}_{m}^2(t))\leq C''(\widehat{E}+\widehat{E}^2)\eqdefa E_1,\quad t\in[0,T].
\]
Therefore, the proof of Theorem \ref{main theorem} is completed.
\end{proof}

\appendix
\section{Appendix}\label{CH,vv*Q}

\subsection{A useful lemma}
\begin{lemma}\label{v*Q,L(Q)}
    If $(\vv,Q)\in H^3\times H^2$ and $\nabla\cdot\vv=0$, then for any  $Q\in\mathbb{S}^3_0$, it follows that
    \[-\langle\vv\cdot\nabla Q,\CL(Q)\rangle \leq C\|\nabla\vv\|_{H^2}\|Q\|_{H^1}^2,\]
where $\CL(Q)$ is defined by {\em(\ref{CL-LP})}.
\end{lemma}

\begin{proof}
Applying $\nabla\cdot\vv=0$ and integrating by parts, we deduce that
\begin{align*}
   - &\langle\CL(Q),\vv\cdot\nabla Q\rangle\\
    &=\int_{\BR}\vv_jQ_{kl,j}\Big(L_1\Delta Q_{kl}+\frac{1}{2}(L_2+L_3)\Big(Q_{km,ml}+Q_{lm,mk}-\frac{2}{3}\delta_{kl}Q_{ip,ip}\Big)\Big)\ud\xx
    \\
&=\int_{\BR}-\Big(L_1\vv_jQ_{kl,mj}Q_{kl,m}+\frac{1}{2}(L_2+L_3)\vv_j(Q_{kl,lj}Q_{km,m}+Q_{kl,kj}Q_{lm,m})\\
    &\qquad+L_1\vv_{j,m}Q_{kl,j}Q_{kl,m}+\frac{1}{2}(L_2+L_3)(\vv_{j,l}Q_{kl,j}Q_{km,m}+\vv_{j,k}Q_{kl,j}Q_{lm,m})\Big)\ud\xx\\
&=\int_{\BR}-\Big(L_1\vv_{j,m}Q_{kl,j}Q_{kl,m}+\frac{1}{2}(L_2+L_3)(\vv_{j,l}Q_{kl,j}Q_{km,m}+\vv_{j,k}Q_{kl,j}Q_{lm,m}\Big)\ud\xx\\
&\leq C\|\nabla \vv\|_{H^2}\|Q\|_{H^1}^2.
\end{align*}
This completes the proof of the lemma.
\end{proof}

\subsection{The proofs of (\ref{CH,Q1}) and (\ref{CH,Q2}) in Lemma \ref{A key lemma}}

For notational simplicity, we denote $\fE_0$ by $\fE$, and introduce the following symbols:
\[\begin{aligned}
    &\dot{\overline{\partial_iQ_R}}=(\partial_t+\tilde{\vv}\cdot\nabla)\partial_i Q_R,\quad\dot{\overline{\Delta Q_R}}=(\partial_t+\tilde{\vv}\cdot\nabla)\Delta Q_R.
\end{aligned}\]

We now prove (\ref{CH,Q1}) and (\ref{CH,Q2}) of Lemma \ref{A key lemma}, respectively.

\begin{proof}[Proof of {\em(\ref{CH,Q1})}]
Using $\PP=\dot{Q}_R+\vv_R\cdot\nabla Q^\ve$, we have
\[ \begin{aligned}
    -\ve\langle\partial_i\CH_{\nn}^{\ve}(Q_R),\partial_i\PP\rangle&=-\ve\langle\partial_i\CH_{\nn}^{\ve}(Q_R),\partial_i(\vv_R\cdot\nabla Q^\ve)\rangle- \ve\langle\partial_i\CH_{\nn}^{\ve}(Q_R),\partial_i\dot{Q}_R\rangle\\
    &\eqdefa W_1+S_1.
\end{aligned}\]
For the term $W_1$, it can be calculated as
\begin{align*}
   W_1
    =&-\ve\langle\partial_i\vv_R\cdot\nabla Q^\ve,\CH_\nn(\partial_iQ_R)-[\CH_\nn,\partial_i]Q_R+\ve\CL(\partial_iQ_R)\rangle\\
   &-\ve\langle\vv_R\cdot\nabla \partial_iQ^\ve,\CH_\nn(\partial_iQ_R)-[\CH_\nn,\partial_i]Q_R+\ve\CL(\partial_iQ_R)\rangle\\
   =&-\ve\langle\partial_i\vv_R\cdot\nabla Q^\ve,\CH_\nn(\partial_iQ_R)\rangle-\ve\langle\partial_i\vv_R\cdot\nabla Q^\ve,[\partial_i,\CH_\nn]Q_R\rangle\\
   & -\ve\langle\partial_i\vv_R\cdot\nabla Q^\ve,\ve\CL(\partial_iQ_R)\rangle-\ve\langle\vv_R\cdot\nabla \partial_iQ^\ve,\CH_\nn(\partial_iQ_R)\rangle\\
   &-\ve\langle\vv_R\cdot\nabla \partial_iQ^\ve,[\partial_i,\CH_\nn]Q_R\rangle -\ve\langle\vv_R\cdot\nabla \partial_iQ^\ve,\ve\CL(\partial_iQ_R)\rangle\\
   \eqdefa&\sum^6_{k=1}\mathfrak{y}_k.
\end{align*}
Armed with $\partial_i\vv_R\cdot \nabla Q_0\in\text{Ker}\CH_\nn$ and lemma \ref{v*Q,L(Q)}, we derive that
\begin{align*}
    \mathfrak{y}_1&=-\ve^2\langle\partial_i\vv_R\cdot\nabla \widehat{Q}^\ve,\CH_\nn(\partial_iQ_R)\rangle-\ve^4\langle\partial_i\vv_R\cdot\nabla Q_R,\CH_\nn(\partial_iQ_R)\rangle\\
    &\leq C\ve^2\|\partial_i\vv_R\|_{L^2}\|\partial_i Q_R\|_{L^2}\big(1+\ve^2\|\nabla Q_R\|_{H^2}\big)\leq C(\ve{\fE}+\ve{\fE}^{\frac{3}{2}}),
\end{align*}
and
\begin{align*}
    \mathfrak{y}_6&=-\ve^2\langle\vv_R\cdot\nabla \partial_i\widehat{Q}^\ve,\CL(\partial_iQ_R)\rangle-\ve^5\langle\vv_R\cdot\nabla \partial_iQ_R,\CL(\partial_iQ_R)\rangle\\
    &\leq C\big(\ve^2\|\vv_R\|_{L^2}\|\partial_iQ_R\|_{H^2}+\ve^5\|\nabla\vv\|_{H^2}\|\partial_iQ_R\|_{H^1}^2\big)\leq C(\fE+\ve\fE\fF^{\frac{1}{2}}),
\end{align*}
where $\widehat{Q}^{\ve}=Q_1+\ve Q_2+\ve^2Q_3$. Using $Q^\ve=\widetilde{Q}+\ve^3Q_R$ and (\ref{fg Hk}), the terms $\mathfrak{y}_i(i=2,3,4,5)$ are estimated as, respectively,
\begin{align*}
\mathfrak{y}_2\leq&~C\|\ve\partial_i\vv_R\|_{L^2}\|\nabla Q^\ve\|_{H^2}\|Q_R\|_{L^2}
    \leq C({\fE}+\ve{\fE}^{\frac{3}{2}}),\\
    \mathfrak{y}_3 \leq&~ C\|\partial_i\vv_R\|_{L^2}\|\nabla Q^\ve\|_{H^2}\|\ve^2\nabla^3Q_R\|_{L^2}
    \leq C({\fE}^{\frac{1}{2}}+\ve{\fE})\fF^{\frac{1}{2}},\\
    \mathfrak{y}_4 \leq&~C\ve\|\vv_R\|_{L^2}\|\partial _iQ_R\|_{L^2}+C\ve\|\ve^2\vv_R\|_{H^2}\|\ve\nabla\partial_iQ_R\|_{L^2}\|\partial_i Q_R\|_{L^2}
    \\\leq&~ C(\ve{\fE}+\ve{\fE}^{\frac{3}{2}}),\\
    \mathfrak{y}_5\leq&~ C\ve\|\vv_R\|_{L^2}\|Q_R\|_{L^2}+C\ve^4\|\vv_R\|_{H^2}\|\nabla\partial_iQ_R\|_{L^2}\|Q_R\|_{L^2}
    \\\leq&~ C(\ve{\fE}+\ve{\fE}^{\frac{3}{2}}).
    \end{align*}
Summarizing the estimates yields
\begin{equation}\label{vQ2}
    W_1\leq C({\fE}+\ve{\fE}^{\frac{3}{2}}+{\fE}^{\frac{1}{2}}\fF^{\frac{1}{2}}+\ve{\fE}\fF^{\frac{1}{2}}).
\end{equation}
Using Lemma \ref{v*Q,L(Q)} again, we infer that
\begin{align}
  S_1=&-\ve\Big\langle\CH_{\nn}^\ve(\partial_iQ_R),\partial_i\Dot{Q}_R\Big\rangle+\ve\Big\langle[\CH_\nn,\partial_i]Q_R,\partial_i\Dot{Q}_R\Big\rangle\nonumber\\
=&-\ve\Big\langle\CH_{\nn}^\ve(\partial_iQ_R),\Dot{\overline{\partial_i{Q}_R}}+\partial_i\Tilde{\vv}\cdot\nabla Q_R\Big\rangle+\ve\Big\langle[\CH_\nn,\partial_i]Q_R,\partial_i\Dot{Q}_R\Big\rangle\nonumber\\
=&-\ve\Big\langle\CH_\nn^\ve(\partial_i Q_R),\partial_i\tilde{\vv}\cdot\nabla Q_R\Big\rangle+\ve\Big\langle[\CH_\nn,\partial_i]Q_R,\dot{Q}_R\Big\rangle\nonumber\\&-\ve^2\langle\CL(\partial_iQ_R),(\tilde{\vv}\cdot\nabla)\partial_iQ_R\rangle+\frac{\ve}{2}\langle[\partial_t+\Tilde{\vv}\cdot\nabla,\CH_\nn]\partial_iQ_R,\partial_iQ_R\rangle\nonumber\\&-\frac{\ve}{2}\frac{\ud}{\ud t}\Big\langle\CH_{\nn}^\ve(\partial_iQ_R),\partial_i{Q}_R\Big\rangle\nonumber\\
\leq&~C\Big(\|\nabla Q_R\|_{L^2}\|\ve\CH_\nn^\ve(\partial_i Q_R)\|_{L^2}+\|Q_R\|_{L^2}\|\ve\partial_i\dot{Q}_R\|_{L^2}\nonumber\\&+\ve^2\|\partial_iQ_R\|^2_{H^1}+\ve\|\partial_iQ_R\|_{L^2}^2\Big)-\frac{\ve}{2}\frac{\ud}{\ud t}\Big\langle\CH_{\nn}^\ve(\partial_iQ_R),\partial_i{Q}_R\Big\rangle\nonumber\\
\leq&-\frac{\ve}{2}\frac{\ud}{\ud t}\Big\langle\CH_{\nn}^\ve(\partial_iQ_R),\partial_i{Q}_R\Big\rangle+C({\fE}+\ve{\fE}^{\frac{3}{2}}).
\label{Q,CH2}
\end{align}
Combining (\ref{vQ2}) with (\ref{Q,CH2}), we deduce that
\begin{align*}
    -\ve\langle\partial_i\CH_{\nn}^{\ve}(Q_R),\partial_i\PP\rangle
    \leq-\frac{\ve}{2}\frac{\ud}{\ud t}\Big\langle\CH_{\nn}^\ve(\partial_iQ_R),\partial_i{Q}_R\Big\rangle+C({\fE}+\ve{\fE}^{\frac{3}{2}}+{\fE}^{\frac{1}{2}}\fF^{\frac{1}{2}}+\ve{\fE}\fF^{\frac{1}{2}}).
\end{align*}
Thus, we obtain the desired estimate.
\end{proof}

\begin{proof}[Proof of {\em(\ref{CH,Q2})}]
Similarly, we have
\[\begin{aligned}
    -\ve^3\langle\Delta\CH_{\nn}^{\ve}(Q_R),\Delta\PP\rangle&=-\ve^3\langle\Delta\CH_{\nn}^{\ve}(Q_R),\Delta(\vv_R\cdot\nabla Q^\ve)\rangle-\ve^3\langle\Delta\CH_{\nn}^{\ve}(Q_R),\Delta\dot{Q}_R\rangle\\&\eqdefa W_2+S_2.
\end{aligned}\]
By a straightforward computation, the term $W_2$ can be expressed by
\begin{align*}
    W_2=&-\ve^3\langle\Delta(\vv_R\cdot\nabla Q^\ve),\Delta\CH_\nn(Q_R)+\ve\CL(\Delta Q_R)\rangle\\
    = & -\ve^3\langle\vv_R\cdot\nabla\Delta Q^\ve,\ve\CL(\Delta Q_R)\rangle-\ve^3\langle\Delta\vv_R\cdot\nabla Q^\ve,\CH_\nn(\Delta Q_R)\rangle\\
    &-\ve^3\langle\Delta\vv_R\cdot\nabla Q^\ve,[\Delta,\CH_\nn] Q_R\rangle-\ve^3\langle\Delta\vv_R\cdot\nabla Q^\ve,\ve\CL(\Delta
Q_R)\rangle\\
&-\ve^3\langle\partial_i\vv_R\cdot\nabla\partial_i Q^\ve,\Delta\CH_\nn (Q_R)\rangle
-\ve^3\langle\vv_R\cdot\nabla\Delta Q^\ve,\Delta\CH_\nn(Q_R)\rangle\\
&-\ve^3\langle\partial_i\vv_R\cdot\nabla\partial_i Q^\ve,\ve\CL(\Delta Q_R)\rangle\\
=&\sum^7_{k=1}\mathfrak{m}_k.
\end{align*}
From integration by parts and Lemma \ref{v*Q,L(Q)}, we infer that
\begin{align*}
    \mathfrak{m}_1&=-\ve^4\langle\vv_R\cdot\nabla\Delta \widetilde{Q},\ve\CL(\Delta Q_R)\rangle-\ve^7\langle\vv_R\cdot\nabla\Delta Q_R,\CL(\Delta Q_R)\rangle\\
    &=\ve^4\langle\partial_i(\vv_R\cdot\nabla\Delta \widetilde{Q}),\ve\CL(\partial_i Q_R)\rangle-\ve^7\langle\vv_R\cdot\nabla\Delta Q_R,\CL(\Delta Q_R)\rangle\\&\leq C(\ve^4\|\vv_R\|_{H^1}\|\nabla^3 Q_R\|_{L^2}+\ve^7\|\nabla\vv_R\|_{H^2}\|\Delta Q_R\|_{H^1}^2)\\&\leq C(\ve{\fE}+\ve{\fE}\fF^{\frac{1}{2}}).
\end{align*}
According to $\Delta \vv_R \cdot\nabla Q_0\in {\rm Ker}\CH$, it follows that
\begin{align*}
    \mathfrak{m}_2&=-\ve^4\langle\Delta\vv_R\cdot\nabla \widehat{Q}^\ve,\CH_\nn(\Delta Q_R)\rangle-\ve^6\langle\Delta\vv_R\cdot\nabla Q_R,\CH_\nn(\Delta Q_R)\rangle\\&\leq C\ve^4\|\Delta \vv_R\|_{L^2}\|\Delta Q_R\|_{L^2}(1+\ve^2\|\nabla Q_R\|_{H^2})\leq C(\ve{\fE}+\ve{\fE}^{\frac{3}{2}}).
\end{align*}
As for the estimates of the terms $\mathfrak{m}_i(i=3,4,\dots,6)$, it is easy to obtain
\begin{align*}
\mathfrak{m}_3\leq&~C\|\ve^2\Delta \vv_R\|_{L^2}\|\nabla Q^\ve\|_{H^2}\|\ve Q_R\|_{H^1}
     \leq C(\ve{\fE}+\ve^2{\fE}^{\frac{3}{2}}),\\
\mathfrak{m}_4\leq&~C\|\ve^2\Delta\vv_R\|_{H^1}\|\nabla Q^\ve\|_{H^2}\|\ve^2\nabla\Delta Q_R\|_{L^2}
\leq C\fF^{\frac{1}{2}}{\fE}^{\frac{1}{2}}(1+\ve{\fE}^{\frac{1}{2}}),\\
\mathfrak{m}_5\leq&~C\|\ve^2\partial_i\vv_R\|_{H^2}\|\nabla\partial_i Q^\ve\|_{L^2}\|\ve Q_R\|_{H^2}
     \leq C({\fE}^{\frac{1}{2}}+\ve^2{\fE})\fF^{\frac{1}{2}},\\
    \mathfrak{m}_6\leq&~C \|\ve^2\vv_R\|_{H^2}\|\nabla\Delta Q^\ve\|_{L^2}\|\ve Q_R\|_{H^2}\leq C({\fE}+\ve{\fE}^{\frac{3}{2}}).
\end{align*}
The term $\mathfrak{m}_7$ can be estimated as
\begin{align*}
    \mathfrak{m}_7&=\ve^4\langle\partial_j(\partial_i \vv_R\cdot\nabla\partial_i \widetilde{Q}),\CL(\partial_jQ_R)\rangle+\ve^7\langle\partial_j(\partial_i \vv_R\cdot\nabla\partial_i Q_R),\CL(\partial_j Q_R)\rangle\\
    &\leq C\ve^4\|\nabla^2\partial_jQ_R\|_{L^2}\big(\|\partial_i\vv_R\|_{H^1}+\ve^3\|\partial_i\vv_R\|_{H^2}\|\nabla\partial_i Q_R\|_{H^1}\big)\\&\leq C(\fE+\ve\fE\fF^{\frac{1}{2}}).
\end{align*}
Consequently, combining the above estimates we obtain
\begin{equation}\label{vQ3}
    W_2\leq C({\fE}+\ve{\fE}^{\frac{3}{2}}+{\fE}^{\frac{1}{2}}\fF^{\frac{1}{2}}+\ve{\fE}\fF^{\frac{1}{2}}).
\end{equation}
Using Lemma \ref{norm:QR,vR} and integration by parts, it holds that
\begin{align}\label{Q,CH3}
 S_2=&-\ve^3\Big\langle\CH_{\nn}^\ve(\Delta Q_R),\Delta\Dot{Q}_R\Big\rangle+\ve^3\Big\langle[\CH_\nn,\Delta]Q_R,\Delta\Dot{Q}_R\Big\rangle\nonumber\\
=&-\ve^3\Big\langle\CH_{\nn}^\ve(\Delta Q_R),\Dot{\overline{\Delta {Q}_R}}+[\Delta,\Tilde{\vv}\cdot\nabla ]Q_R\Big\rangle+\ve^3\Big\langle[\CH_\nn,\Delta]Q_R,\Delta\Dot{Q}_R\Big\rangle\nonumber\\
=&~\ve^3\Big\langle[\CH_\nn,\Delta]Q_R,\Delta \dot{Q}_R\Big\rangle-\ve^3\Big\langle\CH_\nn^\ve(\Delta Q_R),[\Delta,\tilde{\vv}\cdot\nabla]Q_R\Big\rangle\nonumber\\
&-\ve^4\langle\CL(\Delta Q_R),(\tilde{\vv}\cdot\nabla)\Delta Q_R\rangle+\frac{\ve^3}{2}\langle[\partial_t+\Tilde{\vv}\cdot\nabla,\CH_\nn]\Delta Q_R,\Delta Q_R\rangle\nonumber\\
&-\frac{\ve^3}{2}\frac{\ud}{\ud t}\Big\langle\CH_{\nn}^\ve(\Delta Q_R),\Delta{Q}_R\Big\rangle\nonumber\\
=&~\ve^3\Big\langle[\CH_\nn,\Delta]Q_R,\Delta \dot{Q}_R\Big\rangle-\ve^3\Big\langle\CH_\nn(\Delta Q_R),[\Delta,\tilde{\vv}\cdot\nabla]Q_R\Big\rangle\nonumber\\
&+\ve^4\Big\langle\CL(\partial_i Q_R),\partial_i[\Delta,\tilde{\vv}\cdot\nabla]Q_R\Big\rangle-\ve^4\langle\CL(\Delta Q_R),(\tilde{\vv}\cdot\nabla)\Delta Q_R\rangle\nonumber\\&+\frac{\ve^3}{2}\langle[\partial_t+\Tilde{\vv}\cdot\nabla,\CH_\nn]\Delta Q_R,\Delta Q_R\rangle-\frac{\ve^3}{2}\frac{\ud}{\ud t}\Big\langle\CH_{\nn}^\ve(\Delta Q_R),\Delta{Q}_R\Big\rangle\nonumber\\
\leq &~C\Big(\ve^3\|Q_R\|_{H^1}\|\Delta \dot{Q}_R\|_{L^2}+\ve^3\|\CH_\nn(\Delta Q_R)\|_{L^2}\|\nabla Q_R\|_{H^1}\nonumber\\&+\ve^4\|\CL(\partial_iQ_R)\|_{L^2}\|\partial_i[\Delta,\tilde{\vv}\cdot\nabla]Q_R\|_{L^2}+\ve^4\|\Delta Q_R\|_{H^1}^2+\ve^3\|\Delta Q_R\|_{L^2}^2\Big)\nonumber\\&-\frac{\ve^3}{2}\frac{\ud}{\ud t}\Big\langle\CH_{\nn}^\ve(\Delta Q_R),\Delta{Q}_R\Big\rangle\nonumber\\
\leq&-\frac{\ve^3}{2}\frac{\ud}{\ud t}\Big\langle\CH_{\nn}^\ve(\Delta Q_R),\Delta{Q}_R\Big\rangle+C({\fE}+\ve^2{\fE}^{\frac{3}{2}}).
\end{align}
Therefore, we derive from (\ref{vQ3}) and (\ref{Q,CH3}) that
\begin{align*}
     -\ve^3\langle\Delta \CH_{\nn}^{\ve}(Q_R),\Delta\PP\rangle
     \leq-\frac{\ve^3}{2}\frac{\ud}{\ud t}\Big\langle\CH_{\nn}^\ve(\Delta Q_R),\Delta Q_R\Big\rangle+C\big({\fE}+\ve{\fE}^{\frac{3}{2}}+{\fE}^{\frac{1}{2}}\fF^{\frac{1}{2}}+\ve{\fE}\fF^{\frac{1}{2}}\big).
\end{align*}
Then the desired estimate follows.
\end{proof}

\bigskip
\noindent{\bf Acknowledgments.}
 Sirui Li is supported by the NSF of China under Grant No.12061019. Wei Wang is partially supported by NSF of China under Grant No.12271476 and 11931010.


\begin{thebibliography}{33}
\bibitem{ADL1}H. Abels, G. Dolzmann, Y. Liu, Well-posedness of a fully-coupled Navier--Stokes/$Q$-tensor system with inhomogeneous boundary data, SIAM J. Math. Anal. 46 (2014) 3050--3077.


\bibitem{Ball} J. M. Ball, Mathematics and liquid crystals, Mol. Cryst. Liq. Cryst. 647 (2017) 1--27.

\bibitem{BE} A. N. Beris, B. J. Edwards, Thermodynamics of Flowing Systems with Internal Microstructure, Oxford Engrg. Sci. Ser. 36, Oxford University Press, New York, 1994.

\bibitem{CW} Y. Cai, W. Wang, Global well-posedness for the three dimensional simplified inertial Ericksen--Leslie system near equilibrium, J. Funct. Anal. 279 (2020) 108521.


\bibitem{DZ} F. De Anna, A. Zarnescu, Global well-posedness and twist-wave solutions for the inertial Qian--Sheng model of liquid crystals, J. Differ. Equ. 264 (2018) 1080--1118.



\bibitem{DE} M. Doi, S. F. Edwards, The Theory of Polymer Dynamics, vol. 73, Oxford University Press, Oxford, UK, 1986.

\bibitem{DHW} H. Du, X. Hu, C. Wang, Suitable weak solutions for the co-rotational Beris--Edwards system in dimension three,
Arch. Ration. Mech. Anal. 238 (2020) 749--803.

\bibitem{DeG} P. G. de Gennes, The Physics of Liquid Crystals, Clarendon Press, Oxford, UK, 1974.

\bibitem{EZ} W. E, P. Zhang, A molecular kinetic theory of inhomogeneous liquid crystal flow and the small Deborah number limit, Methods Appl. Anal. 13 (2006) 181--198.



\bibitem{Ericksen1}J. L. Ericksen, Conservation laws for liquid crystals, Trans. Soc. Rheol. 5 (1961) 22--34.



\bibitem{FRSZ} E. Feireisl, E. Rocca, G. Schimperna, A. Zarnescu, On a hyperbolic system arising in liquid crystals modeling, J. Hyper. Differ. Equ.  15 (2018)  5--35.


\bibitem{HLWZZ} J. Han, Y. Luo, W. Wang, P. Zhang, Z. Zhang, From microscopic theory to macroscopic theory: Systematic study on modeling for liquid crystals, Arch. Ration. Mech. Anal. 215 (2015) 741--809.

\bibitem{HNPS} {M. Hieber, M. Nesensohn, J. Pr\"{u}ss, K. Schade,} Dynamics of nematic liquid crystal flows: The quasilinear approach. {Ann. Inst. H. Poincar\'{e} Anal. Non Lin\'{e}aire}. {33} (2016) 397--408.


\bibitem{HX}{ M.-C. Hong,  Z.-P. Xin,} Global existence of solutions of the liquid crystal flow for the Oseen--Frank
model in $\mathbb{R}^2$. {Adv. Math. }{231}
 (2012) 1364--1400.

\bibitem{HD}J. Huang, S. Ding, Global well-posedness for the dynamical $Q$-tensor model of liquid crystals, Sci. China Math. 58 (2015)  1349--1366.

\bibitem{HJLZ1} J. Huang, N. Jiang, Y. Luo, L. Zhao, Small data global regularity for 3-D Ericksen--Leslie’s
hyperbolic liquid crystal model without kinematic transport. SIAM J. Math. Anal. 53 (2021) 530--573.




\bibitem{HLW}J. Huang, F.-H. Lin, C. Wang, Regularity and existence of global solutions to the Ericksen--Leslie system in $\mathbb{R}^2$, Commun. Math. Phys. 331 (2014) 805--850.

\bibitem{JL1} N. Jiang, Y. Luo, The zero inertia limit from hyperbolic to parabolic Ericksen--Leslie system of liquid crystal flow,  J. Funct. Anal. 282 (2022) 109280.

\bibitem{JL2} N. Jiang, Y. Luo, On well-posedness of Ericksen--Leslie’s hyperbolic incompressible liquid crystal model, SIAM J. Math. Anal.  51 (2019) 403--434.


\bibitem{JNT1} N. Jiang, Y. Luo, S. Tang, Zero inertia density limit for the hyperbolic system of Ericksen--Leslie’s liquid crystal flow with a given velocity,  Nonlinear Anal. Real World Appl.  45 (2019) 590--608.


\bibitem{KD}N. Kuzuu, M. Doi, Constitutive equation for nematic liquid crystals under weak velocity gradient derived from a molecular kinetic equation, J. Phys. Soc. Japan 52 (1983) 3486--3494.

\bibitem{Leslie1} F. M. Leslie, Some constitutive equations for liquid crystals,  Arch. Ration. Mech. Anal.  28 (1968) 265-283



\bibitem{LW} S.-R. Li, W. Wang, Rigorous justification of the uniaxial limit from the Qian--Sheng inertial Q-tensor theory to the Ericksen--Lesile theory, SIAM J. Math. Anal. 52 (5) (2020) 4421--4468.

\bibitem{LWZ} S.-R. Li, W. Wang, P. Zhang, Local well-posedness and small Deborah limit of a molecule-based $Q$-tensor system,  Discrete Contin. Dyn. Syst. Ser. B 20 (8) (2015)  2611--2655.


\bibitem{LX1} S.-R. Li, J. Xu, Frame hydrodynamics of biaxial nematics from molecular-theory-based tensor models, SIAM J. Appl. Math, 83 (2023) 1467--1495.

\bibitem{LX2} S.-R. Li, J. Xu, Rigorous biaxial limit of a molecular-theory-based two-tensor hydrodynamics, J. Differ. Equ.  366 (2023) 862--911.


\bibitem{Lin2} {F.-H. Lin, J. Lin, C. Wang,} Liquid crystal flows in two dimensions, {Arch. Ration. Mech. Anal}. {197}
 (2010) 297--336.

\bibitem{LW4} F.-H. Lin, C. Wang, Recent developments of analysis for hydrodynamic flow of nematic liquid crystals, Philos. Trans. A. 372 (2014)  20130361.

\bibitem{LW2}F.-H. Lin, C. Wang, Global existence of weak solutions of the nematic liquid crystal flow in dimension three, Comm. Pure Appl. Math. 69 (2016) 1532--1571.


\bibitem{LYW1} Y. Liu, W. Wang, On the initial boundary value problem of a Navier--Stokes/$Q$-tensor model for liquid crystals, Discrete Contin. Dyn. Syst. Ser. B 23 (2018) 3879--3899.

\bibitem{LYW2} Y. Liu, W. Wang, The small Deborah number limit of the Doi--Onsager equation without hydrodynamics, J. Funct. Anal. 275 (2018) 2740--2793.


\bibitem{LM} Y. Luo, Y. Ma, Zero inertia limit of incompressible Qian--Sheng model, Analysis and Applications 20 (2) (2022) 221--284.

\bibitem{Maj} A. Majumdar, Equilibrium order parameters of nematic liquid crystals in the Landau--de Gennes theory,  European J. Appl. Math. 21 (2010) 181--203.

\bibitem{MN} N. J. Mottram, C. Newton, Introduction to $Q$-tensor theory, arXiv:1409.3542, 2014.


\bibitem{PZ1} M. Paicu, A. Zarnescu, Global existence and regularity for the full coupled Navier--Stokes and $Q$-tensor system, SIAM J. Math. Anal. 43 (2011)  2009--2049.

\bibitem{PZ2}M. Paicu, A. Zarnescu, Energy dissipation and regularity for a coupled Navier--Stokes and $Q$-tensor system, Arch. Ration. Mech. Anal. 203 (2012) 45--67.

\bibitem{QS} T. Qian, P. Sheng, Generalized hydrodynamic equations for nematic liquid crystals, Phys. Rev. E.  58 (1998) 7475--7485.



\bibitem{WW} M. Wang, W. Wang, Global existence of weak solution for the 2-D Ericksen--Leslie system, Calc. Var. Partial Differ. Equ. 51 (2014)  915--962.



\bibitem{WZZ1}W. Wang,  P. Zhang,  Z. Zhang, Well-posedness of the Ericksen--Leslie system,  Arch. Ration.
Mech. Anal.  210 (2013) 837--855.

\bibitem{WZZ2}W. Wang, P. Zhang,  Z. Zhang, The small Deborah number limit of the Doi--Onsager equation to the Ericksen--Leslie equation, Comm. Pure Appl. Math.  68 (2015) 1326--1398.

\bibitem{WZZ3}W. Wang, P. Zhang,  Z. Zhang, Rigorous derivation from Landau--de Gennes theory to Ericksen--Leslie theory, SIAM J. Math. Anal. 47 (2015) 127--158.

\bibitem{WZZ4}W. Wang, L. Zhang,  P. Zhang, Modelling and computation of liquid crystals, Acta Numerica. 30 (2021) 765--851.

\bibitem{Wilk} M. Wilkinson, Strictly physical global weak solutions of a Navier--Stokes $Q$-tensor system with singular potential, Arch. Ration. Mech. Anal. 218 (2015) 487--526.

\bibitem{WXL} H. Wu, X. Xu, C. Liu, On the general Ericksen--Leslie system: Parodi's relation, well-posedness and stability, Arch. Ration. Mech. Anal. 208 (2013) 59--107.


\bibitem{XinZ} Z.-P. Xin, X. Zhang, From the Landau--de Gennes theory to the Ericksen--Leslie theory in dimension two, arXiv:2105.10652, 2022.


\end{thebibliography}
\end{document}